\documentclass[11pt, a4paper,leqno]{amsart}
\usepackage{amsmath,amsthm,amscd,amssymb,amsfonts,amsbsy,mathrsfs}
\usepackage{latexsym}
\usepackage{exscale}
\usepackage[dvipsnames]{xcolor}
\usepackage{tikz}
\usepackage[colorlinks,linkcolor=blue,citecolor=red,backref=page,hypertexnames=false]{hyperref}

\calclayout
\allowdisplaybreaks

\theoremstyle{plain}
\newtheorem{theorem}[equation]{Theorem}
\newtheorem{lemma}[equation]{Lemma}
\newtheorem{corollary}[equation]{Corollary}
\newtheorem{proposition}[equation]{Proposition}

\theoremstyle{definition}
\newtheorem{definition}[equation]{Definition}

\newtheorem{remark}[equation]{Remark}

\theoremstyle{remark}

\numberwithin{equation}{section}
\numberwithin{figure}{section}

\newcommand{\e}{{\varepsilon}}

\newcommand{\N}{{\mathbb{N}}}
\newcommand{\pom}{{\partial \Omega}}
\newcommand{\Q}{{\mathbf{Q}}}
\newcommand{\T}{{\mathcal{T}}}
\newcommand{\R}{{\mathbb{R}}}
\newcommand{\Rn}{{\R^n}}

\newcommand{\Wloc}{W^{1,2}_{\mathrm{loc}}}
\def\loc{{\operatorname{loc}}}

\newcommand{\mbf}[1]{\mathbf{#1}}

\newcommand{\X}{{\mbf{X}}}
\newcommand{\x}{{\mbf{x}}}
\newcommand{\Y}{{\mbf{Y}}}
\newcommand{\y}{{\mbf{y}}}
\newcommand{\Z}{{\mbf{Z}}}

\newcommand{\abs}[1]{\left\lvert #1 \right\rvert}
\newcommand{\norm}[1]{\left\lVert #1 \right\rVert}

\newcommand{\mres}{\mathbin{\vrule height 1.6ex depth 0pt width
		0.13ex\vrule height 0.13ex depth 0pt width 1.3ex}} 
\usepackage{textcomp}
\renewcommand{\emptyset}{\mbox{\textup{\O}}}
\DeclareMathOperator{\supp}{supp}
\DeclareMathOperator{\diam}{diam}
\DeclareMathOperator{\dist}{dist}

\def\div{\mathop{\operatorname{div}}\nolimits}

\makeatletter
\def\l@subsection{\@tocline{2}{0pt}{2.5pc}{5pc}{}}
\makeatother 

\newcommand{\Bk}{\color{black}}

\begin{document}
\allowdisplaybreaks

\title[Rough parabolic Dirichlet problem with continuous and Hölder data]{The parabolic Dirichlet problem with continuous and Hölder boundary data, and rough coefficients}



\author{Pablo Hidalgo-Palencia}
\address{Pablo Hidalgo-Palencia, Universitat de Barcelona, Departament de Matemàtiques i Informàtica, Gran Via de les Corts Catalanes 585, 08007 Barcelona, Spain}
\email{pablo.hidalgo@ub.edu}

\author{Cody Hutcheson}
\address{Cody Hutcheson\\
         Department of Mathematics\\
         University of Alabama\\
         Tuscaloosa, AL 35487, USA}
\email{cmhutcheson@crimson.ua.edu}

\author{Joseph Kasel}
\address{Joseph Kasel\\
         Department of Mathematics\\
         University of Alabama\\
         Tuscaloosa, AL 35487, USA}
\email{jekasel@crimson.ua.edu}

\thanks{The authors are indebted to Simon Bortz for his support and guidance along the whole project. The authors are also grateful to Steve Hofmann for some enlightening conversations and for suggestions regarding Theorem~\ref{thm:ExistenceOfParabolicMeasureOnBoundedDomains}, and also to Nicola Garofalo for pointing us to his paper with Fabes and Lanconelli 
for Remark~\ref{rem:symmetry}. P.H.-P. has been supported by the grants CEX2019-000904-S-20-3 and PCI2024-155066-2, funded by MCIN/AEI/10.13039/501100011033 and MICIU/AEI/10.13039/501100011033/UE (co-funded by the European Union), respectively, and acknowledges financial support from MCIN/AEI/10.13039/ 501100011033 grants CEX2019-000904-S and PID2019-107914GB-I00.}

\begin{abstract}
	 
	We provide very mild sufficient conditions for space-time domains (non-necessarily cylindrical) which ensure that the continuous Dirichlet problem and the Hölder Dirichlet problem are well-posed, for any parabolic operator in divergence form with merely bounded coefficients. Concretely, we show that the parabolic measure exists, even for unbounded domains, hence solving an open problem posed by Genschaw and Hofmann (2020). 

    This problem has inherent difficulties because of its parabolic nature, as the behavior of solutions near the boundary may depend strongly on the values of the coefficients of the operator. One of our sufficient conditions, the \textit{time-backwards capacity density condition}, is a quantitative version of the parabolic Wiener's criterion, and hence is adapted to the operator under consideration. The other condition, the \textit{time-backwards Hausdorff content condition}, is (albeit slightly stronger) purely geometrical and independent of the operator, hence much easier to check in practice.
    
\end{abstract}

\maketitle
\tableofcontents


\section{Introduction}
The heat equation has been a central topic of study for mathematicians since it was introduced over 200 years ago by Fourier. In its most general form, we may consider heat diffusion inside domains that may vary in time (that is, non-cylindrical), as in 
\[
\partial_t u(X, t) - \Delta_X u(X, t) = 0 \qquad \text{ for } (X, t) \in \Omega \subseteq \Rn \times \R.
\] 
Along with this equation, we will consider Dirichlet boundary conditions $u = f$ over $\partial \Omega$ (or at least over the portion of $\partial \Omega$ which is relevant for the diffusion). 

A fundamental question in this setting is to understand how $u$ behaves when approaching the boundary, and concretely the manner in which it attains the boundary values $f$. Intense efforts have been made recently to understand what happens when $f \in L^p$ (see e.g. \cite{DPP, AEN, GH, DLP, DN, BHMN, BFHH}), but there were even more fundamental questions that were unanswered in the case when $f$ is continuous.

Indeed, when $f$ is continuous, one would expect $u = f$ to hold pointwise over $\pom$. However, it turns out that this is not a trivial matter. Even in the elliptic setting (i.e. considering the steady state $-\Delta u = 0$, Laplace's equation), the work of Wiener \cite{Wiener} showed that $u$ does not always attain the boundary values $f$ in a continuous way (i.e. $u$ may not belong to $C(\overline{\Omega})$). In fact, he gave a full characterization of the class domains where this happens, using capacities, a key object from potential theory.

In the case of the heat equation, the study of the continuous Dirichlet problem (that is, obtaining solutions $u$ which are continuous all the way up to the boundary, given a continuous boundary datum $f$) is more recent. It took more than 50 years since Wiener's result for his criterion to be generalized to the heat equation, which was done by Evans and Gariepy in \cite{EG} (after some preliminary investigations and partial results in \cite{Levi, Gevrey, P, Pini, K, Landis, EK, L}). Soon afterwards, Garofalo and Lanconelli extended the criterion to parabolic operators in divergence form with smooth coefficients in \cite{GL}, and later to $C^{1, \text{Dini}}$ coefficients in \cite{FGL} jointly with Fabes (related works include \cite{Nov, GZ, BM}). In sharp contrast with the elliptic setting, their criterion depends on the operator taken into consideration. We will elaborate on this soon.

Our goal in this paper is to understand the continuous Dirichlet problem for much more general diffusions. Concretely, we will consider the parabolic equation 
\begin{equation} \label{eq:general_eq}
\partial_t u(X, t) - \div_X (A(X, t) \nabla_X u(X, t)) = 0, 
\qquad 
\text{ for } (X, t) \in \Omega \subseteq \Rn \times \R,
\end{equation}
where $A$ will be assumed to satisfy the minimal ellipticity and boundedness conditions, but not any smoothness assumption. Indeed, considering such general diffusions is beneficial not only for the sake of representing real-life processes that are inhomogeneous and anisotropic, but also to study non-linear equations, where the coefficients possibly depend on the solution itself, so no smoothness can be taken for granted. 

The connection between the continuous Dirichlet problem and the caloric/parabolic measure (the parabolic analogue of the harmonic/elliptic measure) is very classical. Indeed, these measures are a really powerful tool, see e.g. \cite{Dah, CFMS, HL, KPT, GMT, AHMMT, HMMTZ, BHMN, FTV} among many others. However, the existence of the parabolic measure in very general settings (i.e. for general $L$ and $\Omega$) is not easy to establish. For instance, Wiener criteria like those in \cite{EG, GL, FGL} imply that the parabolic measure exists for operators with (somehow) smooth diffusion, in domains satisfying some potential theoretic assumptions. For more general diffusions, Genschaw and Hofmann raised a question (see \cite[p. 1532]{GH}) concerning the existence of the parabolic measure, even under more beneficial geometric conditions. In fact, the authors could only continue their analysis by (artificially) imposing the existence of parabolic measure.

In the first main result of this paper, we solve an open problem posed by Genschaw and Hofmann in \cite{GH} (and therefore extend the works cited right above) by providing very mild sufficient conditions on $\Omega$ that ensure that, for any given operator as in \eqref{eq:general_eq}, the associated parabolic measure exists. Notably, this also means that the continuous Dirichlet problem is solvable on $\Omega$. In fact, our answer is considerably more general than the authors asked in \cite{GH}, because our conditions will be notably weaker (see Remark~\ref{rmk:GH}).

Furthermore, by using the parabolic measure, we will obtain a clear understanding of a very natural quantitative version of the continuous Dirichlet problem: the Hölder Dirichlet problem, that is, the boundary value problem where one imposes Hölder continuous boundary data and one looks for solutions in the same Hölder space. In particular, the Hölder exponent must not get smaller. The quantitative nature of this new problem requires a finer analysis of the boundary behavior of solutions to our equations.

In the second main result of this paper, we establish (under background hypotheses on $\Omega$ that are similar to the ones before) the well-posedness of the Hölder Dirichlet problem for any parabolic operator as in \eqref{eq:general_eq}. That is, we show existence and uniqueness of solutions, and also quantitative bounds for the solution in terms of the boundary values. For that, we extend to the parabolic setting the (elliptic) method introduced by Cao, Martell, Prisuelos-Arribas, Zhao, and the first named author, in \cite{CHMPZ} (see also \cite{BMP} for a recent extension to Besov spaces). 

Both of our main results hold for parabolic operators in divergence form with merely bounded diffusion, and under suitable hypotheses on $\Omega \subseteq \Rn \times \R$. Our first main assumption on $\Omega$, the so-called \textit{time backwards capacity density condition} (TBCDC for short), comes from potential theory (in fact, it is somehow a quantitative form of Wiener's criterion). For the case of the heat equation, Mourgoglou and Puliatti showed in their recent and very complete work \cite{MP} that the TBCDC allows for very fine analysis of solutions of the heat equation around the boundary. We are able to extend this to general operators, making strong use of the Aronson's bounds for fundamental solutions. The relevance of the TBCDC is that it should be (at least) close to optimal for our results, as suggested by the elliptic counterparts developed in \cite{Aikawa, CHMPZ}.

However, the TBCDC has a clear downside: it depends on the coefficients of the operator. This comes as no surprise taking into account that Wiener criteria in the parabolic setting, as mentioned before, necessarily depend on the operator, as shown in the early work of Petrovsky \cite{P}, and later refined by Lanconelli in \cite{L2} (see a more detailed discussion in \cite{GL}). This is in fact one (perhaps shocking) relevant difference between elliptic and parabolic equations (for elliptic equations, Wiener's criterion is independent of the operator, as shown by Littman, Stampacchia and Weinberger in their fundamental paper \cite{LSW}). Ultimately, the reason is the shape of the fundamental solutions: although Aronson's bounds say that every fundamental solution looks somehow like a Gaussian, these Gaussians do not have the same parameters, and the bounds are not sharp enough to distinguish very fine properties as in the continuous Dirichlet problem. 

To overcome the issue that the definition of the TBCDC depends on the operator, we introduce a purely geometrical condition, the so-called \textit{time backwards Hausdorff content condition} (TBHCC for short), which is also sufficient for our two main results.
Although (slightly) stronger than the TBCDC, it has the advantage that it is purely geometrical and does not depend on the operator, so it is definitely much easier to check in practice.


\subsection{Main results}

Before presenting the two main results of this paper, let us start by stating the key technical tool that will pave the way for us. We show that the TBCDC assumption is sufficient for the parabolic measure to be non-degenerate near the boundary (as in \eqref{maineqn:BourgainEstimate}). This estimate showcases the usefulness of the parabolic measure, because it easily implies \eqref{maineqn:HolderEstimate}, a powerful decay estimate for general solutions at the boundary.

\begin{theorem}[Bourgain's estimate\footnote{The estimate \eqref{maineqn:BourgainEstimate} is frequently known as Bourgain's estimate because it originally proved to be useful in \cite{Bourgain}. A very complete and understandable extension to the heat equation can be found in \cite{BG}.} and Hölder continuity up to the boundary]
\label{mainthm:Bourgain}
    Let $L=\partial_t-\div A\nabla$ be a parabolic operator with merely bounded coefficients (see Definition~\ref{def:operator}), and $\Omega\subseteq\R^{n+1}$ be an open set that satisfies the time backwards capacity density condition for $L$ (TBCDC, see Definition~\ref{def:TBCDC}). Suppose, in addition, that the parabolic measure $\omega_L$ exists for $L$ on $\Omega$ (see Definition~\ref{def:parabolic_measure}). Then, there exist $\eta,\gamma>0$ such that, for any $(x_0, t_0) \in \partial_e \Omega \setminus \{\infty\}$ and $r > 0$,
    \begin{equation}
    	\label{maineqn:BourgainEstimate}
    	\omega_L^{X,t}(\Q_r(x_0,t_0))\geq\eta, 
    	\qquad 
    	\text{ for all } (X,t)\in \Q_{\gamma r}(x_0,t_0)\cap\Omega.
    \end{equation} 
    Moreover, there exists $\alpha_H \in (0, 1)$ and $C > 0$ such that: if $u \geq 0$ is a weak solution to $Lu = 0$ in $\Q_{2r}(x_0, t_0) \cap \Omega$ for some $(x_0,t_0)\in\partial_e\Omega\setminus\{\infty\}$ and $r>0$, and $u$ vanishes continuously over $\Q_{2r}(x_0,t_0) \cap \partial_e \Omega$, then 
    \begin{equation}
    \label{maineqn:HolderEstimate}
        u(X,t)
        \leq 
        C \left(\frac{\dist((X,t), \partial_e \Omega)}{r}\right)^{\alpha_H} \sup_{\Q_{2r}(x_0,t_0)\cap\Omega} u, 
        \qquad 
        \text{for all $(X,t)\in \Q_r(x_0,t_0)\cap\Omega$.}
    \end{equation}
    The values of $\eta, \gamma, \alpha_H$ and $C$ only depend on $n, \lambda$ (ellipticity) and the TBCDC constants. 
\end{theorem}

For all the relevant definitions, like $\partial_e \Omega$ or cubes like $\Q_{4r}(x_0, t_0)$, we refer to Section~\ref{sec:preliminaries}. 

\begin{remark}
	Note that in Theorem~\ref{maineqn:BourgainEstimate} (and later in Theorem~\ref{mainthm:ExistenceOfHolderSolutions}) we are making the \textit{a priori} assumption that the parabolic measure for $L$ in $\Omega$ exists. Although the reader could be tempted to think that this means that we can only apply these theorems to \textit{nicely behaved} operators, like those with somehow smooth coefficients (see the parabolic Wiener criteria in Section~\ref{sec:wiener}, and concretely Remark~\ref{rmk:ExistenceOfParabolicMeasureForCinfty}), we will shortly state Theorem~\ref{mainthm:ExistenceOfParabolicMeasure}, where we show that the parabolic measure exists for general operators (check also Corollary~\ref{corol:general_L}).
\end{remark}

\begin{remark}
	The sufficiency of the TBCDC for Theorem~\ref{mainthm:Bourgain}, at least for heat-like equations, was already noticed and very carefully developed by Mourgoglou and Puliatti in \cite{MP}. In Theorem~\ref{mainthm:Bourgain}, we just extend that to general equations with rough diffusion, including the proofs in order to try to make the paper approachable and understandable.
\end{remark}

With Theorem~\ref{mainthm:Bourgain} in hand, we know that solutions vanish in a Hölder fashion close to the regions where the associated boundary value vanishes. Moreover, solutions are naturally Hölder continuous in the interior (as shown by the ground-breaking work of Nash in \cite{N}, see Lemma~\ref{lem:interiorHolder}). Combining these, we are able to extend the elliptic program of \cite{CHMPZ} to parabolic equations to show well-posedness of the Hölder Dirichlet problem.

\begin{theorem}[Well-posedness of the Hölder-Dirichlet problem]
	\label{mainthm:ExistenceOfHolderSolutions}
	Let $L=\partial_t-\div A\nabla$ be a parabolic operator with merely bounded coefficients (see Definition~\ref{def:operator}), and $\Omega\subseteq\R^{n+1}$ be an open set that satisfies the TBCDC for $L$ (see Definition~\ref{def:TBCDC}). Suppose, in addition, that the parabolic measure $\omega_L$ exists for $L$ on $\Omega$ (see Definition~\ref{def:parabolic_measure}).
	Then, there exists some $\alpha\in(0,1)$\footnote{Actually, one can take $\alpha := \min\{\alpha_N,\alpha_H\}$, with $\alpha_N$ from Lemma \ref{lem:interiorHolder} and $\alpha_H$ from Theorem~\ref{mainthm:Bourgain}.} such that for all $\beta\in(0,\alpha)$, the $\dot{C}^\beta$-Dirichlet problem (see Definition~\ref{def:holder_problem})
	\begin{equation*}
		\begin{cases}
			u \in \dot{C}^\beta(\Omega \cup \partial_n \Omega) \cap W^{1, 2}_\loc(\Omega), \\ 
			Lu = 0 \text{ in the weak sense in $\Omega$}, \\ 
			u|_{\partial_e \Omega} = f \in \dot{C}^\beta(\partial_e \Omega),
		\end{cases}
	\end{equation*}
	is well-posed. More specifically, for any $f \in \dot{C}^\beta(\partial_e \Omega)$ there is a unique solution given by
	\[
	u(X, t) 
	:= 
	\int_{\partial_e\Omega}f(y, s)\,d\omega_L^{X, t}(y, s),
	\qquad (X, t) \in \Omega,
	\]
	which belongs to $\dot{C}^\beta(\Omega \cup \partial_n \Omega) \cap W^{1, 2}_\loc(\Omega)$,
	and also satisfies
	\begin{equation}
		\label{3.5.2}
		\lVert f\rVert_{\dot{C}^\beta(\partial_e\Omega)}\leq\lVert u\rVert_{\dot{C}^\beta(\Omega)}\leq C\lVert f\rVert_{\dot{C}^\beta(\partial_e\Omega)}.
	\end{equation}
	The values of $\alpha$ and $C > 0$ only depend on $n, \lambda$ and the TBCDC constants.
\end{theorem}

In this result, uniqueness of solutions is not an issue because we are assuming that the parabolic measure exists, and it is in fact always a probability because, if $\Omega$ is unbounded, we consider the point at infinity as part of the boundary (namely $\infty \in \partial_e \Omega$), which is the usual convention in the literature (see e.g. \cite{W2, GH, MP}). This basically means that we can impose boundary values at infinity. We will explain why this is reasonable in Section~\ref{sec:classification_boundary}. In contrast, we also study in Section~\ref{sec:probability} what happens when we abandon this convention and do not impose boundary values at infinity (more similar to what is done in \cite{CHMPZ}): there will be cases where the parabolic measure is no longer a probability. 

Our second main result asserts that the TBCDC is enough for the parabolic measure to exist, even for general operators. 

\begin{theorem}[Existence of parabolic measure]
	\label{mainthm:ExistenceOfParabolicMeasure}
	Let $L=\partial_t-\div A\nabla$ be a parabolic operator with merely bounded coefficients (see Definition~\ref{def:operator}),
	and $\Omega\subseteq\R^{n+1}$ be an open set that satisfies the TBCDC for $\partial_t-M_2\Delta$ (see Definition~\ref{def:TBCDC}), the operator associated to $L$ in Lemma~\ref{lem:capLtoM} (so $M_2$ only depends on $n$ and ellipticity $\lambda$). 
	Then, the continuous Dirichlet problem for $L$ is solvable in $\Omega$ (see Definition~\ref{def:continuous_problem}), and the parabolic measure for $L$ in $\Omega$ exists (see Definition~\ref{def:parabolic_measure}). 
	Namely, for each $(X,t)\in\Omega$, there exists a (unique) positive Radon measure $\omega_L^{X,t}$ supported on $\partial_e\Omega$ so that for all $f\in C(\partial_e\Omega)$, the function given by
	\[
	u(X,t)=\int_{\partial_e\Omega}f\,d\omega_L^{X,t}, 
    \qquad (X, t) \in \Omega,
	\]
	solves the continuous Dirichlet problem with boundary data $f$. Moreover, for each $(X,t)\in\Omega$, $\omega_L^{X,t}$ is a probability measure, that is, $\omega_L^{X,t}(\partial_e\Omega)=1$.
\end{theorem}

The proof of this result uses strongly Theorems~\ref{mainthm:Bourgain} and \ref{mainthm:ExistenceOfHolderSolutions}. The strategy will be to approximate $L$ by smoother operators, for which we know that the associated parabolic measure exists (see Section~\ref{sec:wiener}). This enables the use Theorem~\ref{mainthm:ExistenceOfHolderSolutions} to obtain some uniform Hölder behavior for the approximate solutions, which translates into nice compactness properties that ultimately let us find a solution to the continuous Dirichlet problem for $L$. 

\begin{remark}
	\label{rmk:WeakenM2Assumption}
	Regarding the optimality of the assumption of the TBCDC for $\partial_t-M_2\Delta$:
	\begin{itemize}
        \item As already hinted, the assumption must be operator-dependent because of Petrovsky's classical example \cite{P} (see also the discussion of \cite[Theorem 1.7]{GL}) 
        \begin{align*}
        \Omega_P
        :=&
        \big\{ 
        (X, t) \in \R \times \R : -1/e < t < 0 \text{  and  } X^2 < -4t \log |\log |t||
        \big\}
        \\ =&
        \big\{ 
        (X, t) \in \R \times \R : -1/e < t < 0 \text{  and  } e^{-\frac{X^2}{4t}} < -\log (-t)
        \big\}.
        \end{align*}
        For this domain, the regularity of the origin (i.e. whether solutions attain the boundary value at the origin in a continuous manner) depends on the operator: for instance, it is regular for $\partial_t - \Delta$, but not for $\partial_t - \frac12 \Delta$. The fact that the origin is a terminal point for $\Omega_P$ is actually not important: one could easily modify $\Omega_P$ (as shown in Figure~\ref{fig:petrowsky2}) so that the origin is still an irregular point, but it belongs to the portion of the boundary where it is meaningful to impose boundary values (precisely, to the parabolic boundary, see Section~\ref{sec:classification_boundary}). Note that this can be easily generalized to higher dimensions (as already done in \cite[Section 4]{P}).

        \begin{figure}
            \centering
            \scalebox{.45}{
            \begin{tikzpicture}[scale=1]
                \filldraw[ultra thick, fill=Cyan!40, fill opacity=0.3] (0.503, 7.255) .. controls (0.503, 7.255) and (0.525, 7.83) .. (0.55, 8.305) .. controls (0.576, 8.781) and (0.605, 9.157) .. (0.652, 9.533) .. controls (0.699, 9.909) and (0.764, 10.285) .. (0.829, 10.657) .. controls (0.894, 11.03) and (0.959, 11.399) .. (1.057, 11.767) .. controls (1.154, 12.136) and (1.284, 12.505) .. (1.436, 12.841) .. controls (1.588, 13.178) and (1.762, 13.481) .. (1.921, 13.698) .. controls (2.08, 13.915) and (2.224, 14.045) .. (2.362, 14.143) .. controls (2.499, 14.241) and (2.629, 14.306) .. (2.789, 14.288) .. controls (2.948, 14.27) and (3.136, 14.168) .. (3.273, 14.056) .. controls (3.41, 13.944) and (3.497, 13.821) .. (3.617, 13.626) .. controls (3.736, 13.431) and (3.888, 13.163) .. (4.018, 12.838) .. controls (4.148, 12.512) and (4.257, 12.129) .. (4.336, 11.746) .. controls (4.416, 11.362) and (4.466, 10.979) .. (4.51, 10.607) .. controls (4.553, 10.234) and (4.589, 9.873) .. (4.618, 9.41) .. controls (4.647, 8.947) and (4.669, 8.383) .. (4.694, 7.27) .. controls (4.641, 6.151) and (4.616, 5.587) .. (4.588, 5.11) .. controls (4.559, 4.633) and (4.528, 4.243) .. (4.479, 3.867) .. controls (4.429, 3.491) and (4.362, 3.13) .. (4.273, 2.762) .. controls (4.184, 2.394) and (4.073, 2.018) .. (3.944, 1.69) .. controls (3.815, 1.362) and (3.669, 1.081) .. (3.548, 0.893) .. controls (3.428, 0.706) and (3.333, 0.613) .. (3.192, 0.496) .. controls (3.05, 0.379) and (2.861, 0.24) .. (2.635, 0.28) .. controls (2.409, 0.32) and (2.147, 0.541) .. (1.942, 0.797) .. controls (1.738, 1.053) and (1.591, 1.345) .. (1.444, 1.687) .. controls (1.296, 2.029) and (1.146, 2.42) .. (1.037, 2.793) .. controls (0.927, 3.166) and (0.857, 3.521) .. (0.801, 3.894) .. controls (0.746, 4.267) and (0.705, 4.658) .. (0.664, 5.135) .. controls (0.624, 5.613) and (0.584, 6.177) .. (0.504, 7.262);
                
                \filldraw[ultra thick, fill=Cyan!40, fill opacity=0.3] (11.289, 7.17) .. controls (11.334, 8.289) and (11.356, 8.853) .. (11.381, 9.33) .. controls (11.407, 9.807) and (11.436, 10.198) .. (11.483, 10.574) .. controls (11.53, 10.95) and (11.595, 11.311) .. (11.682, 11.68) .. controls (11.768, 12.049) and (11.877, 12.425) .. (12.003, 12.754) .. controls (12.13, 13.083) and (12.275, 13.365) .. (12.394, 13.553) .. controls (12.513, 13.741) and (12.607, 13.835) .. (12.748, 13.953) .. controls (12.889, 14.07) and (13.077, 14.211) .. (13.303, 14.173) .. controls (13.529, 14.134) and (13.793, 13.915) .. (13.999, 13.66) .. controls (14.205, 13.405) and (14.354, 13.114) .. (14.504, 12.773) .. controls (14.654, 12.432) and (14.806, 12.042) .. (14.918, 11.669) .. controls (15.03, 11.297) and (15.102, 10.943) .. (15.16, 10.57) .. controls (15.218, 10.198) and (15.261, 9.807) .. (15.305, 9.33) .. controls (15.348, 8.853) and (15.391, 8.289) .. (15.478, 7.204) .. controls (15.522, 14.79) and (15.522, 14.79) .. (15.522, 14.79) .. controls (15.522, 14.79) and (15.522, 14.79) .. (14.111, 14.79) .. controls (12.7, 14.79) and (9.878, 14.79) .. (8.467, 14.79) .. controls (7.056, 14.79) and (7.056, 14.79) .. (7.056, 14.79) .. controls (7.056, 14.79) and (7.056, 14.79) .. (7.099, 7.163) .. controls (7.179, 6.077) and (7.219, 5.513) .. (7.26, 5.035) .. controls (7.3, 4.558) and (7.341, 4.167) .. (7.396, 3.794) .. controls (7.452, 3.421) and (7.522, 3.067) .. (7.632, 2.693) .. controls (7.741, 2.32) and (7.891, 1.929) .. (8.039, 1.587) .. controls (8.187, 1.245) and (8.333, 0.953) .. (8.537, 0.697) .. controls (8.742, 0.441) and (9.005, 0.22) .. (9.23, 0.18) .. controls (9.456, 0.14) and (9.645, 0.28) .. (9.787, 0.396) .. controls (9.928, 0.513) and (10.023, 0.606) .. (10.144, 0.794) .. controls (10.264, 0.981) and (10.411, 1.262) .. (10.539, 1.59) .. controls (10.668, 1.918) and (10.779, 2.294) .. (10.868, 2.662) .. controls (10.957, 3.03) and (11.024, 3.392) .. (11.074, 3.767) .. controls (11.123, 4.143) and (11.154, 4.533) .. (11.183, 5.01) .. controls (11.211, 5.488) and (11.236, 6.052) .. (11.289, 7.17) -- cycle;
                
                \node[circle, fill, inner sep=3pt] at (4.688, 7.15) {};
                \node[circle, fill, inner sep=3pt] at (11.29, 7.202) {};
                
                \draw[ultra thick, ->] (0, 0.397) -- (1.129, 0.397);
                \draw[ultra thick, ->] (0, 0.397) -- (0, 1.525);
                \node[anchor=center, font=\huge] at (0.515, 1.545) {$X$};
                \node[anchor=center, font=\huge] at (1.152, 0.765) {$t$};
                
                \draw[ultra thick, ->] (11.305, 0.397) -- (12.434, 0.397);
                \draw[ultra thick, ->] (11.305, 0.397) -- (11.305, 1.525);
                \node[anchor=center, font=\huge] at (11.926, 1.51) {$X$};
                \node[anchor=center, font=\huge] at (12.493, 0.801) {$t$};
                
                \node[anchor=center, font=\huge] at (3.669, 7.17) {$(0, 0)$};
                \node[anchor=center, font=\huge] at (10.371, 7.17) {$(0, 0)$};
                \node[anchor=center, font=\Huge] at (2.499, 10.369) {$\Omega_P$};
                \node[anchor=center, font=\Huge] at (9.41, 10.511) {$\Omega_P'$};
            \end{tikzpicture}
            }
            \caption{
            Using a mirror reflection of $\Omega_P$ across $t = 0$, one can easily define $\Omega_P'$ as in the picture, and the origin is still a regular point for $\partial_t - \Delta$, and irregular for $\partial_t - \frac12 \Delta$ (since heat only travels towards the future, what affects regularity is the shape of the domain towards the past). Furthermore, it is clear that $(0, 0) \in \mathcal{P}\Omega_P'$ according to the definition in Section~\ref{sec:classification_boundary}.
            }
            \label{fig:petrowsky2}
        \end{figure}
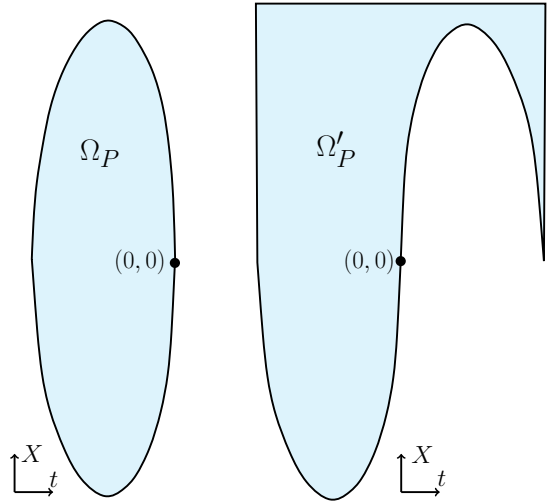
    
		\item As can be seen from the proof, our method does not allow for substituting the assumption of the TBCDC for $\partial_t-M_2\Delta$ in Theorem~\ref{mainthm:ExistenceOfParabolicMeasure} by the TBCDC for $L$ (at least for general rough operators). Indeed, as already commented, our proof runs by an approximation argument (concretely, see Theorem~\ref{thm:ExistenceOfParabolicMeasureOnBoundedDomains}, Part I, Step 1), and since we want to apply Theorem~\ref{mainthm:ExistenceOfHolderSolutions} to these approximating operators, we need the TBCDC to be true also for them. However, boundary points can be regular for one approximating operator and not for others (parabolic Wiener criteria are operator-dependent, see the previous item). Therefore, we make use of a stronger assumption that only takes into account the ellipticity constant of the approximating operators, and not the concrete values: the TBCDC for $\partial_t-M_2\Delta$ implies the TBCDC for all the approximating operators (as we shall see). 
		
		\item 
		Nevertheless, if $L$ has smooth coefficients, the condition that $\Omega$ satisfies the TBCDC for $\partial_t-M_2\Delta$ can be weakened to the condition that $\Omega$ satisfies the TBCDC for $L$. See Section~\ref{sec:wiener}.
	\end{itemize}
\end{remark}

Since Theorem~\ref{mainthm:ExistenceOfParabolicMeasure} ensures the existence of the parabolic measure, we can now remove the extra hypotheses in Theorems~\ref{mainthm:Bourgain} and \ref{mainthm:ExistenceOfHolderSolutions}:

\begin{corollary} \label{corol:general_L}
	Let $L=\partial_t-\div A\nabla$ be a parabolic operator with merely bounded coefficients (see Definition~\ref{def:operator}),
	and $\Omega\subseteq\R^{n+1}$ be an open set that satisfies the TBCDC for $\partial_t-M_2\Delta$ (see Definition~\ref{def:TBCDC}), the operator associated to $L$ in Lemma~\ref{lem:capLtoM}. 
	Then, all the conclusions in Theorems~\ref{mainthm:Bourgain} and \ref{mainthm:ExistenceOfHolderSolutions} hold.
\end{corollary}

So far, we have only used potential theoretical assumptions, in the form of the TBCDC. As already explained above, TBCDC-like conditions are close to being optimal for some of the above results, as \cite{CHMPZ} shows for the elliptic counterpart of Theorem~\ref{mainthm:ExistenceOfHolderSolutions}. However, these assumptions are inevitably operator-dependent. Because of this, we have decided to introduce another condition which is only slightly stronger, but purely geometrical and independent of the operator taken into consideration, so easier to verify in practice.

\begin{proposition}[TBHCC$\implies$TBCDC for all $L$] \label{thm:TBHCCimpliesTBCDCforallL}
	If the open set $\Omega \subseteq \R^{n+1}$ satisfies the time backwards Hausdorff content condition (TBHCC, see Definition~\ref{def:TBHCC}), then it satisfies the TBCDC (see Definition~\ref{def:TBCDC}) for any parabolic operator $L$ with merely bounded coefficients as in Definition~\ref{def:operator}.
\end{proposition}

In light of this relationship, the TBHCC is useful to study the parabolic Dirichlet problem. Concretely, the TBHCC is a sufficient condition for all the results above. 

\begin{corollary}
	Let $L=\partial_t-\div A\nabla$ be a parabolic operator with merely bounded coefficients (see Definition~\ref{def:operator}),
	and $\Omega\subseteq\R^{n+1}$ be an open set that satisfies the TBHCC (see Definition~\ref{def:TBHCC}). 
	Then, all the conclusions in Theorems~\ref{mainthm:Bourgain}, \ref{mainthm:ExistenceOfHolderSolutions} and \ref{mainthm:ExistenceOfParabolicMeasure} hold. Concretely, the parabolic measure for $L$ in $\Omega$ exists.
\end{corollary}

\begin{remark} \label{rmk:GH}
	In particular, our results apply to open sets $\Omega$ whose boundary is time-backwards Ahlfors-David regular (TBADR), as in \cite{GH}. Indeed, if the TBADR condition is satisfied, so is the TBHCC (it essentially corresponds to only looking at thickness over the boundary instead of over all the exterior, and taking $\e = 1$ in Definition~\ref{def:TBHCC}; see a simple proof in Lemma~\ref{lem:TBADR_implies_TBHCC}), and hence also the TBCDC for any operator by Proposition~\ref{thm:TBHCCimpliesTBCDCforallL}. Therefore, our results resolve the open question raised in \cite[p. 1532]{GH}, and actually do so in greater generality, because the TBADR assumption is notably stronger than both of our main assumptions, the TBHCC or the TBCDC.
\end{remark}


\subsection{Outline of the paper}

\begin{itemize}
	\item In Section~\ref{sec:preliminaries}, we explain notation and the necessary definitions to understand the PDE framework in which we will work.
	
	\item In Section~\ref{sec:TBCDC}, we define our key assumptions, the TBCDC and the TBHCC. We give a series of basic results for them, like the proof that the TBHCC is stronger (Proposition~\ref{thm:TBHCCimpliesTBCDCforallL}), and the relationship of the TBCDC with parabolic Wiener criteria. The key will be to use Aronson's bounds for fundamental solutions (see Lemma~\ref{lem:fund_sol}) to relate general operators $L$ with heat-like operators $\partial_t - M\Delta$, which are much easier to understand (this moral will actually be followed throughout the paper).
	
	\item In Section~\ref{sec:bourgain_proof}, we prove Theorem~\ref{mainthm:Bourgain}. First, we show non-degeneracy of parabolic measure with some capacitary estimates, which implies \eqref{maineqn:BourgainEstimate} by the TBCDC (at least over the lateral boundary; at the bottom boundary, one needs other --simple-- arguments). Later, a standard iteration yields the Hölder decay \eqref{maineqn:HolderEstimate} for solutions vanishing over a portion of the boundary.
	
	\item In Section~\ref{sec:holder}, we show Theorem~\ref{mainthm:ExistenceOfHolderSolutions}: the well-posedness of the Hölder Dirichlet problem follows quickly as in \cite{CHMPZ} from Theorem~\ref{mainthm:Bourgain} and Lemma~\ref{lem:interiorHolder}.
	
	\item In Section~\ref{sec:existence_parabolic_measure}, we prove Theorem~\ref{mainthm:ExistenceOfParabolicMeasure}. First, we restrict to bounded domains, and show that one can find solutions to the continuous Dirichlet problem for rough operators by approximating with smoother operators, making use of some compactness granted by Theorem~\ref{mainthm:ExistenceOfHolderSolutions}. Later, we extend the result to unbounded domains, focusing on the role of the point at infinity.
	
	\item In Section~\ref{sec:probability}, we discuss what changes if we do not consider the point at infinity as part of the boundary (when the domain is unbounded): we study the phenomenon that the parabolic measure may not be a probability in such case.  
\end{itemize}


\section{Preliminaries} \label{sec:preliminaries}


\subsection{Notation}

In the sequel, $\Omega$ will be our reference open subset in space-time $\R^{n+1}$. We will use the following notation:
\begin{itemize}
    
    \item We shall orient our coordinate axes so that time runs from left to right.

    \item We use the letters $c$, $C$ to denote harmless positive constants, not necessarily the same at each occurrence, which depend only on dimension and the constants appearing in the hypotheses of the theorems. We shall also write $a\lesssim b$ and $a\approx b$ to mean, respectively, that $a\leq Cb$ and $0<c\leq a/b\leq C$, where the constants $c$ and $C$ are as above. $a \lesssim_\gamma b$ emphasizes that the implicit constant depends on $\gamma$, on top of other parameters which are to be expected, such as dimension and ellipticity.

    \item We denote by $\nabla$ the gradient with respect to spatial variables only.

    \item For the sake of notational brevity, we shall also use boldface capital letters to denote points in space-time $\R^{n+1} = \Rn \times \R$ when there is no need to distinguish their spatial and time coordinates, as in 
    \[\qquad\quad  \mbf{X}=(X,t), \Y = (Y, s) \in \Rn \times \R,
    \quad\mathrm{and}\quad
    \mbf{x}=(x,t), \y = (y, s) \in \pom \subseteq \Rn \times \R.\]
    Similarly, we also denote the space-time origin by $\mbf{0} = (0, 0) \in \Rn \times \R$.

     \item We shall use lower case letters to denote (the spatial component of) points on the boundary $\partial\Omega$, and capital letters for (the spatial component of) generic points in $\R^{n+1}$ (in particular those in $\Omega$).

    \item We denote solid integrals, taken over subsets of space-time $\R^{n+1}$, by $\iint$, whereas we reserve the notation $\int$ for integrals taken over proper subsets of $\R^{n+1}$, like $\pom$, $\Rn$ (space) or $\R$ (time).
   
    \item Given a time $t \in \R$, we denote the restriction to the past/future/present by 
    \[
    \qquad 
    \T_{< t} := \Rn \times (-\infty, t), 
    \qquad 
    \T_{= t} := \Rn \times \{t\}, 
    \qquad 
    \T_{> t} := \Rn \times (t, +\infty).
    \]
    
    \item Given $A\subseteq\R^{n+1}$, we denote initial and terminal times by 
    \[
    \qquad\;\;
    T_{\min}(A)
    := \inf \big\{t \in \R: A\cap \T_{=t} \neq\emptyset\big\}, 
    \;\;
    T_{\max}(A):=\sup\big\{t \in \R: A\cap \T_{=t} \neq\emptyset \big\},
    \]
    which may take infinite values.
    For our distinguished set $\Omega$, we abbreviate
    \[T_{\min} := T_{\min}(\Omega), \qquad T_{\max} := T_{\max}(\Omega).\]
    
    \item (Parabolic norm)
    Given $(X,t)\in\R^{n}\times\R$ in space-time, we say its \textit{parabolic norm} is
    \[\lVert (X, t) \rVert : = \max \big\{ |X|, |t|^{1/2} \big\}, \]
    although other equivalent choices like $|X| + |t|^{1/2}$ would lead to the very same conclusions.
    The distance induced by this norm will be called \textit{parabolic distance}. Further, given $A\subseteq\R^{n+1}$, its diameter with respect to the parabolic distance is
    \[\diam_p(A):=\sup_{(\mbf{X},\mbf{Y})\in A\times A}\lVert\mbf{X}-\mbf{Y}\rVert.\]
    
    \item (Space-time cubes) 
    Given $(X,t)\in\R^{n}\times\R$ in space-time and $r>0$, we define the (space-time) \textit{parabolic cube} centered at $(X, t)$ with radius $r$ by
    \[\Q_r(X,t):=\big\{(Y,s)\in\R^{n}\times\R:\lVert(X,t)-(Y,s)\rVert<r\big\}.\]
    Further, we define the \textit{time-backward} and \textit{time-forward} parabolic cubes by
    \[\Q_r^-(X,t):= \Q_r(X, t) \cap \T_{< t}, \qquad 
    \Q_r^+(X,t):= \Q_r(X, t) \cap \T_{> t}.\]

    \item (Purely spatial cubes) 
    We also use the notation of cubes for purely spatial regions around $X \in \Rn$ (note the lack of boldface fonts): 
    \[
    Q_r(X) := \big\{ Y\in\R^n: | X-Y | <r \big\}.
    \]
    
    \item  If $\mbf{X}\in\Omega$, we set
    \[\delta(\mbf{X}):=\dist(\mbf{X},\partial_e\Omega),\]
    the parabolic distance from $\mbf{X}$ to the essential boundary. In fact, \cite[Lemma 1.17]{GH} guarantees the existence of $\mbf{x_0}\in \partial_e\Omega$ such that $\delta(\mbf{X}) = \lVert \mbf{X}- \mbf{x_0}\rVert$.
     
\end{itemize}


\subsection{General classification of boundary points} \label{sec:classification_boundary}

Heat does not travel backwards in time. As a consequence, it only makes sense to prescribe boundary data on parts of $\pom$ which have some influence in the future, i.e., which have some nearby part of $\Omega$ to their right (recall our convention that time runs from left to right). This includes vertical faces that bound the domain on the left (like initial values), and of course non-vertical parts of the boundary (lateral boundary conditions). See Figure~\ref{fig:example domain 1}. 

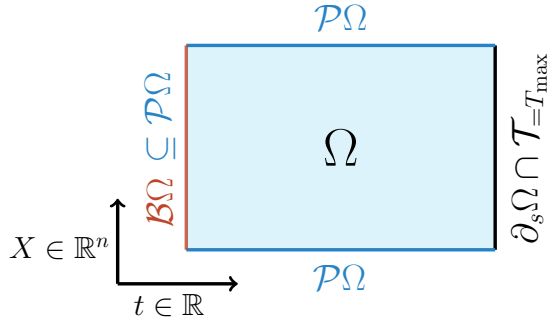
\begin{figure}[h]
	\centering
	\scalebox{.8}{
	\begin{tikzpicture}[scale=1]
		\draw[ultra thick, ->] (1.787, 0.532) -- (3.763, 0.532);
		\draw[ultra thick, ->] (1.787, 0.532) -- (1.787, 1.943);	\node[anchor=center, font=\Large] at (2.634, 0.184) {$t \in \mathbb{R}$};
		\node[anchor=center, font=\Large] at (0.832, 1.118) {$X \in \mathbb{R}^n$};
	
		\fill[Cyan!40, opacity=0.3] (2.916, 1.096) rectangle (7.996, 4.483);
		
		\draw[ultra thick] (7.996, 4.483) -- (7.996, 1.096);
		
		\draw[RoyalBlue!80, ultra thick] (2.916, 4.483) -- (7.996, 4.483);
		\draw[RoyalBlue!80, ultra thick] (2.916, 1.096) -- (7.996, 1.096);
		
		\draw[BrickRed!90, ultra thick, opacity=.8] (2.916, 4.483) -- (2.916, 1.096);
		
		\node[anchor=center, font=\LARGE, text=RoyalBlue!80] at (5.456, 4.939) {$\mathcal{P}\Omega$};
		\node[anchor=center, font=\LARGE, text=RoyalBlue!80] at (5.434, 0.64) {$\mathcal{P}\Omega$};
		\node[rotate=90, anchor=center, font=\LARGE, text=BrickRed!90] at (2.482, 1.878) {$\mathcal{B}\Omega$};
		\node[rotate=90, anchor=center, font=\LARGE] at (8.56, 2.789) {$\partial_s \Omega \cap \mathcal{T}_{=T_{\max}}$};
		\node[rotate=90, anchor=center, font=\LARGE, text=RoyalBlue!80] at (2.482, 3.289) {$\subseteq \mathcal{P}\Omega$};
		\node[anchor=center, font=\Huge] at (5.456, 2.789) {$\Omega$};	
	\end{tikzpicture}
	}
	\caption{
		In a rectangular/cylindrical domain, it only makes sense to prescribe boundary data on the red part of the boundary (initial values) and on the blue parts (lateral/boundary values). It does not make sense to prescribe boundary data on the black part because it has no influence on $\Omega$. Indeed, there is no part of $\Omega$ lying immediately to its right (future).
	}
	\label{fig:example domain 1}
\end{figure}

This suggests the definition of the \textit{parabolic boundary}: \[\mathcal{P}\Omega:= \left\{(x,t)\in\partial\Omega: \; \Q_r^-(x,t) \cap \Omega^c \neq \emptyset \text{  for every } r>0\right\}.\]
Basically, this excludes vertical terminal faces as the black one in Figure~\ref{fig:example domain 1}. The points in $\mathcal{P}\Omega$ are the points where it is reasonable to prescribe boundary values. Among points in $\mathcal{P}\Omega$, we may distinguish those in the bottom boundary 
\[\mathcal{B}\Omega:=\left\{(x,t)\in\mathcal{P}\Omega: \; \exists \, r>0 \textrm{ such that } \Q_r^+(x,t)\subseteq\Omega\right\}.\]
In Figure~\ref{fig:example domain 1}, these points are the red ones, and are usually thought of as initial values.

However, Figure~\ref{fig:example domain 1} is too simple for us: in this paper, we strive to work with non-cylindrical domains, i.e. domains whose cross-section may vary as time passes. This makes it trickier to understand which boundary points have an influence on the equation. Although under the framework of our hypotheses TBCDC and TBHCC, everything will be much simpler (see Section~\ref{sec:structure_boundary_TBCDC}), let us include a very general discussion following \cite{W2} (see also \cite{GH} or \cite{L3}). 

\begin{figure}[h]
	\centering
	\scalebox{.8}{
	\begin{tikzpicture}[scale=1]
		\draw[ultra thick, ->] (1.787, 0.532) -- (3.763, 0.532);
		\draw[ultra thick, ->] (1.787, 0.532) -- (1.787, 1.943);
		\node[anchor=center, font=\Large] at (2.634, 0.184) {$t \in \mathbb{R}$};
		\node[anchor=center, font=\Large] at (0.832, 1.118) {$X \in \mathbb{R}\textsuperscript{$n$}$};
		
		\fill[Cyan!40, opacity=0.3] 
		(0.096, 5.047)
		.. controls (0.096, 5.047) and (0.096, 5.047) .. (0.331, 5.094)
		.. controls (0.566, 5.141) and (1.036, 5.235) .. (1.724, 5.273)
		.. controls (2.411, 5.31) and (3.315, 5.291) .. (3.832, 5.103)
		.. controls (4.35, 4.915) and (4.48, 4.558) .. (4.593, 4.426)
		.. controls (4.705, 4.295) and (4.799, 4.389) .. (4.846, 4.436)
		.. controls (4.893, 4.483) and (4.893, 4.483) .. (4.893, 4.483)
		.. controls (4.893, 4.483) and (4.893, 4.483) .. (4.893, 4.295)
		.. controls (4.893, 4.107) and (4.893, 3.73) .. (4.893, 3.542)
		.. controls (4.893, 3.354) and (4.893, 3.354) .. (4.893, 3.354)
		.. controls (4.893, 3.354) and (4.893, 3.354) .. (4.94, 3.307)
		.. controls (4.987, 3.26) and (5.082, 3.166) .. (5.176, 3.166)
		.. controls (5.27, 3.166) and (5.364, 3.26) .. (5.458, 3.213)
		.. controls (5.552, 3.166) and (5.646, 2.978) .. (5.787, 2.978)
		.. controls (5.928, 2.978) and (6.116, 3.166) .. (6.21, 3.26)
		.. controls (6.305, 3.354) and (6.305, 3.354) .. (6.305, 3.354)
		.. controls (6.305, 3.354) and (6.305, 3.354) .. (6.305, 3.542)
		.. controls (6.305, 3.73) and (6.305, 4.107) .. (6.305, 4.295)
		.. controls (6.305, 4.483) and (6.305, 4.483) .. (6.305, 4.483)
		.. controls (6.305, 4.483) and (6.305, 4.483) .. (6.54, 4.436)
		.. controls (6.775, 4.389) and (7.245, 4.295) .. (7.527, 3.965)
		.. controls (7.81, 3.636) and (7.904, 3.072) .. (8.092, 2.931)
		.. controls (8.28, 2.789) and (8.562, 3.072) .. (8.703, 3.213)
		.. controls (8.845, 3.354) and (8.845, 3.354) .. (8.845, 3.354)
		.. controls (8.845, 3.354) and (8.845, 3.354) .. (8.845, 2.978)
		.. controls (8.845, 2.601) and (8.845, 1.849) .. (8.845, 1.472)
		.. controls (8.845, 1.096) and (8.845, 1.096) .. (8.845, 1.096)
		.. controls (8.845, 1.096) and (8.845, 1.096) .. (8.656, 1.002)
		.. controls (8.468, 0.908) and (8.092, 0.72) .. (7.81, 0.814)
		.. controls (7.527, 0.908) and (7.339, 1.284) .. (7.104, 1.237)
		.. controls (6.869, 1.19) and (6.587, 0.72) .. (6.399, 0.579)
		.. controls (6.21, 0.438) and (6.116, 0.626) .. (6.069, 0.72)
		.. controls (6.022, 0.814) and (6.022, 0.814) .. (6.022, 0.72)
		.. controls (6.022, 0.626) and (6.022, 0.438) .. (5.881, 0.485)
		.. controls (5.74, 0.532) and (5.458, 0.814) .. (5.082, 0.955)
		.. controls (4.705, 1.096) and (4.235, 1.096) .. (3.953, 1.19)
		.. controls (3.67, 1.284) and (3.576, 1.472) .. (3.529, 1.567)
		.. controls (3.482, 1.661) and (3.482, 1.661) .. (3.482, 1.661)
		.. controls (3.482, 1.661) and (3.482, 1.661) .. (3.482, 1.943)
		.. controls (3.482, 2.225) and (3.482, 2.789) .. (3.482, 3.072)
		.. controls (3.482, 3.354) and (3.482, 3.354) .. (3.482, 3.354)
		.. controls (3.482, 3.354) and (3.482, 3.354) .. (3.341, 3.542)
		.. controls (3.2, 3.73) and (2.918, 4.107) .. (2.683, 4.295)
		.. controls (2.447, 4.483) and (2.259, 4.483) .. (1.977, 4.53)
		.. controls (1.695, 4.577) and (1.319, 4.671) .. (1.036, 4.718)
		.. controls (0.754, 4.765) and (0.566, 4.765) .. (0.425, 4.765)
		.. controls (0.284, 4.765) and (0.19, 4.765) .. (0.143, 4.765)
		.. controls (0.096, 4.765) and (0.096, 4.765) .. (0.096, 4.765)
		.. controls (0.096, 4.765) and (0.096, 4.765) .. (0.096, 4.765)
		.. controls (0.096, 4.765) and (0.096, 4.765) .. (0.096, 4.812)
		.. controls (0.096, 4.859) and (0.096, 4.953) .. (0.096, 5)
		.. controls (0.096, 5.047) and (0.096, 5.047) .. cycle;
		
		\draw[ultra thick] (8.845, 3.354) -- (8.845, 1.096);		
		\draw[ultra thick] (4.893, 4.483) -- (4.893, 3.354);
		
		\draw[BrickRed!90, ultra thick, opacity=.8] (3.482, 3.354) .. controls (3.482, 2.225) and (3.482, 1.661) .. (3.482, 1.661);		
		\draw[BrickRed!90, ultra thick, opacity=.8] (6.305, 4.483) -- (6.305, 3.354);
		
		\draw[Green!80, ultra thick] (5.74, 2.225) -- (5.74, 1.378);
		
		\draw[RoyalBlue!80, ultra thick] (4.92, 2.208) .. controls (5.233, 2.078) and (5.32, 1.926) .. (5.74, 1.894);	
			
		\draw[RoyalBlue!80, ultra thick] (5.74, 1.71) .. controls (5.536, 1.634) and (5.531, 1.558) .. (5.497, 1.504) .. controls (5.464, 1.45) and (5.403, 1.418) .. (5.302, 1.491) .. controls (5.201, 1.564) and (5.06, 1.742) .. (4.909, 1.667);
				
		\draw[RoyalBlue!80, ultra thick] (3.484, 1.657) .. controls (3.658, 1.373) and (3.704, 1.333) .. (3.761, 1.295) .. controls (3.819, 1.258) and (3.888, 1.223) .. (3.951, 1.199) .. controls (4.014, 1.174) and (4.07, 1.16) .. (4.134, 1.147) .. controls (4.198, 1.134) and (4.27, 1.122) .. (4.338, 1.112) .. controls (4.406, 1.101) and (4.469, 1.092) .. (4.537, 1.081) .. controls (4.605, 1.07) and (4.677, 1.058) .. (4.751, 1.043) .. controls (4.825, 1.027) and (4.901, 1.009) .. (4.982, 0.982) .. controls (5.063, 0.956) and (5.15, 0.921) .. (5.229, 0.884) .. controls (5.307, 0.847) and (5.378, 0.806) .. (5.436, 0.77) .. controls (5.495, 0.735) and (5.541, 0.704) .. (5.579, 0.677) .. controls (5.618, 0.65) and (5.65, 0.628) .. (5.691, 0.6) .. controls (5.732, 0.572) and (5.783, 0.539) .. (5.83, 0.52) .. controls (5.878, 0.5) and (5.921, 0.495) .. (5.95, 0.505) .. controls (5.979, 0.514) and (5.994, 0.539) .. (6.004, 0.572) .. controls (6.013, 0.605) and (6.018, 0.646) .. (6.025, 0.679) .. controls (6.031, 0.711) and (6.041, 0.735) .. (6.046, 0.745) .. controls (6.052, 0.755) and (6.053, 0.752) .. (6.063, 0.733) .. controls (6.074, 0.713) and (6.093, 0.677) .. (6.113, 0.646) .. controls (6.133, 0.616) and (6.154, 0.591) .. (6.181, 0.572) .. controls (6.208, 0.553) and (6.241, 0.541) .. (6.267, 0.536) .. controls (6.293, 0.532) and (6.312, 0.536) .. (6.328, 0.541) .. controls (6.344, 0.547) and (6.359, 0.554) .. (6.374, 0.564) .. controls (6.39, 0.574) and (6.407, 0.587) .. (6.426, 0.603) .. controls (6.445, 0.62) and (6.466, 0.639) .. (6.494, 0.669) .. controls (6.522, 0.699) and (6.558, 0.739) .. (6.584, 0.769) .. controls (6.611, 0.799) and (6.628, 0.819) .. (6.656, 0.853) .. controls (6.685, 0.886) and (6.726, 0.933) .. (6.757, 0.968) .. controls (6.788, 1.003) and (6.81, 1.026) .. (6.835, 1.052) .. controls (6.861, 1.078) and (6.891, 1.106) .. (6.924, 1.133) .. controls (6.958, 1.16) and (6.994, 1.185) .. (7.04, 1.202) .. controls (7.085, 1.219) and (7.139, 1.227) .. (7.174, 1.229) .. controls (7.21, 1.23) and (7.227, 1.224) .. (7.251, 1.213) .. controls (7.275, 1.201) and (7.305, 1.182) .. (7.337, 1.159) .. controls (7.369, 1.136) and (7.403, 1.109) .. (7.436, 1.08) .. controls (7.469, 1.052) and (7.502, 1.022) .. (7.529, 0.998) .. controls (7.556, 0.974) and (7.578, 0.956) .. (7.62, 0.927) .. controls (7.662, 0.899) and (7.725, 0.86) .. (7.796, 0.836) .. controls (7.867, 0.811) and (7.946, 0.8) .. (7.996, 0.796) .. controls (8.046, 0.792) and (8.067, 0.794) .. (8.11, 0.803) .. controls (8.152, 0.811) and (8.217, 0.826) .. (8.287, 0.849) .. controls (8.357, 0.871) and (8.433, 0.901) .. (8.5, 0.93) .. controls (8.567, 0.959) and (8.625, 0.986) .. (8.845, 1.115);
		
		\draw[RoyalBlue!80, ultra thick] (6.332, 4.477) .. controls (6.54, 4.436) and (6.607, 4.422) .. (6.665, 4.409) .. controls (6.722, 4.396) and (6.77, 4.384) .. (6.811, 4.373) .. controls (6.852, 4.361) and (6.888, 4.351) .. (6.942, 4.33) .. controls (6.997, 4.31) and (7.071, 4.28) .. (7.131, 4.253) .. controls (7.191, 4.225) and (7.237, 4.2) .. (7.259, 4.187) .. controls (7.282, 4.174) and (7.282, 4.174) .. (7.302, 4.16) .. controls (7.321, 4.146) and (7.36, 4.118) .. (7.4, 4.084) .. controls (7.44, 4.051) and (7.481, 4.011) .. (7.513, 3.976) .. controls (7.546, 3.941) and (7.57, 3.91) .. (7.588, 3.885) .. controls (7.606, 3.86) and (7.618, 3.842) .. (7.633, 3.819) .. controls (7.647, 3.796) and (7.663, 3.768) .. (7.686, 3.723) .. controls (7.71, 3.678) and (7.74, 3.617) .. (7.768, 3.554) .. controls (7.796, 3.492) and (7.823, 3.428) .. (7.845, 3.376) .. controls (7.867, 3.324) and (7.884, 3.284) .. (7.903, 3.241) .. controls (7.921, 3.198) and (7.942, 3.154) .. (7.964, 3.112) .. controls (7.986, 3.07) and (8.009, 3.031) .. (8.041, 2.996) .. controls (8.073, 2.961) and (8.114, 2.93) .. (8.156, 2.915) .. controls (8.197, 2.9) and (8.239, 2.901) .. (8.274, 2.908) .. controls (8.309, 2.914) and (8.337, 2.926) .. (8.372, 2.945) .. controls (8.407, 2.965) and (8.449, 2.992) .. (8.494, 3.026) .. controls (8.539, 3.061) and (8.587, 3.103) .. (8.629, 3.142) .. controls (8.67, 3.18) and (8.705, 3.215) .. (8.845, 3.354);
		
		\draw[RoyalBlue!80, ultra thick] (4.893, 3.365) .. controls (4.962, 3.286) and (5.008, 3.248) .. (5.042, 3.223) .. controls (5.076, 3.198) and (5.097, 3.187) .. (5.117, 3.18) .. controls (5.137, 3.173) and (5.157, 3.169) .. (5.187, 3.173) .. controls (5.217, 3.177) and (5.259, 3.189) .. (5.289, 3.198) .. controls (5.318, 3.207) and (5.336, 3.213) .. (5.362, 3.215) .. controls (5.388, 3.217) and (5.423, 3.215) .. (5.449, 3.208) .. controls (5.475, 3.2) and (5.492, 3.188) .. (5.509, 3.173) .. controls (5.527, 3.157) and (5.546, 3.139) .. (5.564, 3.12) .. controls (5.583, 3.1) and (5.603, 3.081) .. (5.631, 3.058) .. controls (5.66, 3.036) and (5.697, 3.011) .. (5.731, 2.998) .. controls (5.764, 2.985) and (5.793, 2.983) .. (5.823, 2.988) .. controls (5.853, 2.993) and (5.884, 3.005) .. (5.927, 3.031) .. controls (5.971, 3.057) and (6.027, 3.098) .. (6.062, 3.125) .. controls (6.097, 3.151) and (6.11, 3.163) .. (6.135, 3.187) .. controls (6.16, 3.21) and (6.196, 3.245) .. (6.305, 3.355);
		
		\draw[RoyalBlue!80, ultra thick] (0.096, 4.765) .. controls (0.317, 4.765) and (0.377, 4.765) .. (0.426, 4.765) .. controls (0.474, 4.765) and (0.509, 4.765) .. (0.546, 4.764) .. controls (0.583, 4.763) and (0.622, 4.762) .. (0.665, 4.759) .. controls (0.709, 4.756) and (0.757, 4.753) .. (0.836, 4.743) .. controls (0.914, 4.733) and (1.022, 4.717) .. (1.089, 4.707) .. controls (1.156, 4.696) and (1.181, 4.692) .. (1.231, 4.681) .. controls (1.282, 4.671) and (1.358, 4.655) .. (1.436, 4.639) .. controls (1.515, 4.622) and (1.596, 4.605) .. (1.673, 4.589) .. controls (1.749, 4.573) and (1.821, 4.559) .. (1.89, 4.546) .. controls (1.959, 4.534) and (2.025, 4.523) .. (2.075, 4.515) .. controls (2.124, 4.507) and (2.156, 4.502) .. (2.199, 4.494) .. controls (2.242, 4.487) and (2.295, 4.476) .. (2.345, 4.462) .. controls (2.395, 4.448) and (2.442, 4.432) .. (2.493, 4.408) .. controls (2.543, 4.383) and (2.597, 4.352) .. (2.648, 4.314) .. controls (2.699, 4.277) and (2.749, 4.235) .. (2.789, 4.197) .. controls (2.83, 4.159) and (2.862, 4.126) .. (2.909, 4.074) .. controls (2.956, 4.022) and (3.017, 3.951) .. (3.067, 3.892) .. controls (3.116, 3.833) and (3.154, 3.786) .. (3.202, 3.724) .. controls (3.251, 3.661) and (3.31, 3.584) .. (3.482, 3.333);
		
		\draw[RoyalBlue!80, ultra thick] (0.096, 5.042) .. controls (0.368, 5.102) and (0.449, 5.117) .. (0.529, 5.131) .. controls (0.61, 5.145) and (0.689, 5.158) .. (0.768, 5.169) .. controls (0.846, 5.181) and (0.923, 5.19) .. (0.992, 5.198) .. controls (1.062, 5.205) and (1.123, 5.211) .. (1.201, 5.22) .. controls (1.28, 5.229) and (1.376, 5.243) .. (1.481, 5.253) .. controls (1.585, 5.263) and (1.698, 5.27) .. (1.797, 5.275) .. controls (1.896, 5.28) and (1.982, 5.283) .. (2.079, 5.285) .. controls (2.177, 5.287) and (2.288, 5.288) .. (2.413, 5.286) .. controls (2.539, 5.283) and (2.68, 5.279) .. (2.821, 5.27) .. controls (2.963, 5.261) and (3.104, 5.248) .. (3.237, 5.23) .. controls (3.37, 5.212) and (3.494, 5.189) .. (3.591, 5.168) .. controls (3.689, 5.146) and (3.761, 5.125) .. (3.825, 5.101) .. controls (3.889, 5.078) and (3.947, 5.053) .. (3.994, 5.029) .. controls (4.042, 5.006) and (4.081, 4.983) .. (4.124, 4.953) .. controls (4.168, 4.922) and (4.217, 4.884) .. (4.261, 4.844) .. controls (4.305, 4.804) and (4.344, 4.763) .. (4.382, 4.717) .. controls (4.42, 4.671) and (4.457, 4.62) .. (4.488, 4.576) .. controls (4.518, 4.533) and (4.542, 4.498) .. (4.566, 4.467) .. controls (4.589, 4.436) and (4.614, 4.41) .. (4.642, 4.395) .. controls (4.671, 4.38) and (4.704, 4.376) .. (4.741, 4.383) .. controls (4.778, 4.39) and (4.818, 4.409) .. (4.893, 4.48);
		
		\draw[Mulberry!80, thick, ->] (1.243, 4.937) -- (0.334, 4.937);
		
		\node[anchor=center, font=\Large, text=BrickRed!90] at (3.029, 2.489) {$\mathcal{B}\Omega$};
		\node[anchor=center, font=\Large, text=BrickRed!90] at (6.74, 3.852) {$\mathcal{B}\Omega$};
		\node[anchor=center, font=\Large, text=RoyalBlue!80] at (2.459, 3.904) {$\mathcal{P}\Omega$};
		\node[anchor=center, font=\Large, text=RoyalBlue!80] at (4.191, 5.329) {$\mathcal{P}\Omega$};
		\node[anchor=center, font=\Large, text=RoyalBlue!80] at (4.684, 1.905) {$\mathcal{P}\Omega$};
		\node[anchor=center, font=\Large, text=Green!80] at (6.34, 1.807) {$\partial_{ss}\Omega$};
		\node[anchor=center, font=\Large, text=RoyalBlue!80] at (5.799, 2.77) {$\mathcal{P}\Omega$};
		\node[anchor=center, font=\Large] at (4.435, 3.95) {$\partial_s \Omega$};
		\node[anchor=center, font=\Large, text=RoyalBlue!80] at (6.913, 0.53) {$\mathcal{P}\Omega$};
		\node[anchor=center, font=\Large, text=RoyalBlue!80] at (7.9, 4.188) {$\mathcal{P}\Omega$};
		\node[anchor=center, font=\Large] at (9.348, 2.283) {$\partial_s \Omega$};
		\node[anchor=center, font=\Large, text=Mulberry!80] at (2.297, 4.949) {$\infty \in \partial_n \Omega$};
	\end{tikzpicture}
	}
	\caption{A more general, and non-cylindrical, space-time domain: there are different kinds of vertical faces ($\mathcal{B}\Omega$, $\partial_s \Omega$ and $\partial_{ss}\Omega$), and even the point at infinity may belong to the boundary if $\Omega$ is unbounded.}
	\label{fig:example domain 2}
\end{figure}
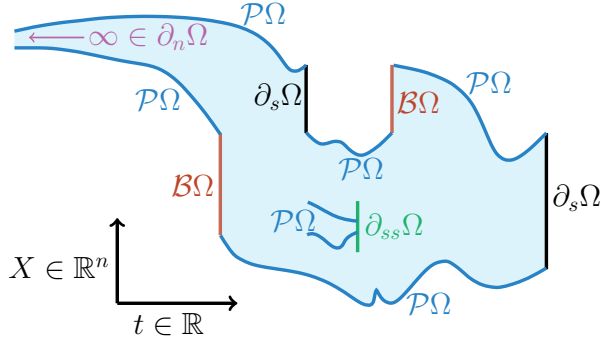

As motivated by Figure~\ref{fig:example domain 1}, it is clear that it does not make sense to prescribe boundary values on vertical faces that have nothing to the right (which will be denoted $\partial_s \Omega$, colored black). The situation with vertical faces like the green one in Figure~\ref{fig:example domain 2} (denoted later by $\partial_{ss}\Omega$) is trickier. Since they have an influence in the near future, we need to be able to impose boundary values there. At the same time, the imposed boundary value can be drastically different from the value suggested by the solution in the near past, which suggests that one should expect to be in trouble when solving the continuous Dirichlet problem there. Indeed, our main assumptions (the TBCDC and the TBHCC) rule out the existence of these vertical walls, so it will be reasonable to tackle the continuous Dirichlet problem. 
For the rest of the boundary points (which will be denoted by $\partial_n \Omega$), it is perfectly reasonable to prescribe boundary values. Concretely, one typically thinks that one can also impose values at the point at infinity, as suggested in Figure~\ref{fig:example domain 2}. 

Let us give rigorous definitions. The \textit{essential boundary} is where it is reasonable to prescribe boundary values (indeed, parabolic measure will be supported here): 
\[
\partial_e\Omega
:=
\partial_n\Omega\cup\partial_{ss}\Omega=\begin{cases}
	\partial\Omega\setminus\partial_s\Omega&\textrm{if }\Omega\textrm{ is bounded}\\
	(\partial\Omega\cup\{\infty\})\setminus\partial_s\Omega&\textrm{if }\Omega\textrm{ is unbounded.}
\end{cases}
\]
Here, in $\partial_n \Omega$, the \textit{normal boundary}, it is crystal clear that boundary values are meaningful:
\[\partial_n\Omega=\begin{cases}
	\mathcal{P}\Omega&\textrm{if }\Omega\textrm{ is bounded}\\
	\mathcal{P}\Omega\cup\{\infty\}&\textrm{if }\Omega\textrm{ is unbounded.}
\end{cases}\]
Its complement $\partial_a\Omega:=\partial\Omega\setminus\partial_n\Omega =\left\{(x,t)\in\partial\Omega:\exists r>0\textrm{ such that }\Q_r^-(x,t)\subseteq\Omega\right\}$, the \textit{abnormal boundary}, takes care of conflictive vertical faces, and can be split as 
\[
\partial_a\Omega=\partial_s\Omega\cup\partial_{ss}\Omega, 
\;
\text{ where }
\;\;\;
\begin{array}{c}
	\partial_s\Omega :=\big\{(x,t)\in\partial_a\Omega:\exists\, r>0\textrm{ such that }\Q_r^+(x,t)\cap\Omega=\emptyset\big\}, \\[.2cm]
	\partial_{ss}\Omega :=\big\{(x,t)\in\partial_a\Omega: \Q_r^+(x,t) \cap \Omega \neq \emptyset \text{  for all } r > 0\big\},
\end{array}
\]
called \textit{singular} and \textit{semi-singular} boundary, respectively. We discussed above that it does not make sense to impose boundary values over $\partial_s \Omega$ (that is why it is not included in $\partial_e \Omega$), whereas it does make sense over $\partial_{ss} \Omega$ (although we should be careful there about values of the solutions coming from the past not matching those coming from the future).

In any case, the situation throughout this paper will be much simpler: the TBCDC and the TBHCC rule out most of the patological behaviors (see Section~\ref{sec:structure_boundary_TBCDC}).

Finally, the \textit{quasi-lateral boundary} $\Sigma$, which removes only initial and terminal faces, is 
\[\Sigma:=\begin{cases}
    \partial\Omega&\textrm{if }T_{\min}=-\infty\textrm{ and }T_{\max}=\infty,\\
    \partial\Omega \setminus \big(\mathcal{B}\Omega \cap \T_{=T_{\min}}\big) & \textrm{if }T_{\min}>-\infty\textrm{ and }T_{\max}=\infty,\\
    \partial\Omega \setminus \big(\partial_s\Omega \cap \T_{=T_{\max}}\big) & \textrm{if }T_{\min}=-\infty\textrm{ and }T_{\max}<\infty,\\
    \partial\Omega\setminus\Big( \big(\mathcal{B}\Omega \cap \T_{=T_{\min}} \big) \cup \big(\partial_s\Omega \cap \T_{=T_{\max}}\big) \Big)
    &\textrm{if }T_{\min}>-\infty\textrm{ and }T_{\max}<\infty.
\end{cases}\]
One can show that $\partial_e\Omega$ and $\Sigma$ are closed sets (see \cite[Lemma 1.17]{GH}).


\subsection{PDE framework} \label{sec:pde}
Across the text, recall that $\Omega\subseteq\R^{n+1}$ is an open set.

\begin{definition}[Parabolic operator with bounded coefficients] \label{def:operator}
    We consider second order parabolic operators in divergence form 
    \begin{equation*} 
        L:=\partial_t-\mathrm{div}(A(X,t)\nabla),
        \qquad 
        (X, t) \in \Omega,
    \end{equation*}
    where $A(X, t)$ is a (not necessarily symmetric) $n\times n$ matrix with real entries for every $(X, t) \in \Omega$, which satisfies, for some $\lambda >0$:
    \begin{equation*} 
        \lambda|\xi|^2
        \leq
        A(X,t)\xi \cdot \xi,\quad\lVert A\rVert_{L^\infty(\R^n)}\leq\lambda^{-1}, 
        \qquad 
        \forall \, \xi\in\R^{n}, \; \text{a.e.} \, (X,t)\in\Omega.
    \end{equation*}
\end{definition}

\begin{definition}[Weak solution]
    We say that $u$ is a weak solution to $Lu = 0$ in $\Omega$ (or that $Lu= 0$ in $\Omega$ in the weak sense) if $u\in W^{1,2}_\loc(\Omega)$ (that is, $u\in L^2_\loc(\Omega)$, and the distributional spatial gradient satisfies $|\nabla u|\in L^2_\loc(\Omega)$) and for every $\psi \in C^\infty_c(\Omega)$, it holds $\iint_\Omega (-u\partial_t \psi + A \nabla u \cdot \nabla \psi) \, dX dt= 0$.    
\end{definition}    

We are now in the position to define two of the main objects of study in the paper.

\begin{definition}[Continuous Dirichlet problem] \label{def:continuous_problem}
    We say that the \textit{continuous Dirichlet problem} is solvable for $L$ (as in Definition~\ref{def:operator}) in $\Omega$, if for every $f \in C(\partial_e \Omega)$, there is a solution to
    \[
        \begin{cases}
            u\in C(\Omega\cup\partial_n\Omega)\cap\Wloc(\Omega),\\
            Lu=0\textrm{ in the weak sense in $\Omega$},\\
            u =f \text{ on } \partial_e\Omega.
        \end{cases}
    \]
    By the last equality $u =f \text{ on } \partial_e\Omega$, we mean that $u$ agrees pointwise with $f$, i.e.,
    \[
        \lim_{(X,t)\to(y,s)}u(X,t)=f(y,s),\qquad\qquad(y,s)\in\partial_n\Omega,
    \]
    and
    \[
        \lim_{(X,t)\to(y,s^+)}u(X,t)=f(y,s),\qquad\qquad(y,s)\in\partial_{ss}\Omega.
    \]
    Indeed, for the semi-singular boundary, we can only expect the imposed boundary values to be have an impact on the future, possibly not being consistent with the values to the near past (revisit the discussion in Section~\ref{sec:classification_boundary} for an intuition about this). 
\end{definition}

\begin{definition}[Parabolic measure] \label{def:parabolic_measure}
    Given an open set $\Omega \subseteq \R^{n+1}$ and an operator $L$ as in Definition~\ref{def:operator}, the \textit{parabolic measure} for $L$ in $\Omega$, denoted by $\{\omega_{L, \Omega}^{X, t}\}_{(X, t) \in \Omega}$ (we will often omit the dependence on $\Omega$), is a family of positive, Radon, probability measures supported on $\partial_e \Omega$ such that, for each $f \in C(\partial_e \Omega)$, the solution to the continuous Dirichlet problem with datum $f$ is given by 
    \begin{equation*} 
        u(X,t)=\int_{\partial_e\Omega}f\,d\omega_L^{X,t}, 
        \qquad (X, t) \in \Omega.
    \end{equation*}
\end{definition}

\begin{remark} \label{rem:probability_bounded}
    When $\Omega$ is bounded, the uniqueness of solutions to the continuous Dirichlet problem is granted by the classical maximum principle. When $\Omega$ is unbounded, we consider the point at infinity as part of the (essential) boundary (see Section~\ref{sec:classification_boundary}). Hence, maximum principles are also available (see the details in Lemma~\ref{lem:Uniqueness}), which again ensures uniqueness for the continuous Dirichlet problem. This justifies that our definition of parabolic measure assumes that they are probability measures (i.e. $\omega_L^{X, t}(\partial_e \Omega) = 1$ for every $(X, t) \in \Omega$). We will analyze some different conventions in Section~\ref{sec:probability}.
\end{remark}

Moreover, we also consider a natural quantitative version of the continuous Dirichlet problem, the Hölder Dirichlet problem, which is the focus of Theorem~\ref{mainthm:ExistenceOfHolderSolutions}.

\begin{definition}[Hölder Dirichlet problem] \label{def:holder_problem}
    Given $\alpha > 0$, we define the \textit{$\dot{C}^\alpha$-Dirichlet problem}  by
    \[
        \begin{cases}
            u\in \dot{C}^\alpha(\Omega\cup\partial_n\Omega)\cap\Wloc(\Omega),\\
            Lu=0\textrm{ in the weak sense in $\Omega$},\\
            u|_{\partial_e\Omega}=f\in \dot{C}^\alpha(\partial_e\Omega),
        \end{cases}
    \]
    where the (homogeneous) $\alpha$-H\"older space on a set $E\subseteq\R^{n+1}$ is defined by 
    \[
        \dot{C}^\alpha(E):=
        \bigg\{u:E\to\R:\lVert u\rVert _{\dot{C}^\alpha(E)}
        :=
        \sup_{\substack{\mbf{X},\mbf{Y}\in E \\ \mbf{X}\neq\mbf{Y}}}\frac{|u(\mbf{X})-u(\mbf{Y})|}{\lVert\mbf{X}-\mbf{Y}\rVert ^\alpha}<\infty
        \bigg\}.
    \]
\end{definition}

One of the main tools to work with parabolic equations are \textit{fundamental solutions}. Indeed, for the heat equation, solutions in the whole space are given by convolution against the Gaussian profile. More generally, the fundamental solution of $\partial_t-M\Delta$ (for $M > 0$) is
\begin{equation} \label{eq:fund_sol_heat}
\Gamma_{\partial_t-M\Delta}(X,t;Y,s)=(4\pi M(t-s))^{-n/2}\exp\left(-\frac{|X-Y|^2}{4M(t-s)}\right)\mathbf{1}_{\{t>s\}},
\end{equation}
which satisfies the simple upper bound (see \cite[(A.8)]{GH} or \cite[(3.2)]{MP})
\begin{equation} \label{eq:aronson_easy_M}
	\Gamma_{\partial_t - M\Delta}(\X; \Y)\leq C_M\lVert \X-\Y \rVert^{-n}, 
	\qquad \X, \Y \in \R^{n+1}.
\end{equation}

For more general operators with merely bounded coefficients, fundamental solutions are also known to exist and satisfy Gaussian-like estimates, as shown by the early works of Aronson \cite{Aronson} and Nash \cite{N}. In full generality, for possibly non-symmetric and merely bounded coefficients, we refer to \cite[Theorem 5.5]{QX}:
\begin{lemma}
\label{lem:fund_sol}
	Let $L$ be a parabolic operator as in Definition~\ref{def:operator}. Then, there exists a fundamental solution for $L$, denoted by $\Gamma_L(\cdot; \cdot)$, satisfying:
	\begin{enumerate}
		\item Given $f \in L^2(\Rn)$ and $t_0 \in \R$, the function $u(X, t):= \int_{\Rn} \Gamma_L(X, t; Y, t_0) f(Y) \, dY$ is the weak solution to $Lu = 0$ in $\T_{> t_0}$, with initial datum $f$ on $\T_{=t_0}$, and decaying at infinity (see \cite{QX} for a more precise statement),
		
		\item (Aronson's bounds) There exists some $N>0$, depending only on $n$ and $\lambda$, such that
		\begin{multline*} 
			\frac{1}{N(t-s)^{n/2}}\exp\left(-\frac{N|X-y|^2}{t-s}\right) \mathbf{1}_{\{t>s\}} 
			\leq
			\Gamma_L(X,t; Y,s)
			\\ \leq
			\frac{N}{(t-s)^{n/2}}\exp\left(-\frac{|X-y|^2}{N(t-s)}\right) \mathbf{1}_{\{t>s\}}, 
            \quad (X, t), (Y, s) \in \R^{n+1}.
		\end{multline*}
	\end{enumerate}
\end{lemma}

The fact that the fundamental solution exists for any parabolic operator, and that it satisfies Aronson's bounds, is what will enable most of our results to hold for operators with very general diffusion. An important (and sometimes surprising fact) about parabolic equations is that the precise value of the constant $N$ in Aronson's bounds is very influential for some properties of the PDE. This is seen, for example, in Wiener-like criteria (see Section~\ref{sec:wiener}), which are inherently operator-dependent, as can be seen in \cite{P} (see the discussion around \cite[Theorem 1.7]{GL}). In this paper, the relevance of $N$ will only arise when comparing the TBCDC for different operators. 
 
We will also make use of the following well-known estimates for parabolic equations. The reader may find them in references like \cite[Section 0]{FGS}, \cite[Section 3]{HL}, or even \cite{Aronson}. The original sources are the works of Moser \cite{Moser} and Nash \cite{N}.

\begin{lemma}[Harnack's Inequality]\label{lem:harnack}
	Let $L$ be a parabolic operator as in Definition~\ref{def:operator}. Let $u \geq 0$ solve $Lu = 0$ in the weak sense in $\Q_{4r}(X, t)$, where $(X, t) \in \R^{n+1}$. Then 
    \[u(Z,\tau)\leq u(Y,s)\exp\left[C\left(\frac{|Y-Z|^2}{|s-\tau|}+1\right)\right], 
    \quad 
    \forall \, (Y,s),(Z,\tau)\in \Q_{2r}(X,t) \text{ with } \tau < s.\]
    The constant $C > 0$ depends only on $n$ and $\lambda$ (ellipticity).
\end{lemma}

\begin{lemma}[Caccioppoli's estimate]\label{lem:caccioppoli}
    Let $L$ be a parabolic operator as in Definition~\ref{def:operator}. Let $u$ solve $Lu = 0$ in the weak sense in $\Q_{4r}(X, t)$, where $(X, t) \in \R^{n+1}$. Then
    \[\iint_{\Q_r(X,t)}|\nabla u(Y, s)|^2\,dYds
    \lesssim 
    r^{-2}\iint_{\Q_{2r}(X,t)}u(Y,s)^2 \,dYds.\]
    The implicit constant depends only on $n$ and $\lambda$. 
\end{lemma}

\begin{lemma}[Interior Hölder continuity]
\label{lem:interiorHolder}
    Let $L$ be a parabolic operator as in Definition~\ref{def:operator}. Let $u$ solve $Lu = 0$ in the weak sense in $\Q_{4r}(X, t)$, where $(X, t) \in \R^{n+1}$. Then
    \[
    | u(Y, s) - u(Z, \tau) |
    \lesssim
    \left( \frac{\lVert (Y, s)-(Z, \tau)\rVert}{r} \right)^{\alpha_N}
    \|u\|_{L^\infty(\Q_{4r}(X, t))}
    , 
    \quad 
    (Y, s),(Z, \tau)\in \Q_{2r}(\mbf{X}).
    \]
    The values of the implicit constant and $\alpha_N \in (0, 1)$ depend only on $n$ and $\lambda$.
\end{lemma}


\section{The TBCDC and TBHCC assumptions} \label{sec:TBCDC}

In this section, we will define and develop some basic properties of our two main assumptions along the paper: the time-backwards capacity density condition (TBCDC) and the time-backwards Hausdorff content condition (TBHCC). Both are fairly mild conditions that allow for good solvability of the Dirichlet problem (as we shall see later on). They are somehow related, but still fairly different in nature: the TBCDC is based on capacities (from potential theory), whereas the TBHCC is based on Hausdorff measures (from geometric measure theory).

A key feature of these in the parabolic world is that the TBCDC will depend on the underlying operator $L$, whereas the TBHCC is purely geometrical and does not depend on the operator under consideration. Throughout the paper, there will be an interplay between both conditions: the TBCDC is usually a sharper assumption (as seen with the Wiener criterion in Subsection~\ref{sec:wiener}), but it depends on $L$. Conversely, the TBHCC is a stronger assumption (i.e. more difficult to satisfy than the TBCDC, as will be seen in Subsection~\ref{sec:implies}), but at the same time purely geometric and independent of the operator, which makes it easier to verify. We have preferred to work with both assumptions, including comparisons between them, to provide the reader a versatile toolkit for possible applications.

Let us start by defining the TBCDC, a potential theoretic notion whose connection to the heat equation has been explored thoroughly by Mourgoglou and Puliatti in \cite{MP}. This notion (with slight modifications) has been known for decades, for instance, in connection with Wiener criteria. 
Following \cite[Section 3]{MP} or \cite[Chapter 5]{W2}, we first define thermal capacity.

\begin{definition}[($L$-Thermal) capacities] \label{def:capacity}
    Let $L$ be a parabolic operator as in Definition~\ref{def:operator}. Given a Borel measure $\mu$ on $\R^{n+1}$, we define its \textit{$L$-heat potential} by 
    \[
        \Gamma_L\mu(\mbf{X}):=\iint_{\R^{n+1}} \Gamma_L(\mbf{X};\mbf{Y})\,d\mu(\mbf{Y}),
    \] 
    where $\Gamma_L$ is the fundamental solution from Lemma~\ref{lem:fund_sol}.
    Then, given a compact set $K\subseteq\R^{n+1}$, we define its $L$-\textit{thermal capacity} to be
    \begin{equation} \label{eq:def_cap}
        \text{Cap}_L(K):=\sup \big\{ \mu(K): \; \Gamma_L\mu\leq1\text{ in }\R^{n+1} , \; \mu\geq0, \; \text{and spt}(\mu)\subseteq K \big\}.
    \end{equation}
\end{definition}

With this definition in hand, we can define the TBCDC. Essentially, it asserts that the exterior of our domain $\Omega$ is uniformly ``large'', in the sense of capacities, to the past of every point in the lateral boundary (away from the initial boundary). Naturally, we need to look at the boundary (or, indirectly, the exterior) to the past, because those boundary values are the ones affecting the present (remember that heat only flows to the future).

\begin{definition}[The TBCDC] \label{def:TBCDC}
    Let $L$ be a parabolic operator as in Definition~\ref{def:operator}. Let $\Omega \subseteq \R^{n+1}$ be an open set. We say that $\Omega$ satisfies the \textit{time backwards capacity density condition (TBCDC)} for $L$ if there exist $a\in(0,1)$ and $c>0$ such that
    \begin{equation*} 
        \frac{ \text{Cap}_L\left( \big( \overline{Q_r(x_0)} \times [t_0-r^2,t_0-(ar)^2] \big) \, \cap \, \Omega^c\right)} {\text{Cap}_L\left(\overline{Q_r(x_0)}\times [t_0-r^2,t_0-(ar)^2] \right)} 
    	\geq c
    \end{equation*}
    for all $(x_0,t_0)\in\Sigma$ and $0<r<\sqrt{t_0-T_\text{min}}/4$.
\end{definition}

As already suggested in Section~\ref{sec:classification_boundary}, the TBCDC rules out many pathological behaviors of non-cylindrical domains, like the existence of most vertical faces (abnormal boundary). We will explain this in more detail in Section~\ref{sec:structure_boundary_TBCDC}.

We stress here that the definition of $L$-thermal capacity depends on the fundamental solution, which in turn depends on the parabolic operator $L$. 
To have a more versatile set of hypotheses and also operator-independent conditions, we introduce the TBHCC next. To that end, we begin by defining Hausdorff contents and measures.

\begin{definition}[Hausdorff content and measures]
    Given $s\in[0,\infty)$ and $\delta\in(0,\infty]$, we set
    \[\mathcal{H}^s_{\delta,p}(A):=\inf\bigg\{\sum_{i=1}^\infty (\diam_p A_i)^s:\; A\subseteq\bigcup_{i=1}^\infty A_i, \; \diam_p A_i\leq\delta \bigg\}, 
    \qquad A \subseteq \R^{n+1},\]
    where we emphasize with the subscript $p$ that we are measuring diameters with respect to the parabolic distance. 
    
    We call $\mathcal{H}^s_{\infty,p}$ the $s$-dimensional parabolic \textit{Hausdorff content} on $\R^{n+1}$, and we define the s-dimensional parabolic \textit{Hausdorff measure} by
    \[\mathcal{H}^s_p(A):=\lim_{\delta\to0^+}\mathcal{H}^s_{\delta,p}(A), \qquad A \subseteq \R^{n+1}.\]
\end{definition}

We are ready to define the TBHCC, our geometrical and operator-independent condition. It asserts that the exterior of our domain $\Omega$ is ``large'' to the past of every lateral boundary point (far from the bottom boundary). But, in contrast to the TBCDC, this ``size'' is measured with Hausdorff contents instead of capacities. 

\begin{definition}[The TBHCC] \label{def:TBHCC}
    We say that $\Omega \subseteq \R^{n+1}$ satisfies the \textit{time-backwards Hausdorff content condition (TBHCC)} if there exist $b,\e>0$ such that
    \[
        \mathcal{H}^{n+\e}_{\infty, p}(\Q_r^-(x_0,t_0)\cap\Omega^c)\geq b \, r^{n+\e}
    \]
    for all $(x_0,t_0)\in\Sigma$ and $0<r<\sqrt{t_0-T_\text{min}}/4$.
\end{definition}

The relevance of the threshold $n$ (codimension 2 in $\R^{n+1}$ endowed with the parabolic distance) in the exponent is classical, and will be seen in Remark~\ref{rem:codimension}.

Before studying the TBHCC and TBCDC more deeply, let us quickly show that the TBADR condition in \cite{GH} is stronger than the TBHCC. The proof is inspired by \cite{De}.

\begin{lemma}[TBADR implies TBHCC] \label{lem:TBADR_implies_TBHCC}
    Let $\Omega \subseteq \R^{n+1}$ be an open set satisfying the TBADR condition, as defined in \cite[Definition 1.22]{GH}. Then, it satisfies the TBHCC, with parameters $\varepsilon = 1$ and $b$ only depending on $n$ and the TBADR constants.
\end{lemma}
\begin{proof}
    First note that, by \cite[Appendix B]{BHHLN2}, the TBADR condition can be expressed in terms of ampleness with respect to $\mathcal{H}^{n+1}_p|_\Sigma$ instead of $\sigma$ (the surface measure over $\Sigma$, as defined in \cite[p. 1537]{GH}). Moreover, note that the upper bound in the ADR condition implies that, for any $A \subseteq \R^{n+1}$, it holds
    \[
    \mathcal{H}^{n+1}_p(A \cap \Sigma) 
    \leq 
    \mathcal{H}^{n+1}_p(\Q_{\diam_p(A \cap \Sigma)}(\x_A) \cap \Sigma)
    \lesssim 
    (\diam_p(A \cap \Sigma))^{n+1},
    \]
    where we have used any point $\x_A \in A \cap \Sigma$.    
    Now let $\x_0 \in \Sigma$ with $0 < r < \sqrt{t_0 - T_{\min}} / 4$. Using the previous bound and some elementary manipulations, we have
    \begin{align*}
        \mathcal{H}^{n+1}_{\infty, p}(\Q_r^-(\x_0) \cap \Omega^c)
        & \geq 
        \mathcal{H}^{n+1}_{\infty, p}(\Q_r^-(\x_0) \cap \Sigma)
        \\ & =
        \inf \Big\{ \sum_j (\diam_p (A_j \cap \Sigma))^{n+1} : \Q_r^-(\x_0) \cap \Sigma \subset \bigcup_j (A_j \cap \Sigma) \Big\}
        \\ & \gtrsim 
        \inf \Big\{ \sum_j \mathcal{H}^{n+1}_p (A_j \cap \Sigma) : \Q_r^-(\x_0) \cap \Sigma \subset \bigcup_j (A_j \cap \Sigma) \Big\}
        \\ & \geq 
        \mathcal{H}^{n+1}_p (\Q_r^-(\x_0) \cap \Sigma)
        \\ & \gtrsim 
        r^{n+1},
    \end{align*}
    where in the last step we are using the lower bound in the TBADR condition.\footnote{
        We note here that $\Q_r^-(\x_0) \supset \widetilde{\Q}_{r/\sqrt{n}}^-(\x_0)$ if we denote by $\widetilde{\Q}$ the cubes from \cite{GH} (which are defined with respect to a different --yet equivalent-- parabolic metric, see \cite[Section 1]{GH}). Therefore, our condition that $0 < r < \sqrt{t_0 - T_{\min}} / 4$ is equivalent to $0 < r / \sqrt{n} < \sqrt{t_0 - T_{\min}} / (4\sqrt{n})$, which allows us to use the TBADR as in \cite{GH}.
    }
\end{proof}


\subsection{The TBHCC implies the TBCDC for all operators} \label{sec:implies}

The main goal of this subsection is to show Proposition~\ref{thm:TBHCCimpliesTBCDCforallL}, which states that the TBHCC is indeed a stronger condition than the TBCDC, regardless of the operator.

Before proving it, let us develop a series of tools. First, we show that fundamental solutions and capacities associated to a general parabolic operators can be controlled by those of operators of the form $\partial_t - M\Delta$ for $M > 0$, whose behavior is very close to that of the heat equation. This will simplify many arguments later.

\begin{lemma} \label{lem:capLtoM}
    Let $L$ be a parabolic operator as in Definition~\ref{def:operator}. Then, there exist $M_1,M_2>0$  such that
    \begin{equation} \label{eq:comparison_fund_sol}
        \Gamma_{\partial_t-M_1\Delta}
        \, \lesssim \,
        \Gamma_L
        \, \lesssim \,
        \Gamma_{\partial_t-M_2\Delta}, 
        \qquad 
        \text{ in } \R^{n+1} \times \R^{n+1},
    \end{equation}
	so that concretely it holds 
	\begin{equation} \label{eq:aronson_easy}
		\Gamma_L(\X; \Y) \, \lesssim \, \lVert \X-\Y \rVert^{-n}, 
		\qquad \X, \Y \in \R^{n+1}.
	\end{equation}
    Moreover, for all compact $K\subseteq\R^{n+1}$,
    \begin{equation} \label{eq:comparison_capacities}
        \mathrm{Cap}_{\partial_t-M_1\Delta}(K)
        \, \gtrsim \,
        \mathrm{Cap}_L(K)
        \, \gtrsim \,
        \mathrm{Cap}_{\partial_t-M_2\Delta}(K).
    \end{equation} 
    Here, $M_1, M_2$ and the implicit constants in the above estimates only depend on $n$ and $\lambda$.
\end{lemma}
\begin{proof}
	Let us obtain the bounds where $M_2$ appears; the ones involving $M_1$ are analogous. The upper bound in \eqref{eq:comparison_fund_sol} follows at once from Lemma~\ref{lem:fund_sol} (and \eqref{eq:fund_sol_heat}) choosing $M_2 := N/4$. Let us call $C_2$ the implicit constant obtained in this estimate. Check also that \eqref{eq:aronson_easy} follows from this upper bound and \eqref{eq:aronson_easy_M}.
	
	With this choice of $M_2$, verifying the lower bound in \eqref{eq:comparison_capacities} is elementary. Fix $K\subseteq\R^{n+1}$ compact, and let $\mu$ be a positive Borel measure on $\R^{n+1}$ such that $\Gamma_{\partial_t-M_2\Delta}\mu\leq 1$ and $\textrm{spt}(\mu)\subseteq K$ (recall Definition~\ref{def:capacity}). Then, by the comparison between fundamental solutions obtained in the last paragraph, we obtain $\Gamma_L(\mu /C_2 ) \leq \Gamma_{\partial_t - M_2\Delta} \mu \leq 1$, so by definition of capacity, $(\mu/C_2)(K) \leq \text{Cap}_L(K)$, i.e., $\mu(K) \leq C_2 \text{Cap}_L(K)$. By arbitrariness of $\mu$, we obtain $\text{Cap}_{\partial_t - M_2 \Delta}(K) \leq C_2 \text{Cap}_L(K)$.
\end{proof}

This allows us to obtain an auxiliary estimate that will be used frequently in the sequel. 

\begin{lemma} \label{lem:cap_cylinder}
	Let $L$ be a parabolic operator as in Definition~\ref{def:operator}. Then, if $0 < b < a \leq 1$, 
	\[
	\mathrm{Cap}_L \Big(\overline{Q_r(X)} \times [t-(ar)^2, t-(br)^2] \Big)
	\approx r^n, 
	\qquad 
	\forall \, (X, t) \in \R^{n+1},
	\]
	with constants only depending on $n, a, b$, and ellipticity $\lambda$.
\end{lemma}
\begin{proof}
    Let us quickly sketch the proof, referring to \cite[Corollary 3.5]{MP} for more details. Abbreviate $K := \overline{Q_r(X)} \times [t-(ar)^2, t-(br)^2]$. For the upper bound, we have
	\begin{equation*}
		\text{Cap}_L(K)
        \lesssim 
        \text{Cap}_{\partial_t - M_1 \Delta}(K)
		\leq 
		\text{Cap}_{\partial_t - M_1 \Delta}(K,\mbf{Q}_{2r}(X,t))
		\lesssim 
		r^n,
	\end{equation*}
	where we have used Lemma~\ref{lem:capLtoM}, and the last inequalities follow from (3.11) and Lemma 3.3 from \cite{MP}, respectively. (They only work with the heat operator, but their arguments readily translate to $\partial_t - M_1\Delta$ at the expense of allowing constants depend on $M_1$.)
    
	For the lower bound, we again use Lemma~\ref{lem:capLtoM} to infer 
	\begin{equation*}
		\text{Cap}_L(K)
		\gtrsim 
		\text{Cap}_{\partial_t - M_2 \Delta}(K)
		\gtrsim 
		\frac{\mathcal{H}_p^{n+2}(K)}{r^2}
		\gtrsim 
		r^n,
	\end{equation*}
	where the last step is easy because $\mathcal{H}_p^{n+2}$ is a constant multiple of the Lebesgue measure, and the second-to-last step follows from \cite[Lemma 3.4]{MP} (and the comment in the proof of \cite[Corollary 3.5]{MP} to use $\mathcal{H}_p^{n+2}$ instead of $\mathcal{H}_{p, \infty}^{n+2}$).\footnote{Indeed, although in \cite{MP}, they state the result with capacities relative to $\Q_{2r}(X, t)$, one can check that their proof works for our capacities (relative to the whole $\R^{n+1}$) because they only use the easy upper bound \eqref{eq:aronson_easy} for the Green's function relative to $\Q_{2r}(X, t)$, and this is also available for our $\Gamma_L$. }
\end{proof}

A key tool for Proposition~\ref{thm:TBHCCimpliesTBCDCforallL} is Frostman's lemma (see e.g. \cite[Theorem 8.8]{M}).

\begin{lemma}[Frostman's lemma] \label{lem:frostman}
    Let $E\subseteq\R^{n+1}$ be a Borel set. Then, $\mathcal{H}^s_p(E)>0$ if and only if there exists a Borel measure $\mu$ on $\R^{n+1}$ such that $\mu(A)\leq C_1(\diam_p A)^s$ for all Borel sets $A\subseteq\R^{n+1}$ and $\mu(E)\geq c_2\mathcal{H}^s_{\infty, p}(E)$, where $C_1,c_2>0$ depend only on $n$.
\end{lemma}

With Frostman's lemma in hand, we can prove the following lemma regarding the self-improvement of the TBHCC, in the spirit of \cite[Appendix A]{GH} or \cite[Lemma 2.8]{MP}.

\begin{lemma}
\label{lem:TBEHCCselfimprove}
    Let the open set $\Omega \subseteq \R^{n+1}$ satisfy the TBHCC (with constants $b$ and $\varepsilon$), and fix $(x_0,t_0)\in\Sigma$ and $0<r<\sqrt{t_0-T_\mathrm{min}}/4$. Then, there exists a Borel measure $\mu$ and constants $a\in(0,1/2)$, $c_3>0$, depending only on $n$ and the TBHCC constants, such that 
    \[
        \mu\Big( \big(Q_r(x_0)\times(t_0-r^2,t_0-(ar)^2) \big) \, \cap \, \Omega^c\Big)\geq c_3 r^{n+\e},
    \]
    and for all Borel sets $A\subseteq\R^{n+1}$ it also holds
    \[
        \mu(A)\leq C_1(\diam_p A)^{n+\e},
    \]
    where $C_1$ is the constant from Frostman's lemma (Lemma \ref{lem:frostman}).
\end{lemma}

\begin{proof}
    For simplicity, assume $(x_0,t_0)=(0,0)$, and write $\Q_r^-:=\Q_r^-(0,0)$. By the TBHCC,  
    \begin{equation} \label{eq:using_TBHCC}
    	\mathcal{H}^{n+\e}_{\infty, p}(\Q_r^-\cap\Omega^c)\geq br^{n+\e}.
    \end{equation}
	Concretely, it holds $\mathcal{H}^{n+\e}_p(\Q_r^-\cap\Omega^c)>0$ (see \cite[Proposition 2.2]{BG}), whence Frostman's lemma (Lemma \ref{lem:frostman}) asserts that there exists a Borel measure $\mu$ such that
    \begin{equation} \label{eq:using_frostman}
        \mu(\Q_r^-\cap\Omega^c)
        \geq 
        c_2\mathcal{H}^{n+\e}_{\infty, p}(\Q_r^-\cap\Omega^c)
    \end{equation}
    and $\mu(A)\leq C_1(\diam_p A)^{n+\e}$ for all Borel sets $A\subseteq\R^{n+1}$.
    
    Now, set $\Phi_{ar}:=\overline{Q_r(0)}\times[-(ar)^2,0]$ for $a := 2^{-m}$, where $m\in\N$ will be chosen large enough momentarily. Let us decompose $\Phi_{ar}$ into a union of subcubes $\overline{\Q^i_{ar}}$ of sidelength $ar$ (all of the same size): clearly this can be done with $a^{-n}$ subcubes (this is the size of the \textit{lateral face} $\overline{Q_r(0)} \times \{0\}$), as shown in Figure~\ref{fig:decomposition}.
    Thus, using the properties of $\mu$ from above,
    \[
        \mu(\Omega^c\cap\Phi_{ar})
        =
        \mu \bigg( \Omega^c \cap \bigcup_{i=1}^{a^{-n}} \overline{\Q_{ar}^i} \bigg)
        \leq
        \sum_{i=1}^{a^{-n}} \mu(\Omega^c\cap \overline{\Q_{ar}^i})
        \leq 
        C_1\sum_{i=1}^{a^{-n}}(ar)^{n+\e}
        =
        C_1a^\e r^{n+\e}.
    \]
    
    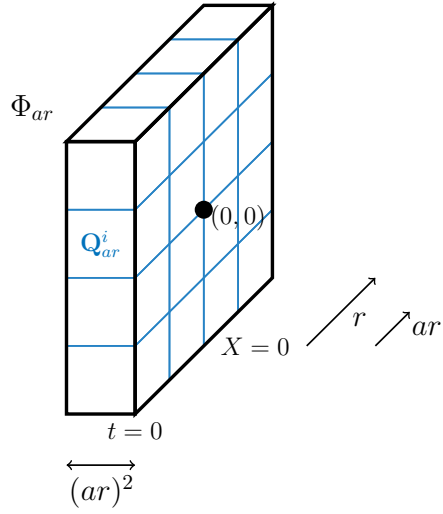
\begin{figure}
    	\centering 
    	\scalebox{.8}{
    	\begin{tikzpicture}[scale=1]
    		\draw[RoyalBlue!80, thick] (0.929, 6.063) rectangle (2.058, 4.934);
    		\draw[RoyalBlue!80, thick] (0.929, 4.934) rectangle (2.058, 3.805);
    		\draw[RoyalBlue!80, thick] (0.929, 3.805) rectangle (2.058, 2.677);
    		\draw[RoyalBlue!80, thick] (0.929, 2.677) rectangle (2.058, 1.548);
    		\draw[RoyalBlue!80, thick] (0.929, 4.37) rectangle (0.929, 4.37);
    		\draw[RoyalBlue!80, thick] (0.929, 6.063) -- (2.058, 6.063) -- (2.623, 6.628) -- (1.494, 6.628) -- cycle;
    		\draw[RoyalBlue!80, thick] (1.494, 6.628) -- (2.623, 6.628) -- (3.187, 7.192) -- (2.058, 7.192) -- cycle;
    		\draw[RoyalBlue!80, thick] (2.058, 7.192) -- (3.187, 7.192) -- (3.751, 7.757) -- (2.623, 7.757) -- cycle;
    		\draw[RoyalBlue!80, thick] (2.623, 7.757) -- (3.751, 7.757) -- (4.316, 8.321) -- (3.187, 8.321) -- cycle;
    		\draw[RoyalBlue!80, thick] (2.058, 6.063) -- (2.623, 6.628) -- (2.623, 5.499) -- (2.058, 4.934) -- cycle;
    		\draw[RoyalBlue!80, thick] (2.623, 5.499) -- (2.058, 4.934) -- (2.058, 3.805) -- (2.623, 4.37) -- cycle;
    		\draw[RoyalBlue!80, thick] (2.623, 4.37) -- (2.058, 3.805) -- (2.058, 2.677) -- (2.623, 3.241) -- cycle;
    		\draw[RoyalBlue!80, thick] (2.623, 3.241) -- (2.058, 2.677) -- (2.058, 1.548) -- (2.623, 2.112) -- cycle;
    		\draw[RoyalBlue!80, thick] (3.187, 7.192) -- (2.623, 6.628) -- (2.623, 5.499) -- (3.187, 6.063) -- cycle;
    		\draw[RoyalBlue!80, thick] (3.751, 7.757) -- (3.187, 7.192) -- (3.187, 6.063) -- (3.751, 6.628) -- cycle;
    		\draw[RoyalBlue!80, thick] (4.316, 8.321) -- (3.751, 7.757) -- (3.751, 6.628) -- (4.316, 7.192) -- cycle;
    		\draw[RoyalBlue!80, thick] (4.316, 7.192) -- (3.751, 6.628) -- (3.751, 5.499) -- (4.316, 6.063) -- cycle;
    		\draw[RoyalBlue!80, thick] (3.751, 6.628) -- (3.187, 6.063) -- (3.187, 4.934) -- (3.751, 5.499) -- cycle;
    		\draw[RoyalBlue!80, thick] (3.187, 6.063) -- (2.623, 5.499) -- (2.623, 4.37) -- (3.187, 4.934) -- cycle;
    		\draw[RoyalBlue!80, thick] (3.187, 4.934) -- (2.623, 4.37) -- (2.623, 3.241) -- (3.187, 3.805) -- cycle;
    		\draw[RoyalBlue!80, thick] (3.187, 3.805) -- (2.623, 3.241) -- (2.623, 2.112) -- (3.187, 2.677) -- cycle;
    		\draw[RoyalBlue!80, thick] (3.751, 5.499) -- (3.187, 4.934) -- (3.187, 3.805) -- (3.751, 4.37) -- cycle;
    		\draw[RoyalBlue!80, thick] (4.316, 6.063) -- (3.751, 5.499) -- (3.751, 4.37) -- (4.316, 4.934) -- cycle;
    		\draw[RoyalBlue!80, thick] (3.751, 4.37) -- (3.187, 3.805) -- (3.187, 2.677) -- (3.751, 3.241) -- cycle;
    		\draw[RoyalBlue!80, thick] (4.316, 4.934) -- (3.751, 4.37) -- (3.751, 3.241) -- (4.316, 3.805) -- cycle;
    		
    		\draw[ultra thick] (4.316, 8.321) -- (3.187, 8.321) -- (0.929, 6.063) -- (2.058, 6.063) -- cycle;
    		\draw[ultra thick] (2.058, 6.063) -- (4.316, 8.321) -- (4.316, 3.805) -- (2.058, 1.548) -- cycle;
    		\draw[ultra thick] (0.929, 6.063) -- (2.058, 6.063) -- (2.058, 1.548) -- (0.929, 1.548) -- cycle;
    		
    		\node[circle, fill, inner sep=3pt] at (3.187, 4.934) {};
    		
    		\draw[thick, <->] (0.929, 0.701) -- (2.058, 0.701);    		
    		\draw[thick, ->] (4.88, 2.677) -- (6.009, 3.805);
            \draw[thick, ->] (6.009, 2.677) -- (6.574, 3.241);
    		
    		\node[anchor=center, font=\Large] at (1.508, 0.268) {$(ar)\textsuperscript{2}$};
    		\node[anchor=center, font=\large] at (2.058, 1.265) {$t = 0$};
    		\node[anchor=center, font=\large] at (4.034, 2.677) {$X = 0$};
    		\node[anchor=center, font=\Large] at (5.75, 3.109) {$r$};
            \node[anchor=center, font=\Large] at (6.856, 2.959) {$ar$};
    		\node[anchor=center, font=\Large] at (0.365, 6.628) {$\Phi_{ar}$};
    		\node[anchor=center, font=\large, text=RoyalBlue] at (1.494, 4.37) {$\Q^i_{ar}$};
    		\node[anchor=center, font=\large] at (3.751, 4.804) {$(0, 0)$};
    	\end{tikzpicture}
    	}
    	\caption{Because of the parabolic scaling, the big slab $\Phi_{ar}$ can be subdivided in smaller cubes $\Q_{ar}^i$ of the same size: one needs as many as to cover the front-most face $Q_r(0) \times \{0\}$.}
    	\label{fig:decomposition}
    \end{figure}
    
    Choosing $m$ large (and hence $a$ small) so that $C_1a^\e \leq \frac{1}{2}bc_2$, we obtain, using \eqref{eq:using_TBHCC} and \eqref{eq:using_frostman},
    \[
        \mu(\Omega^c\cap\Phi_{ar})
        \leq
        \frac{1}{2}bc_2r^{n+\e}
        \leq
        \frac{1}{2}c_2\mathcal{H}^{n+\e}_{\infty, p}(\Omega^c\cap \Q_r^-)
        \leq
        \frac{1}{2}\mu(\Omega^c\cap \Q_r^-).
    \]
    Therefore, we can hide the contribution of $\Phi_{ar}$ and finish the proof using \eqref{eq:using_TBHCC} and \eqref{eq:using_frostman}:
    \[
    \mu(\Omega^c\cap(\Q_r^-\setminus\Phi_{ar}))
    \geq
    \frac{1}{2}\mu(\Omega^c\cap \Q_r^-)
    \geq
    \frac{1}{2}bc_2r^{n+\e}.
    \]
\end{proof}

\begin{remark} \label{rem:codimension}
	The key point of the proof is to be able to hide the contribution of $\Phi_{ar}$ for small $a$. For that, we have used that after the decomposition in subcubes, a power $a^\e$ pops up. It is of course crucial that the exponent is positive, which happens in our setting with the TBHCC using exponents up to codimension 2. This is in close relationship with the distinguished homogeneity of the time variable in the parabolic setting.
\end{remark}

After these preliminary investigations, we are ready to tackle the proof of the main result of this subsection, Proposition~\ref{thm:TBHCCimpliesTBCDCforallL}, stating that the TBHCC is stronger than the TBCDC. 

\begin{proof}[{\textbf{Proof of Proposition~\ref{thm:TBHCCimpliesTBCDCforallL}}}]
	Let us prove the TBCDC at the scale of $\Q_r(x_0, t_0)$, where $(x_0, t_0) \in \Sigma$ and $0<r<\sqrt{t_0-T_\textrm{min}}/4$.
	Without loss of generality, take $(x_0,t_0) = (0, 0) = \mbf{0}$. 	
	Let us abbreviate also by writing 
	\[
	\Q_{a, r}^- := Q_r(0) \times (-r^2, -(ar)^2).
	\]

    First, we note that it suffices to prove the result for the operators of the form $L = \partial_t - M\Delta$ for some $M > 0$. Indeed, assuming that, Lemmas~\ref{lem:capLtoM} and \ref{lem:cap_cylinder} yield
    \begin{equation*}
        \frac{\text{Cap}_L ( \overline{\Q_{a, r}^-} \, \cap \, \Omega^c )}
        {\text{Cap}_L (\overline{\Q_{a, r}^-} )}
        \gtrsim 
        \frac{\text{Cap}_{\partial_t - M_2 \Delta} ( \overline{\Q_{a, r}^-} \, \cap \, \Omega^c )}
        {r^n}
        \gtrsim 
        1,
    \end{equation*}
    which directly finishes the proof.
			
	Therefore, we are left to consider the case of $L = \partial_t - M\Delta$ for some $M > 0$. By Lemma \ref{lem:TBEHCCselfimprove}, there exists a Borel measure $\mu$ and constants $a\in(0,1/2)$ and $C_1,c_3>0$ such that
    \begin{equation} \label{eqn:mubound1}
        \mu\left(\Q_{a, r}^- \, \cap \, \Omega^c \right)
        \geq 
        c_3 r^{n+\e},
    \end{equation} 
    and for all Borel sets $A\subseteq\R^{n+1}$,
    \begin{equation}
    \label{eqn:mubound2}
        \mu(A)\leq C_1(\diam_p A)^{n+\e}.
    \end{equation}

	We would like to test the definition of capacity with the measure $\mu$. But for that, we need to modify it slightly so that it falls into the domain of the supremum in \eqref{eq:def_cap}. For that purpose, first define its restriction
    \[
        \nu 
        := 
        \mu\mres\left(\Q_{a, r}^- \, \cap \, \Omega^c\right).
    \]
    
    Let us now compute $\norm{\Gamma_L \nu}_\infty$. First, consider $\X \in \R^{n+1} \setminus \Q_{2r}(\mbf{0})$, so that for any $\Y \in \text{spt} (\nu) \subseteq \Q_r(\mbf{0})$, we have $\norm{\X - \Y} \geq r$. Hence, using \eqref{eq:aronson_easy_M} we can compute 
    \begin{equation*}
        \Gamma_L\nu(\X) 
        \leq 
        C_Mr^{-n}\nu\left(\Q_{a, r}^- \cap\Omega^c\right)
        \leq 
        C_1 C_Mr^{-n}r^{n+\e}
        =
        C_1 C_Mr^\e,
    \end{equation*}
    where in the second to last step we used (\ref{eqn:mubound2}). 
    On the other hand, let us consider $\X \in \Q_{2r}(\mbf{0})$. Then, for any $\Y \in \text{spt} (\nu)$ it holds $\norm{\X - \Y} < 3r$ by the triangle inequality, whence again using \eqref{eq:aronson_easy_M} and later \eqref{eqn:mubound2}, we estimate 
    \begin{multline*}
        \Gamma_L\nu(\X) 
        \leq
        \sum_{k=-\infty}^{\log_2 r + 3} \iint_{\Q_{2^{k+1}}(\X) \setminus \Q_{2^k}(\X)} \Gamma_L(\X;\Y)\,d\nu(\Y)
        \\
        \lesssim
        \sum_{k=-\infty}^{\log_2 r + 3} 2^{-kn} \nu(\Q_{2^{k+1}}(\X))
        \lesssim
        \sum_{k=-\infty}^{\log_2 r + 3} 2^{-kn}(2^{k+1})^{n+\e}
        \approx
        \sum_{k=-\infty}^{\log_2 r + 3}2^{k\e}
        \approx
        r^\e.
    \end{multline*}

    Thus, joining both cases, we have shown that there is some $C_\nu>0$ such that $\lVert \Gamma_L\nu\rVert_\infty\leq C_\nu r^\e$. Hence, defining $\tilde{\nu} := (C_\nu r^\e)^{-1} \, \nu$, it holds:
    \begin{enumerate}
        \item $\textrm{spt}(\tilde{\nu})\subseteq \overline{\Q_{a, r}^-} \, \cap \, \Omega^c$ (see the definition of $\nu$ above), 
        
        \item $\lVert \Gamma_L\tilde{\nu}\rVert_\infty\leq 1$ by the normalization in the definition of $\tilde{\nu}$, and
        
        \item $\tilde{\nu} ( \overline{\Q_{a, r}^-} \, \cap \, \Omega^c ) \geq \frac{c_3}{C_\nu}r^n$ by \eqref{eqn:mubound1}.
    \end{enumerate}
    Thus, testing the definition of capacity with $\tilde{\nu}$ and using Lemma~\ref{lem:cap_cylinder}, we conclude because
    \[
        \frac{\text{Cap}_L ( \overline{\Q_{a, r}^-} \, \cap \, \Omega^c )}
        {\text{Cap}_L ( \overline{\Q_{a, r}^-} )}
        \gtrsim
        \frac{\tilde{\nu} ( \overline{\Q_{a, r}^-} \, \cap \, \Omega^c )}{r^n}
        \gtrsim 
        1.
    \]
\end{proof}


\subsection{The TBCDC implies the parabolic Wiener's criterion} \label{sec:wiener}

Already more than a century ago, Wiener showed that there is a potential theoretic criterion (based on capacities) characterizing the sets in which solutions to the Laplace equation are continuous all the way up to the boundary (of course, one needs to impose continuous boundary data for that purpose). This result turned out to be the inception of the very successful joint development of potential theory and PDE throughtout the last century. 

In the parabolic setting, a suitable version of Wiener's criterion is true, as was shown by Garofalo and Lanconelli in \cite{GL} (relying on a previous result by Evans and Gariepy \cite{EG}), see Theorem~\ref{thm:ParabolicWienerCriterion}. Our goal in this subsection is to show that the TBCDC is a stronger condition than that in \cite{GL} (which we do in Lemma~\ref{thm:TBCDCimpliesWiener}), whence it concretely implies that the continuous Dirichlet problem is solvable (at least for the operators considered in \cite{GL}). Actually, the TBCDC is somehow a quantitative form of the Wiener's criterion, which is qualitative in nature. For the concrete case of the heat equation, all of this was extensively studied in \cite[Section 3]{MP}.

Before proving the main result, let us give some preliminary definitions.

\begin{definition}[$L$-regularity]
	Let $\Omega \subseteq \R^{n+1}$ be a bounded open set. Let $L$ be a parabolic operator as in Definition~\ref{def:operator}. We say that a point $\x_0 \in\partial_e\Omega$ is \textit{$L$-regular} if
	\[
	\lim_{\X \to \x_0} u(\X)=f(\x_0)
	\]
	for all $f\in C_c(\partial_e\Omega)$, where $u$ is the Perron-Wiener-Brelot-Bauer solution for $L$ in $\Omega$ with boundary data $f$.\footnote{
        A sketch of this (very classical) construction can be found, for instance, in the introduction of \cite{GL}. A very thorough development can be found in the book of Watson \cite[Chapter 8]{W2} (although only for the heat equation). The more general version including parabolic equations dates back to Bauer \cite[Chapter IV]{Bauer}, which generalizes and builds an axiomatic theory upon the subsequent works of Perron, Wiener and Brelot. The first to note that Perron's method works for the heat equation was Sternberg \cite{Sternberg}.  
    }
\end{definition}

\begin{definition}[$L$-heat ball]
	Let $L$ be a parabolic operator as in Definition~\ref{def:operator}. Given $\x_0 \in\R^{n+1}$ and $r > 0$, we define the \textit{$L$-heat ball} of radius $r$ centered at $\x_0$ by 
	\[
	\mathcal{B}^L_r(\x_0) := \big\{ \X \in \R^{n+1} : \; \Gamma_L(\x_0, \X) > (4\pi r)^{-n/2} \big\}. 
	\]
\end{definition}

With this, we can state the parabolic Wiener criterion of Garofalo and Lanconelli, shown in \cite[Theorem 1.1]{GL} (see also \cite[(3.22)]{MP}).

\begin{theorem}[Parabolic Wiener Criterion]
	\label{thm:ParabolicWienerCriterion}
	Let $\Omega\subseteq\R^{n+1}$ be a bounded open set, and $L$ be a parabolic operator as in Definition~\ref{def:operator}, where additionally $A$ is $C^\infty$. Then, a point $\x_0 \in\partial_e\Omega$ is $L$-regular if and only if
	\begin{equation*}
		\sum_{k=1}^\infty \lambda^{-kn/2} \mathrm{Cap}_L \Big(\Omega^c \, \cap \, \big(\{\x_0\} \cup 
        \overline{\mathcal{B}^L_{\lambda^k}(\x_0) \setminus \mathcal{B}^L_{\lambda^{k+1}}(\x_0)}\big) \Big)
		=
		+\infty
	\end{equation*}
	for some $\lambda \in (0, 1)$ (and actually any $\lambda \in (0, 1)$, since the convergence of the series does not depend on the value of $\lambda$ in that range).
\end{theorem}

\begin{remark}
	\label{rmk:ExistenceOfParabolicMeasureForCinfty} 
	We note that if every point in $\partial_e\Omega$ is $L$-regular, then the Perron-Wiener-Brelot-Bauer solutions solve the continuous Dirichlet problem. In particular, the parabolic measure exists on $\Omega$ and yields solutions to the continuous Dirichlet problem. 
\end{remark}

\begin{remark} \label{rem:symmetry}
    The reader may note that, in Theorem~\ref{thm:ParabolicWienerCriterion}, we have removed the assumption on symmetry of $A$ from the original reference \cite{GL}. This was also commented in \cite{GH}, but let us give some concrete references to ascertain it. Indeed, one can remove the symmetry assumption because of the result in \cite[Theorem 3.1]{GZ}.\footnote{
        To verify the hypotheses of \cite[Theorem 3.1]{GZ}, we need to ensure that solutions attain the boundary values in a \textit{weak} sense (as in \cite[Section 2]{GZ}). Indeed, PWB solutions attain boundary values in a Sobolev sense (see \cite[Corollary 9.29]{HKM}), which is stronger than the aforementioned weak sense (see \cite[p. 293]{Z}).
    }
    Alternatively, one can check that the proof of \cite[Theorem 1.1]{FGL} (which is an extension of \cite{GL} to $C^{1, \text{Dini}}$ coefficients) extends to the case of non-symmetric coefficients with virtually no changes,\footnote{
        Indeed, the main points when checking that \cite{FGL} works for non-symmetric coefficients are the following. When writing the operator in non-divergence form, in (2.2) (adopting the terminology of \cite{FGL} from now on), $a_{ij}$ can be taken to be symmetric, i.e., one can use $(a_s)_{ij}$, the entries of the symmetric part $A_s := (A + A^T)/2$ of $A$. The same applies to $a_{ij}$ in (2.3), but the drift coefficient $b_i(\zeta)$ in (2.3) should be changed to $\sum_{j=1}^n D_{x_j} a_{ji}(\zeta)$: in the end, this change will not be relevant because it still satisfies the same Dini continuity estimate. Afterwards, one can check that the fundamental solution has the shape of (2.7), but modifying (2.8) to read $Q_z(y) := A_s^{-1}(z) y \cdot y$. Then, one should add the hypothesis that $\abs{t-\tau} \leq 1$ to Lemma 2.1, what does not have an impact on the statement of Theorem 2.2 because the sets $\Omega_r^0(z)$ are bounded both in space and time. Upon routine modifications based on these observation, the rest of the proofs of the paper remain valid, without any need for symmetry.
    } as would be expected from the fact that fundamental solution enjoys very good properties even in the presence of merely measurable coefficients, without any symmetry assumption, as shown in \cite{QX} (see also \cite{SSSZ}).
\end{remark}

We are ready to show the main result of the subsection: stating that the TBCDC is stronger (and in fact, a quantification) of the parabolic Wiener's criterion (compare \cite[Lemma 3.12]{MP}).

\begin{lemma}[TBCDC implies Wiener's criterion]
\label{thm:TBCDCimpliesWiener}
    Let $\Omega\subseteq\R^{n+1}$ be a bounded open set, and $L$ be a parabolic operator as in Definition~\ref{def:operator}, where additionally $A$ is $C^\infty$. If $\Omega$ satisfies the TBCDC for $L$, then every point $(x_0,t_0)\in\partial_e\Omega$ is $L$-regular. In particular, the parabolic measure exists for $L$ on $\Omega$.
\end{lemma}

\begin{proof}
    Fix $(x_0,t_0)\in\partial_e\Omega$. Let $\lambda \in (0, 1)$ and $k \in \N$. We claim that if one chooses $\lambda$ small enough and one puts $r_k := \frac{(4\pi)^{1/2} N^{1/n}}{a} \lambda^{(k+1)/2}$ ($N$ is the constant from Lemma~\ref{lem:fund_sol} and $a$ comes from the TBCDC), then it holds
    \begin{equation} \label{claim:wiener_content}
    \overline{Q_{r_k}(x_0)} \times [t_0-r_k^2,t_0-(ar_k)^2]
    \subseteq 
    \overline{\mathcal{B}^L_{\lambda^k}(x_0,t_0) \setminus \mathcal{B}^L_{\lambda^{k+1}}(x_0,t_0)}.
    \end{equation}
    Taking this claim momentarily for granted, it is easy to finish. Indeed, using the claim, the TBCDC (and Lemma~\ref{lem:cap_cylinder}), and later the choice of $r_k$, we obtain 
    \begin{multline*}
    	\sum_{k=1}^\infty \lambda^{-kn/2} \mathrm{Cap}_L \Big(\Omega^c \, \cap \, \big(\{\x_0\} \cup \overline{\mathcal{B}^L_{\lambda^k}(\x_0) \setminus \mathcal{B}^L_{\lambda^{k+1}(\x_0)}} \big) \Big)
    	\\ \geq 
    	\sum_{k=1}^\infty \lambda^{-kn/2} \mathrm{Cap}_L\Big( \Omega^c \cap \big( \overline{Q_{r_k}(x_0)} \times [t_0 - r_k^2, t_0 - (ar_k)^2] \big) \Big)
    	\\ 
    	\gtrsim 
    	\sum_{k=1}^\infty \lambda^{-kn/2} r_k^n
    	\approx 
    	\sum_{k=1}^\infty \lambda^{-kn/2} (\lambda^{(k+1)/2})^n
    	= 
    	\sum_{k=1}^\infty \lambda^{n/2}
    	=
    	+\infty,
    \end{multline*}
    so Theorem \ref{thm:ParabolicWienerCriterion} implies that $(x_0,t_0)$ is $L$-regular. In particular, Remark \ref{rmk:ExistenceOfParabolicMeasureForCinfty} implies that the parabolic measure exists for $L$ on $\Omega$.
    
    Therefore, everything boils down to establishing the claim~\eqref{claim:wiener_content}. For that purpose, let $(X, t) \in \overline{Q_{r_k}(x_0)} \times [t_0-r_k^2,t_0-(ar_k)^2]$. Using Lemma~\ref{lem:fund_sol}, we easily get 
    \[
    \Gamma_L(x_0,t_0; X,t) \leq
    \frac{N}{(t_0-t)^{n/2}}\exp\left(-\frac{|x_0-X|^2}{N(t_0-t)}\right)
    \leq 
    N(ar_k)^{-n}
    \leq 
    (4\pi\lambda^{k+1})^{-n/2}
    \]
    using the definition of $r_k$. Similarly, using Lemma~\ref{lem:fund_sol} and the definition of $r_k$, we get
    \begin{multline*}
    	\Gamma_L(x_0,t_0; X,t)
    	\geq
    	\frac{1}{N(t_0-t)^{n/2}}\exp\left(-\frac{N|x_0-X|^2}{t_0-t}\right)
    	\geq
    	\frac{1}{Nr_k^n}\exp\left(-\frac{Nr_k^2}{(ar_k)^2}\right)
    	\\
    	=
    	a^nN^{-2}\lambda^{-n/2}\exp(-Na^{-2}) (4\pi\lambda^k)^{-n/2}
    	\geq
    	(4\pi\lambda^k)^{-n/2},
    \end{multline*}
	where the last step is true if one chooses $\lambda > 0$ small enough. This finishes the proof of the claim \eqref{claim:wiener_content}, and hence of the lemma.
\end{proof}

\subsection{Structure of boundaries of domains satisfying the TBCDC} \label{sec:structure_boundary_TBCDC}

In Section~\ref{sec:classification_boundary}, we have seen a very thorough classification of the different parts of the boundary of a non-cylindrical domain. Luckily, under our assumption TBCDC (or also TBHCC, by Section~\ref{sec:implies}), this classification becomes much simpler and understandable. For instance, it is trivial to note that 
\begin{quote}
	If $\Omega$ satisfies the TBCDC for any $L$ (or the TBHCC), then 
	\[
	\partial_a \Omega \cap \T_{< T_{\max}} = \emptyset, 
	\qquad 
	\partial_a \Omega = \partial_s \Omega \cap \T_{=T_{\max}}.\] 
	Thus,
	\[ 
	\partial_e \Omega = \partial_n \Omega 
    =\begin{cases}
		\mathcal{P}\Omega&\textrm{if }\Omega\textrm{ is bounded}\\
		\mathcal{P}\Omega\cup\{\infty\}&\textrm{if }\Omega\textrm{ is unbounded.}
	\end{cases}
	\]
\end{quote} 
Concretely, this informs us that the only possible vertical faces that may exist are at the terminal time $T_{\max}$ (which is reasonable because it is not included in $\Sigma$, and we only impose the TBCDC on $\Sigma$) or are part of the bottom boundary $\mathcal{B}\Omega$ (which is also reasonable because it allows to impose initial values).

The following result gives us deeper understanding of boundaries of domains satisfying the TBCDC, and will be useful later.

\begin{lemma}
\label{lem:StructureOfUnboundedTBCDCDomains}
     Let $\Omega\subseteq\R^{n+1}$ be an unbounded open set that satisfies the TBCDC for some parabolic operator. Then, $\partial_e\Omega\setminus\{\infty\}$ is either unbounded or empty.
\end{lemma}

\begin{proof}
    Suppose $\partial_e\Omega\setminus\{\infty\}\neq\emptyset$. We split into cases.
    
    \textbf{Case 1: $T_\mathrm{min}>-\infty$ and $T_\mathrm{max}<+\infty$.}
    In this case, the spatial projection of $\Omega$, namely
    \[
        A:= \big\{X\in\R^n:\; \exists \, t\in\R\textrm{ such that }(X,t)\in\Omega \big\},
    \]
    is an unbounded set. For each $X\in A$, set
    \[
        T_\textrm{min}(X):=\inf \big\{t \in \R: \; (X,t)\in\Omega\big\}.
    \]
    Then, for each $X\in A$, $(X,T_\textrm{min}(X))\in\partial\Omega$. Moreover, the TBCDC ensures that $\partial\Omega=\partial_e\Omega\cup\big(\partial_s\Omega \cap \T_{=T_\textrm{max}}\big)$ in this case. Thus, since $T_\textrm{min}\leq T_\textrm{min}(X)<T_\textrm{max}$ for each $X\in A$, it follows that $(X,T_\textrm{min}(X))\in\partial_e\Omega$ for each $X\in A$. Thus, $\{(X,T_\textrm{min}(X)):X\in A\}\subseteq\partial_e\Omega\setminus\{\infty\}$, so since $A$ is unbounded, it must be the case that $\partial_e\Omega\setminus\{\infty\}$ is unbounded.

    \textbf{Case 2: $T_\mathrm{min}>-\infty$ and $T_\mathrm{max}=+\infty$.} In this case, if $A$ as in Case 1 is unbounded, we can repeat the argument above. Otherwise, $A$ is bounded, so the projection in time 
    \[
        B:= \big\{t\in\R: \; \exists \, X\in\R^n\textrm{ such that }(X,t)\in\Omega \big\}
    \]
    is unbounded. For each $t\in B$, choose some $(X_1(t),X_2(t),\ldots,X_n(t),t)\in\Omega$ and set
    \[
        M(t):=\sup \big\{Y \in \R:(Y,X_2(t),\ldots,X_n(t),t)\in\Omega \big\},
    \]
    which is a well-defined real number because $A$ is bounded.
    Then, for each $t\in B$, it holds $(M(t),X_2(t),\ldots,X_n(t),t)\in\partial\Omega$. Moreover, the TBCDC ensures that $\partial\Omega=\partial_e\Omega$ in this case. Thus,
    \[
        \big\{(M(t),X_2(t),\ldots,X_n(t),t):t\in B\big\}
        \subseteq
        \partial_e\Omega\setminus\{\infty\},
    \] 
    so since $B$ is unbounded, it must be the case that $\partial_e\Omega\setminus\{\infty\}$ is unbounded.

    \textbf{Case 3: $T_\mathrm{min}=-\infty$.} In this case, suppose for the sake of obtaining a contradiction that $\partial_e\Omega\setminus\{\infty\}$ is bounded. Set
    \[
        t_0:=T_\textrm{min}(\partial_e\Omega\setminus\{\infty\})>-\infty.
    \]
    We claim that $\T_{< t_0} \subseteq \Omega$. Taking this claim momentarily for granted, let us obtain the desired contradiction. Since $\partial_e\Omega$ is closed (recall Section~\ref{sec:classification_boundary}), there exists some $x_0\in\R^n$ so that $(x_0,t_0)\in\partial_e\Omega$. Then, by our claim, $\Q^-_r(x_0,t_0)\subseteq\Omega$ for all $r>0$, which implies that $(x_0,t_0)\not\in\mathcal{P}\Omega$. Hence, $(x_0,t_0)\in\partial_{ss}\Omega$, which is actually a contradiction because the TBCDC implies that $\partial_{ss}\Omega=\emptyset$ (see the paragraphs before this lemma). 
    
    Thus, we are left to show the claim. First, we show that if a time slice (in $\T_{<t_0}$) contains some point in $\Omega$, then it is fully contained in $\Omega$. Indeed, assume that $(X_1,X_2,\ldots,X_n,t)\in\Omega$ is a point with $t<t_0$. Then, if any of the values 
    \begin{align*}
        \inf\{Y \in \R :(Y,X_2,\ldots,X_n,t)\in\Omega\},
        &\quad&
        \sup\{Y \in \R:(Y,X_2,\ldots,X_n,t)\in\Omega\}
        \\
        \vdots\qquad\qquad\qquad&&\vdots\qquad\qquad\qquad
        \\
        \inf\{Y\in \R :(X_1,X_2,\ldots,Y,t)\in\Omega\},
        &\quad&
        \sup\{Y \in \R:(X_1,X_2,\ldots,Y,t)\in\Omega\},
    \end{align*}
	is finite, we can proceed as in Case 2 to show that there exists a point in the boundary at time $t<t_0$, a contradiction with the definition of $t_0$. Thus, for all $X\in\R^n$, $(X,t)\in\Omega$. 
	
	Therefore, to finish the proof of the claim, it suffices show that all time slices (in $\T_{<t_0}$) contain some point in $\Omega$. For that, assume that $(X, t) \in \Omega$ satisfies $t < t_0$. Then set
	\[
	\tilde{t} := \sup\{s_0\in \R: \; s_0\geq t, \; \text{and } (X, s)\in\Omega\textrm{ for all }t\leq s\leq s_0\}.
	\]
	If $\tilde{t} = +\infty$, then certainly $\tilde{t} \geq t_0$. Otherwise $\tilde{t} < +\infty$, so it must happen that $(X, \tilde{t}) \in \pom$, so $\tilde{t} \geq t_0$. (Indeed, if $(X, \tilde{t}) \in \partial_e \Omega$, it follows from the definition of $t_0$. Otherwise, the TBCDC implies that $(X, \tilde{t}) \in \partial_s \Omega \cap \T_{=T_{\max}}$, which trivially implies $\tilde{t} = T_{\max} \geq t_0$.) Thus, by the definition of $\tilde{t}$, $(X, s) \in \Omega$ for every $t \leq s < t_0$. 
	Since we can repeat this argument with points $(X, t)$ with $t \to -\infty$ (recall that $T_{\min} = -\infty$), this shows that every time slice (in $\T_{< t_0}$) contains some point of $\Omega$. This finishes the proof of the claim.
\end{proof}


\section{Proof of Theorem \ref{mainthm:Bourgain}, Hölder decay up to the boundary} \label{sec:bourgain_proof}

In this section, we will show Theorem~\ref{mainthm:Bourgain}. Note that in that theorem we assume \textit{a priori} the existence of the parabolic measure for $L$. For sufficiently nice operators $L$, the existence of the parabolic measure is granted by the Wiener's criterion, as discussed in Remark~\ref{rmk:ExistenceOfParabolicMeasureForCinfty}. It will turn out that under the TBCDC assumption, the parabolic measure exists for any $L$ with merely bounded coefficients as in Definition~\ref{def:operator}: we will show this in Section~\ref{sec:existence_parabolic_measure} with the aid of Theorem~\ref{mainthm:Bourgain}, actually. 

\subsection{Proof of Bourgain's estimate (\ref{maineqn:BourgainEstimate})}

Instead of proving \eqref{maineqn:BourgainEstimate} directly within the framework of Theorem~\ref{mainthm:Bourgain}, let us prove a more general non-degeneracy estimate for the parabolic measure (which will easily imply \eqref{maineqn:BourgainEstimate} under the TBCDC assumption). 

\begin{lemma}[Capacity estimate for the parabolic measure]
\label{thm:capacityestimate}
    Let $L$ be a parabolic operator as in Definition~\ref{def:operator}. Assume that the parabolic measure for $L$ in $\Omega$ exists, and denote it by $\omega^{X,t}_L$. Then, there exist constants $M_0,c>0$, depending only on $n, \lambda$, and $a$, such that
    \[
        \omega_L^{X,t}\Big(Q_{M_0r}(x_0) \times (t_0 - r^2, t_0 + r^2) \Big)
        \geq 
        c \, \frac{\mathrm{Cap}_L \Big( \big( \overline{Q_r(x_0)}\times [t_0-r^2,t_0-(ar)^2] \big)\, \cap \, \Omega^c \Big)}{r^n}
    \]
    for all $(x_0,t_0)\in\partial_e\Omega\setminus\{\infty\}$, $r>0$, and $(X,t) \in \big( Q_r(x_0) \times (t_0-(ar)^2/2,t_0+r^2) \big) \, \cap \, \Omega$.
\end{lemma}

\begin{proof}
	The proof is a slight modification of that in \cite[Lemma 3.15]{MP}.
    We may assume that $(x_0,t_0)=(0,0)$. For brevity, set
    \[
        K := \big( \overline{Q_r(0)} \times [-r^2,-(ar)^2] \big) \, \cap \, \Omega^c,
    \]
    \[
        \Q^-_{a,r} := Q_r(0) \times (-r^2,-(ar)^2),
        \qquad 
        \Q^+_{a,r} := Q_r(0) \times (-(ar)^2/2,r^2).
    \]
    
    By definition of capacity, there exists a Radon measure $\mu$ supported on $K$ such that $\lVert \Gamma_L \mu\rVert_\infty\leq 1$ and $\mu(K) \geq \text{Cap}_L(K)/2$. Note that if $(X,t)\in \Q^+_{a,r}$ and $(Y,s)\in K \subseteq \Q^-_{a,r}$,
    \[
        (ar)^2/2\leq t-s\leq 2r^2\qquad\textrm{and}\qquad|X-Y|\leq 2r,
    \]
    whence using Aronson's bounds (from Lemma~\ref{lem:fund_sol}) it holds
    \begin{align*}
        \Gamma_L(X,t;Y,s)
        \geq
        \frac{1}{N(2r^2)^{n/2}} \exp\left(-\frac{N(2r)^2}{(ar)^2/2}\right)
        =\frac{\exp(-8Na^{-2})}{2^{n/2}N} r^{-n}
        =:c_4r^{-n},
    \end{align*}
	which trivially implies that
    \begin{equation} \label{eq:potential_below_bourgain}
        \Gamma_L\mu(X,t)\geq c_4 r^{-n} \mu(K), 
        \qquad \text{if } (X,t)\in \Q^+_{a,r}.
    \end{equation}
    
    On the other hand, fix $M_0>0$ large enough. If $(X,t)\in\R^{n+1}\setminus \Q_{M_0r}(\mbf{0})$ and $(Y,s)\in \Q^-_{a,r}$,
    \[
        \lVert(X,t)-(Y,s)\rVert
        \geq
        \lVert(X,t)\rVert-\lVert(Y,s)\rVert
        \geq 
        M_0r-r 
        \geq 
        \frac{M_0}{2} r.
    \]
    Moreover, by Lemma~\ref{lem:capLtoM} and \eqref{eq:aronson_easy}, we have
    \[
        \Gamma_L(X,t;Y,s)
        \leq 
        C \Gamma_{\partial_t-M_2\Delta}(X,t;Y,s)
        \leq 
        C_5 \lVert(X,t)-(Y,s)\rVert^{-n}.
    \]
    Putting together both estimates, we have shown that
    \begin{equation} \label{eq:potential_above_bourgain}
        \Gamma_L\mu(X,t)
        \leq 
        C_5 \left(\frac{M_0}{2}\right)^{-n} r^{-n} \mu(K), 
        \qquad (X,t)\in\R^{n+1}\setminus \Q_{M_0r}(\mbf{0}).
    \end{equation}
    
    Next, choose $M_0$ large enough so that $C_5 \left(\frac{M_0}{2}\right)^{-n} \leq \frac{c_4}{2}$, and define 
    \[
	    u
	    :=
	    \Gamma_L\mu - C_5 \left(\frac{M_0}{2}\right)^{-n} r^{-n} \mu(K).
    \]
    Then, the following hold:
    \begin{enumerate}
        \item $u$ is continuous in $\R^{n+1}$ and $Lu=0$ in $\Omega$ (as for single-layer potentials),
        \item $u\leq 1$ on $\R^{n+1}$ (since $\lVert \Gamma_L\mu\rVert_\infty\leq 1$ by definition of capacity),
        \item $u\leq 0$ on $\R^{n+1}\setminus \Q_{M_0r}(\mbf{0})$ (by \eqref{eq:potential_above_bourgain}), and also on $\T_{\leq -r^2}$ (because $\Gamma_L\mu$ vanishes there since $\text{spt}(\mu) \subseteq K \subseteq \T_{> -r^2}$), and
        \item $u\geq\frac{c_4}{2} r^{-n} \mu(K)$ on $\Q^+_{a,r}$ (by \eqref{eq:potential_below_bourgain} and our choice of $M_0$).
    \end{enumerate}

    Now set $F := \Q_{M_0r}(\mbf{0}) \, \cap \, \T_{> -r^2} \, \cap \, \partial_e\Omega$. Then, we have $u(X,t)\leq\omega_L^{X,t}(F)$ for $(X, t) \in \partial_e\Omega$ by (2) and (3). Thus, $u(X,t)\leq\omega_L^{X,t}(F)$ for $(X,t)\in\Omega$ by the maximum principle (recall (1)). Then, noting that $\textrm{spt}(\omega_L^{X,t})\subseteq\partial_e\Omega\cap \T_{< r^2}$ for $(X,t)\in \Q^+_{a,r}$, we have
    \begin{multline*}
        \omega_L^{X,t}\left(Q_{M_0r}(0) \times (-r^2,r^2) \right) 
        =
        \omega_L^{X,t}(F)
        \geq 
        u(X,t)
        \\ \geq 
        \frac{c_4}{2} r^{-n} \mu(K)
        \geq 
        \frac{c_4}{4} r^{-n} \mathrm{Cap}_L(K), 
        \qquad 
        \text{if } (X,t)\in \Q^+_{a,r} \cap \Omega,
    \end{multline*}
    after using (4), which completes the proof.
\end{proof}

We are now in the position to prove Bourgain's estimate \eqref{maineqn:BourgainEstimate} from Theorem~\ref{mainthm:Bourgain}: the capacitary estimate shows its full power far from the bottom boundary because of the TBCDC assumption, and close to the bottom boundary we can run a simple ad-hoc geometrical argument.

\begin{proof}[\textbf{Proof of Theorem \ref{mainthm:Bourgain}, estimate (\ref{maineqn:BourgainEstimate})}]
    Fix $a\in(0,1)$ from the TBCDC, and $M_0$ from Lemma~\ref{thm:capacityestimate}. Fix also $c > 0$ (from the statement of Theorem~\ref{mainthm:Bourgain}) small enough, depending on these parameters. Let us split the proof into multiple cases, depending on the position of our surface ball $\Q_r(x_0, t_0)$ with respect to the bottom boundary.
    

    \textbf{Case 1: $(x_0,t_0)\in\Sigma$ and $0<r<M_0\sqrt{t_0-T_\mathrm{min}}/4$} (far from the bottom boundary).
    In this case, using Lemmas~\ref{thm:capacityestimate} and \ref{lem:cap_cylinder}, and the TBCDC, we obtain
    \begin{multline*}
        \omega_L^{X,t}(\Q_r(x_0,t_0))
        \geq
        \omega_L^{X,t}\bigg(Q_r(x_0) \times \Big(t_0 - \Big(\frac{r}{M_0}\Big)^2, t_0 + \Big(\frac{r}{M_0}\Big)^2 \Big) \bigg)\\
        \gtrsim 
        \frac{\mathrm{Cap}_L\Big( \big(\overline{Q_{r/M_0}(x_0)} \times \big[t_0-(\frac{r}{M_0})^2,t_0-(a\frac{r}{M_0})^2 \big] \big) \, \cap \, \Omega^c \Big)}{r^n}
        \gtrsim 
        1.
    \end{multline*}
	The application of Lemma~\ref{thm:capacityestimate} is justified if $\gamma < \frac{a}{\sqrt{2}M_0}$ because in such case 
	$(X,t)
	\in 
	\Q_{\frac{a}{\sqrt{2}M_0}r}(x_0,t_0) \cap \Omega
	\subseteq
	Q_{\frac{r}{M_0}}(x_0) \times \big(t_0-(a\frac{r}{M_0})^2/2,t_0+(\frac{r}{M_0})^2 \big)
	\cap
	\Omega$. The use of the TBCDC is also legitimate because of the restriction imposed on $r$ in Case 1.
	

    \textbf{Case 2: $t_0=T_\mathrm{min}$ and $r>0$} (at the bottom boundary). In this case, everything lying to the past from the center is clearly in $\Omega^c$, so it holds
    \begin{multline*} 
    \mathrm{Cap}_L\bigg( \Big(\overline{Q_{r/M_0}(x_0)} \times \Big[t_0-\big(\frac{r}{M_0}\big)^2,t_0 - \big(a\frac{r}{M_0}\big)^2 \Big] \Big) \, \cap \, \Omega^c \bigg)
    \\ =
    \mathrm{Cap}_L \Big(\overline{Q_{r/M_0}(x_0)} \times \Big[t_0-\big(\frac{r}{M_0}\big)^2,t_0 - \big(a\frac{r}{M_0}\big)^2 \Big] \Big)
    \gtrsim 
    r^n,
    \end{multline*}
	using Lemma~\ref{lem:cap_cylinder}. This allows us to mimic the computations in Case 1, without having needed to use the TBCDC assumption.
    

    \textbf{Case 3: $(x_0,t_0)\in\Sigma$ and $M_0\sqrt{t_0-T_\mathrm{min}}/4<r<M_0\sqrt{t_0-T_\mathrm{min}}/a$} (\textit{slightly} intersecting the bottom boundary). In this case, we can forget about a portion of the cube and still obtain the estimate. Concretely, our restriction on $r$ yields $\frac{a}{4} r<\frac{M_0\sqrt{t_0-T_\mathrm{min}}}{4}$, so we can apply the reasoning in Case 1 to obtain (recalling also that $a \in (0, 1)$) 
    \[
	    \omega_L^{X,t}(\Q_r(x_0,t_0))
	    \geq
	    \omega_L^{X,t}(\Q_{\frac{a}{4}r}(x_0,t_0))
	    \gtrsim 
	    1
    \]
    for $(X, t) \in \Q_{\gamma r}(x_0, t_0)$ as soon as we choose $\gamma \leq \frac{a}{\sqrt{2}M_0} \frac{a}{4} = \frac{a^2}{4M_0\sqrt{2n}}$.
    

    \textbf{Case 4: $(x_0,t_0)\in\Sigma$ and $r>M_0\sqrt{t_0-T_\mathrm{min}}/a$} (largely intersecting the bottom boundary). In this case, a lot of the mass lies to the past of the bottom boundary. Concretely, it holds $\overline{Q_{r/M_0}(x_0)} \times \big(t_0-(\frac{r}{M_0})^2,t_0-(a\frac{r}{M_0})^2 \big) \subseteq \Omega^c$ as in Case 2 (because $t_0-(ar/M_0)^2<T_\mathrm{min}$), so the very same reasoning of Case 2 yields the desired estimate.
\end{proof}

\begin{remark}
\label{rmk:M2TBCDCimpliesBourgain}
    In light of Lemma~\ref{lem:capLtoM} (concretely \eqref{eq:comparison_capacities}), if instead of assuming the TBCDC for $L$, we assume the TBCDC for $\partial_t-M_2\Delta$, the conclusion of Theorem~\ref{mainthm:Bourgain} is still true for $\omega_L$. Of course, assuming the TBCDC for $\partial_t-M_2\Delta$ is stronger than assuming the TBCDC for $L$, but it has the upshot that it is a condition that depends only on the ellipticity of the coefficient matrix (recall that $M_2$ depends only on $n$ and $\lambda$) and not potentially on the precise values of the entries of the matrix $A$ defining $L$.
\end{remark}

\subsection{Proof of the Hölder decay up to the boundary (\ref{maineqn:HolderEstimate})}

Given that the non-degeneracy estimate \eqref{maineqn:BourgainEstimate} for the parabolic measure holds, it is standard that some Hölder behavior takes place close to the boundary of our domain. Let us in any case include a general statement and a short proof for completeness.

\begin{lemma}
\label{lem:iteration_holder}
    Let $\x_0\in\partial_e\Omega\setminus\{\infty\}$ and $R > 0$. Assume that there exist $\gamma, \eta > 0$ such that for every $\x\in \partial_e\Omega$ and $r > 0$ satisfying $\Q_r(\x) \subseteq \Q_{3R}(\x_0)$ the following holds:
    \begin{quote}
        If $v$ is any function satisfying (i) $v \geq 0$ in $\Q_r(\x) \cap \Omega$, (ii) $Lv=0$ weakly in $\Q_r(\x) \cap \Omega$, and (iii) $v = 1$ continuously on $\Q_r(\x)\cap\partial_e\Omega$, then
        \[
            v(\Y) \geq \eta, \qquad \forall \, \Y\in \Q_{\gamma r}(\x) \cap \Omega.
        \]
    \end{quote}
    Then, there exist $C, \alpha_H > 0$, depending only on $\eta, \gamma$ and $n$, such that the following holds: 
    \begin{quote}
        If $u$ is any function satisfying (i) $u \geq 0$ in $\Q_{3R}(\x_0) \cap \Omega$, (ii) $Lu=0$ weakly in $\Q_{3R}(\x_0) \cap \Omega$, and (iii) $u = 0$ continuously on $\Q_{3R}(\x_0)\cap\partial_e\Omega$, then 
        \[
            u(\Y) \leq C \left( \frac{\delta(\Y)}{R} \right)^{\alpha_H} \sup_{\Q_{3R}(\x_0)\cap\Omega} u, \qquad \forall \, \Y \in \Q_R(\x_0) \cap \Omega.
        \]
    \end{quote}
\end{lemma}

\begin{proof}
    The proof follows by a well-known iteration argument. Fix $\Y \in \Q_R(\x_0) \cap \Omega$. Find $\hat{\y} \in \partial_e\Omega$ satisfying $\norm{\Y-\hat{\y}} = \delta(\Y)$. It is easy to see that $\hat{\y}\in \Q_{2R}(\x_0)$, so $\Q_R(\hat{\y}) \subseteq \Q_{3R}(\x_0)$, which will justify the application of the assumption repeatedly from now on. 

    Upon dividing by $\sup_{\Q_{3R}(\x_0)} u$, we may assume that $0 \leq u \leq 1$ in $\Q_{3R}(\x_0) \cap \Omega$.
    Define $w_1 := 1-u$, which satisfies (recalling our hypotheses on $u$) $w_1 \geq 0$ in $\Q_R(\hat{\y})$, $Lw_1 = 0$ in $\Q_R(\hat{\y})$, and also $w_1 = 1$ continuously on $\Q_R(\hat{\y})\cap\partial_e\Omega$. Then, by our assumption, $w_1(\Z) \geq \eta$ for every $\Z\in \Q_{\gamma R}(\hat{\y})$. That is, $u(\Z) \leq 1-\eta$ for every $\Z \in \Q_{\gamma R}(\hat{\y})$.

    Defining $w_2 := 1 - \frac{u}{1-\eta}$, we just showed that $w_2 \geq 0$ in $\Q_{\gamma R}(\hat{\y})$. Since it also satisfies $Lw_2 = 0$ in $\Q_R(\hat{\y})$, and $w_2 = 1$ continuously on $\Q_R(\hat{\y})\cap\partial_e\Omega$, applying again our assumption we infer that $w_2(\Z) \geq \eta$ for every $\Z \in \Q_{\gamma^2 R}(\hat{\y})$, that is, $u(\Z) \leq (1-\eta)^2$ for every $\Z \in \Q_{\gamma^2 R}(\hat{\y})$. A simple iteration yields $u(\Z) \leq (1-\eta)^j$ for any $\Z \in \Q_{\gamma^j R}(\hat{\y})$ and $j \geq 0$.

    Now, choose $j$ so that $\delta(\Y) \in [\gamma^{j+1}R, \gamma^jR)$. Simple manipulations yield 
    \[
    u(\Y)
    \leq 
    (1-\eta)^j
    \leq 
    (1-\eta)^{\frac{\log ( \delta(\Y)/R ) }{2\log \gamma}}
    =
    \Big( \frac{\delta(\Y)}{R} \Big)^{\frac{\log(1-\eta)}{2\log \gamma}}, 
    \]
    which is the desired result, with $\alpha_H := \frac{\log(1-\eta)}{2\log \gamma} > 0$ since $0 < \gamma, \eta < 1$.
\end{proof}
    
\begin{remark}
    One can check from the proof that one may instead assume that the functions $v$ in the assumption are supersolutions, and $u$ in the conclusion are subsolutions.
\end{remark}

\begin{proof}[\textbf{Proof of Theorem \ref{mainthm:Bourgain}, estimate \eqref{maineqn:HolderEstimate}}]
    Under the assumptions of Theorem \ref{mainthm:Bourgain}, \eqref{maineqn:BourgainEstimate} implies that the assumptions of Lemma \ref{lem:iteration_holder} hold, so Lemma \ref{lem:iteration_holder} directly yields \eqref{maineqn:HolderEstimate}.
\end{proof}


\section{Proof of Theorem \ref{mainthm:ExistenceOfHolderSolutions}, solvability of the H\"older Dirichlet problem} \label{sec:holder}

In this section, we will prove Theorem~\ref{mainthm:ExistenceOfHolderSolutions}, that is, that the Dirichlet problem in Hölder spaces is well-posed. Our proof will follow closely the elliptic one in \cite{CHMPZ}, but there will also be differences.
Uniqueness follows easily because in Theorem~\ref{mainthm:ExistenceOfHolderSolutions} we assume the existence of the parabolic measure (check Definition~\ref{def:parabolic_measure}). In turn, existence and the quantitative estimates for solutions with respect to boundary data \eqref{3.5.2}, measured in the same Hölder space, require a bit more work. The key tools will be the interior Hölder continuity granted by Lemma~\ref{lem:interiorHolder}, and the behavior at the boundary from Theorem~\ref{mainthm:Bourgain}. Let us use the latter to obtain strong estimates for tails, as in \cite[Lemma 2.17]{CHMPZ}.

\begin{lemma}\label{CHMPZ 2.17}
	Let $L$ be a parabolic operator as in Definition~\ref{def:operator}, and let $\Omega\subseteq\mathbb{R}^{n+1}$ be an open set satisfying the TBCDC for $L$. Assume also that the parabolic measure $\omega_L$ exists for $L$ on $\Omega$. Then, for any $\mbf{x_0}\in\partial_e\Omega$ and $r>0$,
	\begin{equation}
		\label{CHMPZ 2.17.1}
		\omega_L^\X(\partial_e\Omega\,\setminus\,\Q_{4r}(\mbf{x_0}))
		\lesssim 
		\left(\frac{\lVert\mbf{X}-\mbf{x_0}\rVert}{r}\right)^{\alpha_H},\quad\mbf{X}\in\Q_r(\mbf{x_0})\cap\Omega,
	\end{equation}
	where $\alpha_H\in(0,1)$ is the constant from Theorem \ref{mainthm:Bourgain}. Furthermore, if $0<\beta<\alpha_H$, then
	\begin{equation}
		\label{CHMPZ 2.17.3}
		\int_{\partial_e\Omega}\lVert\mbf{y}-\mbf{x_0}\rVert^\beta\,d\omega_L^\mbf{X}(\mbf{y})
		\lesssim 
		\delta(\mbf{X})^\beta
	\end{equation}
	for all $\mbf{X}\in\Omega$ and $\mbf{x_0}\in\partial_e\Omega$ such that $\lVert\mbf{X}-\mbf{x_0}\rVert\leq 3\delta(\mbf{X})$. The implicit constants in all the inequalities depend only on $n, \lambda$ and the TBCDC constants.
\end{lemma}

\begin{proof}
	Given $j\gg1$, choose $\phi_j\in C_c^\infty(\mathbb{R}^{n+1})$ such that $\phi_j\equiv 1$ in $\Q_{2^jr}(\mbf{x_0})\setminus \Q_{4r}(\mbf{x_0})$, $\phi_j\equiv 0$ in $\Q_{2r}(\mbf{x_0})\cup(\partial\Omega\setminus \Q_{2^{j+1}r}(\mbf{x_0}))$, and $0\leq\phi_j\leq 1$ everywhere. Hence,
	\[
	\omega_L^{\mbf{X}}(\Q_{2^jr}(\mbf{x_0})\setminus\Q_{4r}\mbf{(x_0}))
	\leq
	\int_{\partial_e\Omega}\phi_j\,d\omega_L^{\mbf{X}}=:v_j(\mbf{X}), 
	\quad \text{for all $\mbf{X}\in\Omega$.}
	\]
	Moreover, applying \eqref{maineqn:HolderEstimate} ($v_j$ is continuous up to $\pom$ because so is $\phi_j$, and the Wiener criterion holds by Lemma~\ref{thm:TBCDCimpliesWiener}) and noting that $0\leq v_j\leq 1$ since $0\leq \phi_j\leq 1$, it holds
	\[
	v_j(\mbf{X})
	\leq
	C\left(\frac{\delta(\mbf{X})}{r}\right)^{\alpha_H}\sup_{\mbf{Y}\in \Q_{3r}(\mbf{x_0})\cap\Omega}v_j(\mbf{Y})
	\leq
	C\left(\frac{\delta(\mbf{X})}{r}\right)^{\alpha_H}, 
	\quad 
	\text{for all $\mbf{X}\in \Q_r(\mbf{x_0})\cap\Omega$.}
	\]
	Thus, letting $j\to\infty$, \eqref{CHMPZ 2.17.1} follows by monotone convergence.
	
	Now fix $\X, \x_0$ as in \eqref{CHMPZ 2.17.3}. Writing $\widetilde{\Q}_k:=\Q_{2^k\delta(\mbf{X})}(\mbf{x_0}) \cap \partial_e \Omega$ for $k\in\N$, we have by \eqref{CHMPZ 2.17.1}
	\begin{multline*}
		\int_{\partial_e\Omega}\lVert\mbf{y}-\mbf{x_0}\rVert^\beta\,d\omega_L^\mbf{X}(\mbf{y})
		=
		\int_{\widetilde{\Q}_4}\lVert\mbf{y}-\mbf{x_0}\rVert^\beta\,d\omega_L^\mbf{X}(\mbf{y})
		+\sum_{k=4}^\infty \int_{\widetilde{\Q}_{k+1}\setminus\widetilde{\Q}_k} \lVert\mbf{y}-\mbf{x_0}\rVert^\beta\,d\omega_L^\mbf{X}(\mbf{y})
		\\ \lesssim
		\omega_L^\mbf{X}(\widetilde{\Q}_4)\delta(\mbf{X})^\beta
		+\sum_{k=4}^\infty (2^k\delta(\mbf{X}))^\beta \left(\frac{\delta(\mbf{X})}{2^k \delta(\mbf{X})}\right)^{\alpha_H}
		\lesssim
		\delta(\mbf{X})^\beta \Big( 1+ \sum_{k=4}^\infty 2^{k(\beta-\alpha_H)}\Big)
		\lesssim
		\delta(\mbf{X})^\beta,
	\end{multline*}
	where we have also used that $\omega_L^\mbf{X}(\widetilde{\Q}_4)\leq\omega_L^{\mbf{X}}(\partial_e\Omega)\leq 1$ and $\beta<\alpha_H$.
\end{proof}

With this estimate in hand, we can continue with the main proof of the section.

\begin{proof}[\textbf{Proof of Theorem \ref{mainthm:ExistenceOfHolderSolutions}}]
	We can set $\alpha := \min\{\alpha_H,\alpha_N\}$, with $\alpha_H$ from Lemma \ref{CHMPZ 2.17}, and $\alpha_N$ from Lemma \ref{lem:interiorHolder}. Then, if $0<\beta<\alpha$, for all $f\in \dot{C}^\beta(\partial_e\Omega)$, the function in the statement of the theorem, namely
	\[
	u(\mbf{X}) = \int_{\partial_e\Omega}f\,d\omega_L^{\mbf{X}}, 
	\qquad \X \in \Omega,
	\]
	is the unique solution to the continuous Dirichlet problem by Definition~\ref{def:parabolic_measure}.\footnote{
        The reader may note that the integral defining $u$ is always absolutely convergent, even when $\partial_e \Omega$ is unbounded. Indeed, in such case, $\infty \in \partial_e \Omega$, which implies that $f$ is actually bounded because $f \in \dot{C}^\beta(\partial_e\Omega)$.
    }
	Thus, it remains to show that $u\in \dot{C}^\beta(\Omega\cup\partial_n\Omega)$ with the estimates \eqref{3.5.2}. To do so, we will mimic the proof of \cite[Theorem 3.1]{CHMPZ}. Forgetting about the trivial case of constant boundary values (solutions are constant in that case), let us assume that $\lVert f\rVert_{\dot{C}^\beta(\partial_e\Omega)}=1$.

	Let $\mbf{X},\mbf{Y}\in\Omega$, and assume that $\delta(\mbf{X}) \leq \delta(\mbf{Y})$. Find $\mbf{\hat{x}},\mbf{\hat{y}}\in\partial_e\Omega$ such that $\lVert\mbf{X}-\mbf{\hat{x}}\rVert=\delta(\mbf{X})$ and $\lVert\mbf{Y}-\mbf{\hat{y}}\rVert=\delta(\mbf{Y})$. Now, we split into cases.
	
	\textbf{Case 1:} $\lVert\mbf{X}-\mbf{Y}\rVert<\delta(\mbf{Y})/4$. In this case, $\mbf{X}\in \Q_{\delta(\mbf{Y})/4}(\mbf{Y})$. Then, the interior H\"older continuity estimate (Lemma \ref{lem:interiorHolder}) implies that
	\begin{equation*}
		|u(\mbf{X})-u(\mbf{Y})|
		\! = \!
		|(u(\mbf{X})-f(\mbf{\hat{y}}))-(u(\mbf{Y})-f(\mbf{\hat{y}}))|
		\lesssim \!
		\left(\frac{\lVert\mbf{X}-\mbf{Y}\rVert}{\delta(\mbf{Y})}\right)^{\alpha_N}
		\!\!
		\norm{u(\cdot) - f(\mbf{\hat{y}})}_{L^\infty(\Q_{\frac{\delta(\mbf{Y})}{2}}(\Y))}.
	\end{equation*}
	Now, if $\Z \in \Q_{\delta(\mbf{Y})/2}(\Y)$, then 
	\[
	\norm{\Z - \mbf{\hat{y}}}
	\leq 
	\norm{\Z - \Y} + \norm{\Y - \mbf{\hat{y}}}
	\leq 
	\frac32 \delta(\Y)
	\leq 
	3 \delta(\Z), 
	\]
	which allows us to use Lemma~\ref{CHMPZ 2.17} to further estimate (recalling that $\omega_L$ is a probability)
	\begin{equation*}
		|u(\mbf{Z})-f(\mbf{\hat{y}})|
		\leq
		\int_{\partial_e\Omega} |f(\mbf{y})-f(\mbf{\hat{y}})|\,d\omega_L^{\mbf{Z}}(\y)
		\leq
		\int_{\partial_e\Omega}\lVert\mbf{y}-\mbf{\hat{y}}\rVert^\beta\,d\omega_L^{\mbf{Z}}(\y)
		\lesssim
		\delta(\mbf{Z})^\beta
		\lesssim 
		\delta(\Y)^\beta.
	\end{equation*}
	All these estimates finally give (recalling that $\lVert\mbf{X}-\mbf{Y}\rVert<\delta(\mbf{Y})$ and $\beta<\alpha\leq\alpha_N$)
	\begin{equation*}
		|u(\mbf{X})-u(\mbf{Y})|
		\lesssim
		\left(\frac{\lVert\mbf{X}-\mbf{Y}\rVert}{\delta(\mbf{Y})}\right)^{\alpha_N}\delta(\mbf{Y})^\beta
		\leq
		\left(\frac{\lVert\mbf{X}-\mbf{Y}\rVert}{\delta(\mbf{Y})}\right)^{\beta}\delta(\mbf{Y})^\beta
		=\lVert\mbf{X}-\mbf{Y}\rVert^\beta.
	\end{equation*}
	
	\textbf{Case 2:} $\delta(\mbf{Y})\leq 4\lVert\mbf{X}-\mbf{Y}\rVert$. In this case, we have (recalling that $\delta(\X) \leq \delta(\Y)$)
	\begin{equation*}
		\lVert\mbf{\hat{x}}-\mbf{\hat{y}}\rVert
		\leq
		\lVert\mbf{\hat{x}}-\mbf{X}\rVert+\lVert\mbf{X}-\mbf{Y}\rVert+\lVert\mbf{Y}-\mbf{\hat{y}}\rVert
		=
		\delta(\mbf{X})+\lVert\mbf{X}-\mbf{Y}\rVert+\delta(\mbf{Y})
		\leq 
		9\lVert\mbf{X}-\mbf{Y}\rVert.
	\end{equation*}
	Therefore, using Lemma~\ref{CHMPZ 2.17} (recalling that $\omega_L$ is a probability), it follows that
	\begin{equation*}
		|u(\mbf{X})-f(\mbf{\hat{x}})|
		\leq
		\int_{\partial_e\Omega}|f(\mbf{z})-f(\mbf{\hat{x}})|\,d\omega_L^{\mbf{X}}(\mbf{z})
		\leq
		\int_{\partial_e\Omega}\lVert\mbf{z}-\mbf{\hat{x}}\rVert^\beta\,d\omega_L^{\mbf{X}}(\mbf{z})
		\lesssim 
		\delta(\mbf{X})^\beta
		\lesssim
		\lVert\mbf{X}-\mbf{Y}\rVert^\beta,
	\end{equation*}
	and a similar computation shows $
	|u(\mbf{Y})-f(\mbf{\hat{y}})|\lesssim \lVert\mbf{X}-\mbf{Y}\rVert^\beta$.
	Therefore, we obtain
	\begin{multline*}
		|u(\mbf{X})-u(\mbf{Y})|
		\leq
		|u(\mbf{X})-f(\mbf{\hat{x}})|+|f(\mbf{\hat{x}})-f(\mbf{\hat{y}})|+|f(\mbf{\hat{y}})-u(\mbf{Y})|\\
		\lesssim \lVert\mbf{X}-\mbf{Y}\rVert^\beta+\lVert\mbf{\hat{x}}-\mbf{\hat{y}}\rVert^\beta
		\lesssim 
		\lVert\mbf{X}-\mbf{Y}\rVert^\beta.
	\end{multline*}
	
	Thus, in both Cases 1 and 2, we have shown that $|u(\mbf{X})-u(\mbf{Y})|\lesssim \lVert\mbf{X}-\mbf{Y}\rVert^\beta$, where the implicit constant depends only on $n,\lambda$ and the TBCDC constants, which shows the upper bound in \eqref{3.5.2}. Concretely, since $u$ is uniformly continuous, we can extend it to the boundary and obtain $u \in \dot{C}^\beta(\Omega \cup \partial_n \Omega)$, and the lower bound in \eqref{3.5.2} follows trivially since $u = f$ over $\partial_e \Omega$ (because $u$ was already known to be a solution to the continuous Dirichlet problem with datum $f$ from Lemma~\ref{mainthm:ExistenceOfParabolicMeasure}). 
\end{proof}

The result that we just showed is easily generalizable to more general spaces of functions, as was done in \cite{CHMPZ}. 

\begin{definition}[Growth functions in $\mathcal{G}_\beta$]
	A function $\varphi:(0,+\infty)\to(0,+\infty)$ is said to belong to the $\mathcal{G}_\beta$ class, for $\beta > 0$, if
	\begin{enumerate}
		\item $\varphi(t)\to0$ as $t\to0^+$, and
		\item there exists some $C_\varphi$ such that
		\[
		\int_0^t \varphi(s)\frac{ds}{s}+t^\beta\int_t^\infty \frac{\varphi(s)}{s^\beta}\frac{ds}{s}\leq C_\varphi\varphi(t),\qquad\textrm{for all }t>0.
		\]
	\end{enumerate}
\end{definition}

Then, if we define the \textit{$\varphi$-H\"older space} on a set $E\subseteq\R^{n+1}$ by 
\[
\dot{C}^\varphi(E)
:=
\bigg\{
u:E\to\R:\lVert u\rVert_{\dot{C}^\varphi(E)}:=\sup_{\substack{\mbf{X},\mbf{Y}\in E \\ \mbf{X}\neq\mbf{Y}}}\frac{|u(\mbf{X})-u(\mbf{Y})|}{\varphi(\lVert\mbf{X}-\mbf{Y}\rVert)}<\infty
\bigg\},
\]
we can obtain the following result, which is analogous to \cite[Theorem 3.1]{CHMPZ}.

\begin{theorem}
	Let $L$ be a parabolic operator as in Definition~\ref{def:operator}, and let $\Omega\subseteq\R^{n+1}$ be an open set satisfying the TBCDC for $L$. Assume also that the parabolic measure $\omega_L$ exists for $L$ on $\Omega$. Then, there exists $\alpha\in(0,1)$ such that for all $\beta\in(0,\alpha]$ and $\varphi\in\mathcal{G}_\beta$, the $\dot{C}^\varphi$-Dirichlet problem
	\[
	\begin{cases}
		u\in \dot{C}^\varphi(\Omega\cup\partial_n\Omega)\cap\Wloc(\Omega),\\
		Lu=0\textrm{ in the weak sense in $\Omega$},\\
		u|_{\partial_e\Omega}=f\in \dot{C}^\varphi(\partial_e\Omega).
	\end{cases}
	\] 
	is well-posed. More specifically, there is a unique solution given by
	\[
	u(\mbf{X})=\int_{\partial_e\Omega}f(\mbf{y})\,d\omega_L^{\mbf{X}}(\mbf{y}),\qquad\mbf{X}\in\Omega,
	\]
	that satisfies
	\[
	\lVert f\rVert_{\dot{C}^\varphi(\partial_e\Omega)}
	\leq
	\lVert u\rVert_{\dot{C}^\varphi(\Omega)}
	\lesssim 
	\lVert f\rVert_{\dot{C}^\varphi(\partial_e\Omega)}.
	\]
	The implicit constant and $\alpha$ only depend on $n, \lambda, C_\varphi$ and TBCDC constants.
\end{theorem}

We omit the proof: it is analogous to the proof of Theorem~\ref{mainthm:ExistenceOfHolderSolutions}, similarly to \cite[Theorem 3.1, Part 2]{CHMPZ}. On top of that, one uses facts about growth functions that can be found in \cite[Lemma 2.15]{CHMPZ}.

\begin{remark}
    If instead of following the convention that $\infty \in \partial_e \Omega$ whenever $\Omega$ is unbounded, one would like to follow the spirit of the original proofs in \cite{CHMPZ} by not imposing any boundary value at infinity, the results would be equally true. Indeed, the reader can check that the proofs from \cite{CHMPZ} still work almost verbatim. For existence of solutions, one uses $\omega_L$ restricted to $\partial_e \Omega \setminus \{\infty\}$, and shows that solutions are well-defined as in \cite[Section 3.1, Step 1]{CHMPZ} using Harnack's inequality from Lemma~\ref{lem:harnack}. In turn, uniqueness of solutions follows \cite[Section 3.2]{CHMPZ} word by word.
\end{remark}


\section{Proof of Theorem \ref{mainthm:ExistenceOfParabolicMeasure}, existence of the parabolic measure} \label{sec:existence_parabolic_measure}

As commented in the introduction of Section~\ref{sec:bourgain_proof}, we showed Theorems~\ref{mainthm:Bourgain} and \ref{mainthm:ExistenceOfHolderSolutions} under the \textit{a priori} assumption that the parabolic measure for $L$ exists. In this section, we prove that this is actually the case for any parabolic operator $L$ with merely bounded coefficients as in Definition~\ref{def:operator}. We will actually construct the parabolic measure for $L$ as the limit of the ones associated to certain regularized versions of $L$, for which the parabolic measure is known to exist by Wiener's criterion (see Remark~\ref{rmk:ExistenceOfParabolicMeasureForCinfty}). In fact, we will rely on Theorems~\ref{mainthm:Bourgain} and \ref{mainthm:ExistenceOfHolderSolutions} applied to these regularized operators.

\subsection{Existence of the parabolic measure for bounded domains} 

Let us first consider the case when $\Omega$ is bounded, which is simpler to understand. Recall that if we are given any point $(X, t) \in \Omega$, only what happens in $\mathcal{T}_{< t}$ influences PDEs at $(X, t)$. Concretely, if we want to determine whether $\omega_L^{X, t}$ exists, we could modify $\Omega$ to the future of $(X, t)$ without any side effects.

\begin{theorem}
\label{thm:ExistenceOfParabolicMeasureOnBoundedDomains}
    Let $L=\partial_t-\div A\nabla$ be a parabolic operator as in Definition~\ref{def:operator}, and $\Omega\subseteq\R^{n+1}$ be a bounded open set that satisfies the TBCDC for the operator $\partial_t-M_2\Delta$, where $M_2$ is the constant from Lemma~\ref{lem:capLtoM}. Then, for each $(X,t)\in\Omega$, there exists a unique positive Radon measure $\omega_L^{X,t}$ supported on $\partial_e\Omega$ so that for all $f\in C(\partial_e\Omega)$, the function 
    \[
        u(X,t)=\int_{\partial_e\Omega}f\,d\omega_L^{X,t}, 
        \qquad 
        (X, t) \in \Omega,
    \]
    solves the continuous Dirichlet problem $Lu = 0$ in $\Omega$, and $u = f$ on $\partial_e \Omega$ continuously.
\end{theorem}
  
\begin{proof}
    \textbf{Part I: when $f$ is Hölder.} First, consider $f\in \dot{C}^\beta(\partial_e\Omega)$, where $\beta<\alpha$ for $\alpha$ in the statement of Theorem \ref{mainthm:ExistenceOfHolderSolutions}.
    
    \textbf{Step 1: approximation of $L$.} First, extend $A$ by the identity matrix outside of $\Omega$. Let $\eta$ be the standard mollifier, and define $A_k := \eta_k*A$ for $k \in \N$, where the convolution is done componentwise and $\eta_k(x) := k^n \eta(kx)$. Then, the matrices $A_k$ are $C^\infty$, and $A_k \to A$ in $L^2_{\loc}$ as $k \to \infty$. Since $\int \eta_k = 1$ and $A$ is elliptic and bounded, it is easy to check that
    \begin{equation*}
    	\lambda|\xi|^2
    	\leq
    	A_k(\X)\xi \cdot \xi,
    	\quad
    	\lVert A_k\rVert_{L^\infty(\R^n)}\leq\lambda^{-1},
    	\qquad 
    	\text{a.e. } \X\in\R^{n+1}, \; \forall \xi \in \Rn,
    \end{equation*}
	with the same constants as $A$. This means that the family of operators $L_k := \partial_t-\div A_k\nabla$ satisfy the ellipticity and boundedness conditions uniformly on $k$. 
	Thus, since the value of $M_2$ from Lemma~\ref{lem:capLtoM} only depends on $n$ and $\lambda$, we can consider a common $M_2$ that makes Lemma~\ref{lem:capLtoM} work for all $L_k$ simultaneously. Concretely, since we have assumed the TBCDC for $\partial_t-M_2\Delta$, \eqref{eq:comparison_capacities} implies that $\Omega$ satisfies the TBCDC for $L_k$ for all $k\in\N$. 
	
	Therefore, since the $A_k$ are smooth, Lemma~\ref{thm:TBCDCimpliesWiener} implies that the parabolic measure $\omega_{L_k}^\X$ exists for $L_k$ in $\Omega$. Hence, we may define
    \[
        u_k(\X) := \int_{\partial_e\Omega}f\,d\omega_{L_k}^\X, 
        \qquad \X \in \Omega,
    \]
    which, by Theorem \ref{mainthm:ExistenceOfHolderSolutions}, satisfies $u_k\in \dot{C}^\beta(\Omega\cup\partial_n\Omega)\cap W^{1,2}_\loc(\Omega)$ and $u_k|_{\partial_e\Omega}=f$, for all $k\in\N$. In fact, we can find extensions of the $u_k$, which we will by abuse of notation again call $u_k$, so that $u_k\in \dot{C}^\beta(\overline{\Omega})\cap W^{1,2}_\loc(\Omega)$.
    
    
    \textbf{Step 2: extracting a limit.} The $u_k$ are uniformly bounded in $\dot{C}^\beta(\overline{\Omega})$ by \eqref{3.5.2}, and also in $L^\infty$ by the maximum principle. Thus, by the Arzelà-Ascoli theorem, there exists $u \in \dot{C}^\beta(\overline{\Omega})$ such that $u_k \to u$ in $L^\infty(\overline{\Omega})$ as $k \to \infty$.


    \textbf{Step 3: $u\in W^{1,2}_\loc(\Omega)$.} Let $\Q$ be a parabolic cube such that $4\Q\subseteq\Omega$. Then, by Caccioppoli's inequality (Lemma~\ref{lem:caccioppoli}) and the maximum principle,
    \begin{align*}
        \sup_{k\in \N} \iint_\Q |\nabla u_k|^2
        \lesssim_\Q
        \sup_{k\in \N} \iint_{2\Q} |u_k|^2 
        \lesssim_\Q
        \sup_{k\in \N} \lVert u_k\rVert_\infty^2
        \leq
        \lVert f\rVert_\infty^2.
    \end{align*}
    Since this bound is uniform (the constants in Caccioppoli's inequality depend on ellipticity, which is uniformly bounded for the $L_k$), we get an uniform bound for $\norm{u_k}_{W^{1, 2}(\Q)}$, independent of $k$. Hence, by compactness, $u_k \to u_\Q$ and $\nabla u_k \rightharpoonup \nabla u_\Q$ in $L^2(\Q)$ for some $u_\Q \in W^{1, 2}(\Q)$ (up to a subsequence).    
    Concretely, up to a further subsequence, $u_k \to u_\Q$ a.e. in $\Q$, so actually $u_\Q = u$ a.e. in $\Q$ because of the uniform convergence obtained in the previous step. 
    
    Upon covering any compact set inside $\Omega$ by cubes $\Q$ as in the last paragraph (with a Vitali-like argument), we obtain $u \in W^{1, 2}_\loc(\Omega)$, with $\nabla u_k \rightharpoonup \nabla u$ in $L^2_\loc(\Omega)$.
    
    
    \textbf{Step 4: $u$ is a weak solution.} Let $\psi\in C_c^\infty(\Omega)$. Then, since $L_k u_k = 0$ in $\Omega$, it holds
    \begin{multline*}
        \left|\iint_\Omega -u\partial_t\psi+A\nabla u\cdot\nabla\psi\right|
        =\left|\iint_\Omega -(u-u_k)\partial_t\psi+(A\nabla u-A_k\nabla u_k)\cdot\nabla\psi\right|
        \\
        \leq
        \abs{ \iint_{\supp \psi} (u-u_k)\partial_t\psi }
        +\left|\iint_{\supp \psi} (A\nabla u-A_k\nabla u_k)\cdot\nabla\psi\right|
        =:
        \textrm{I}+\textrm{II}.
    \end{multline*}
	Clearly $\textrm{I} \to 0$ as $k \to \infty$ because $u_k \to u$ uniformly in $\supp \psi$. In turn,
    \begin{multline*}
        \textrm{II}
        \leq
        \left|\iint_{\supp \psi} \!\!\!\! A(\nabla u-\nabla u_k)\cdot\nabla\psi\right|
        +\iint_{\supp \psi} \!\!\!\! |(A-A_k)\nabla u_k\cdot\nabla\psi|
       	\\ \leq 
        \left|\iint_{\supp \psi} \!\!\!\! (\nabla u-\nabla u_k)\cdot A^T \nabla\psi\right|
        +
        \left(\iint_{\supp \psi} \!\!\!\! |A-A_k|^2 \right)^{1/2} \left(\iint_{\supp \psi}  \!\!\!\! |\nabla u_k|^2 \right)^{1/2} \norm{\nabla\psi}_\infty.
    \end{multline*}
	The first term converges to 0 because $\nabla u_k \rightharpoonup \nabla u$ in $L^2_\loc(\Omega)$, and so does the second because $A_k \to A$ in $L^2_{\loc}$ (using also the uniform bounds for $u_k$ in $W^{1, 2}_\loc(\Omega)$).
    Combining these estimates yields
    $
        \iint_\Omega -u\partial_t\psi+A\nabla u\cdot\nabla\psi=0,
    $
    which implies that $Lu=0$ in $\Omega$.
    

    \textbf{Part II: when $f$ is continuous.} Let us now consider the general case when $f\in C(\partial_e\Omega)$. We will follow similar steps to the ones above to construct a solution in this case.
    

    \textbf{Step 1: approximation of $f$.} Since $\partial_e\Omega$ is bounded, $f$ is uniformly continuous, so we can find $\{f_j\}\subseteq \dot{C}^\beta(\partial_e\Omega)$ such that $f_j\to f$ uniformly on $\partial_e\Omega$ (see \cite[Lemma 6.8]{Heinonen}). Then, for each $j\in\N$, we can use Part I to find a solution $u_j$ to
    \[
        \begin{cases}
            u_j\in \dot{C}^\beta(\overline{\Omega})\cap W^{1,2}_\loc(\Omega), \\
            Lu_j=0\textrm{ in the weak sense on }\Omega, \\
            u_j=f_j \text{ on } \partial_e\Omega.
        \end{cases}
    \]


    \textbf{Step 2: extracting a limit.} Note that for all $j,k\in\N$, $L(u_j-u_k)=0$ weakly in $\Omega$, with boundary data $f_j-f_k$, so we may apply the maximum principle to get
    \[
        \lVert u_j-u_k\rVert_\infty\leq\lVert f_j-f_k\rVert_\infty\underset{j,k\to\infty}{\longrightarrow}0.
    \]
    Thus, $\{u_j\}$ is a Cauchy 
    sequence in $C(\overline{\Omega})$, whence (up to a subsequence) $u_j \to u \in C(\overline{\Omega})$, uniformly. Concretely, we have $u = f$ on $\partial_e\Omega$ because $u_j \to u$ and $f_j \to f$ uniformly.
    

    \textbf{Step 3: $u\in W^{1,2}_\loc(\Omega)$.} 
    Similarly as in Part I, if $\Q$ is a cube with $4\Q\subseteq\Omega$, by Caccioppoli's inequality (Lemma~\ref{lem:caccioppoli}) and the maximum principle, we get
    \begin{equation*}
        \sup_{j\in \N} \iint_\Q |\nabla u_j|^2
        \lesssim_\Q
        \sup_{j\in \N} \iint_{2\Q} |u_j|^2 
        \lesssim_\Q
        \sup_{j\in \N} \lVert u_j\rVert_\infty^2
        \leq
        \sup_{j\in\N}\lVert f_j\rVert_\infty^2
        \leq
        \lVert f\rVert_\infty^2 + 1,
    \end{equation*}
	where in the last estimate we have used the uniform convergence $f_j \to f$ (and we possibly need to take a subsequence). This shows (after a covering argument) that the $u_j$ are uniformly bounded in $W^{1, 2}(\Q)$, so by compactness there exists a weak limit that must actually coincide with $u$ from Step 2 (we already explained this in Part I). In the end, we obtain that $u \in W^{1, 2}_\loc(\Omega)$, and $\nabla u_j \rightharpoonup \nabla u$ in $L^2_\loc(\Omega)$.
	

    \textbf{Step 4: $u$ is a weak solution.} Let $\psi\in C_c^\infty(\Omega)$. Since $Lu_j = 0$ in $\Omega$, we have
    \begin{multline*}
        \iint_\Omega -u\partial_t\psi+A\nabla u\cdot\nabla\psi
        =
        \iint_\Omega -(u-u_j)\partial_t\psi+A(\nabla u-\nabla u_j)\cdot\nabla\psi
        \\ =
        \iint_{\supp \psi} (u-u_j)\partial_t\psi
        + \iint_{\supp \psi} (\nabla u-\nabla u_j)\cdot A^T \nabla\psi, 
    \end{multline*}
	and both terms converge to 0 as $j \to \infty$ by the weak convergence $u_j \rightharpoonup u$ in $W^{1, 2}_{\loc}(\Omega)$. Thus, $       \iint_\Omega -u\partial_t\psi+A\nabla u\cdot\nabla\psi=0$, so $Lu=0$ in the weak sense on $\Omega$.
	

    \textbf{Part III: existence of parabolic measure.} In Part II we have found, for every $f \in C(\partial_e \Omega)$, a solution to $Lu= 0$ in $\Omega$ satisfying $u = f$ on $\partial_e \Omega$. Note that this construction is linear because the classical maximum principle grants uniqueness of solutions to $Lu=0$ in $\Omega$ with continuous boundary values. Thus, if we fix $\X \in \Omega$, the map $f \mapsto u(\X)$ is linear. This map is easily seen to be positive, too.\footnote{
		If $f \geq 0$, the construction in Part II, Step 1, can be taken so that $f_j \geq 0$ (this is in fact what is done in \cite[Lemma 6.8]{Heinonen}). Using these $f_j$ in Part I, we obtain $u_j \geq 0$. (Indeed, in Step 1 we clearly get $u_k \geq 0$ because $\omega_{L_k}$ is a positive measure, so taking the uniform limit in Step 2 yields a non-negative limit function.) Taking uniform limits in Part II, Step 2, yields $u \geq 0$.
	} Therefore, the existence of the parabolic measure $\omega_L^\X$ is now a consequence of the Riesz representation theorem. The fact that it is a probability measure follows easily because $\Omega$ is bounded, as explained in Remark~\ref{rem:probability_bounded}.
\end{proof}


\subsection{Existence of the parabolic measure for unbounded domains}

After having worked out the case of bounded domains, strongly taking advantage of Theorem~\ref{mainthm:ExistenceOfHolderSolutions} to grant the compactness necessary to run the approximation argument, let us pass to the unbounded setting. Recall that if $\Omega$ is unbounded, we consider that the point at infinity is part of $\partial_e \Omega$, so uniqueness should not be a major problem (compare this approach to the one in \cite{CHMPZ}, for instance). However, it is still not trivial to obtain global bounds for (candidates to) solutions, so we need to refine the arguments in Theorem~\ref{thm:ExistenceOfParabolicMeasureOnBoundedDomains} to be able to construct global solutions and, in that way, be able to construct the parabolic measure in unbounded domains.

Our approach will be based on approximating the unbounded set $\Omega$ by bounded truncations of it, to which we will apply Theorem~\ref{thm:ExistenceOfParabolicMeasureOnBoundedDomains}. Namely, we will consider
\[
    \Omega_R:=\Omega \cap \Q_R(0,0), 
    \qquad 
    R > 0,
\]
and we will work to get uniform bounds for solutions as $R \to \infty$.
We let $\Sigma_R$ be the quasi-lateral boundary of $\Omega_R$, and $\omega_R^{X,t}$ be the parabolic measure for $L$ on $\Omega_R$: let us first prove a small lemma to ensure that the latter exists, under the assumptions of Theorem \ref{thm:ExistenceOfParabolicMeasureOnBoundedDomains}.
\Bk%

\begin{lemma}
\label{lem:TBCDConOmegaR}
    Suppose $\Omega$ satisfies the TBCDC for $L$. Then, so does $\Omega_R$ for all $R>0$. Moreover, the TBCDC constants of $\Omega_R$ do not depend on $R$, but only on those of $\Omega$, and $n$ and $\lambda$.
\end{lemma}

\begin{remark}	
    We include a short proof of this lemma because one has to be careful about the time-directedness of the TBCDC.
\end{remark}

\begin{proof}	
    Let $(x_0,t_0)\in\Sigma_R$ and $0<r<\sqrt{t_0-T_\textrm{min}(\Omega_R)}/4$. We depict the possible situations in Figure~\ref{fig:restricted_boundary}. If $(x_0,t_0)\in\Sigma$ (resp. $\Sigma(\Q_R(\mbf{0}))$, that is, the quasi-lateral boundary of the set $\Q_R(\mbf{0})$), we may simply use the TBCDC estimate for $\Omega$ (resp. $\Q_R(\mbf{0})$, which satisfies the TBCDC with constant depending only on $n$ and $\lambda$) because $\Omega_R^c \supset \Omega^c$ (resp. $\Omega_R^c \supset \Q_R(\mbf{0})^c$). 
    These are actually the only cases to take into account because $\Sigma_R$ does not include initial and terminal faces, and it clearly holds $T_{\min}, T_{\min}(\Q_R) \leq T_{\min}(\Omega_R) \leq T_{\max}(\Omega_R) \leq T_{\max}, T_{\max}(\Q_R)$.
\end{proof}

\begin{figure}
	\centering 
	\scalebox{.8}{
	\begin{tikzpicture}[scale=1]
		\fill[Cyan!40, opacity=0.5] 
		(0, 3.058)
		.. controls (0, 3.058) and (0, 3.058) .. (0.188, 3.388)
		.. controls (0.376, 3.717) and (0.753, 4.375) .. (1.223, 4.422)
		.. controls (1.693, 4.47) and (2.258, 3.905) .. (2.775, 3.67)
		.. controls (3.293, 3.435) and (3.763, 3.529) .. (4.186, 3.811)
		.. controls (4.61, 4.093) and (4.986, 4.564) .. (5.315, 4.846)
		.. controls (5.644, 5.128) and (5.927, 5.222) .. (6.209, 5.269)
		.. controls (6.491, 5.316) and (6.773, 5.316) .. (7.103, 5.034)
		.. controls (7.432, 4.752) and (7.808, 4.187) .. (7.996, 3.905)
		.. controls (8.184, 3.623) and (8.184, 3.623) .. (8.184, 3.623)
		.. controls (8.184, 3.623) and (8.184, 3.623) .. (8.184, 4.375)
		.. controls (8.184, 5.128) and (8.184, 6.633) .. (8.184, 7.386)
		.. controls (8.184, 8.138) and (8.184, 8.138) .. (8.184, 8.138)
		.. controls (8.184, 8.138) and (8.184, 8.138) .. (6.82, 8.138)
		.. controls (5.456, 8.138) and (2.728, 8.138) .. (1.364, 8.138)
		.. controls (0, 8.138) and (0, 8.138) .. (0, 8.138)
		.. controls (0, 8.138) and (0, 8.138) .. (0, 8.138)
		.. controls (0, 8.138) and (0, 8.138) .. (0, 7.292)
		.. controls (0, 6.445) and (0, 4.752) .. (0, 3.905)
		.. controls (0, 3.058) and (0, 3.058) .. cycle;
		
		\draw[ultra thick, opacity=1] (0.004, 3.065) .. controls (0.292, 3.565) and (0.397, 3.725) .. (0.463, 3.823) .. controls (0.529, 3.92) and (0.555, 3.955) .. (0.572, 3.976) .. controls (0.588, 3.997) and (0.594, 4.004) .. (0.605, 4.016) .. controls (0.615, 4.029) and (0.631, 4.047) .. (0.658, 4.076) .. controls (0.685, 4.105) and (0.723, 4.144) .. (0.76, 4.178) .. controls (0.798, 4.213) and (0.834, 4.243) .. (0.877, 4.273) .. controls (0.919, 4.303) and (0.968, 4.333) .. (1.012, 4.355) .. controls (1.057, 4.377) and (1.097, 4.392) .. (1.145, 4.403) .. controls (1.193, 4.413) and (1.249, 4.419) .. (1.304, 4.418) .. controls (1.358, 4.416) and (1.412, 4.408) .. (1.495, 4.38) .. controls (1.579, 4.351) and (1.694, 4.303) .. (1.802, 4.247) .. controls (1.91, 4.192) and (2.012, 4.129) .. (2.11, 4.068) .. controls (2.207, 4.007) and (2.3, 3.946) .. (2.382, 3.894) .. controls (2.464, 3.842) and (2.537, 3.799) .. (2.627, 3.752) .. controls (2.717, 3.706) and (2.826, 3.656) .. (2.932, 3.621) .. controls (3.038, 3.585) and (3.142, 3.564) .. (3.215, 3.552) .. controls (3.288, 3.54) and (3.331, 3.538) .. (3.382, 3.54) .. controls (3.434, 3.542) and (3.494, 3.547) .. (3.568, 3.561) .. controls (3.642, 3.575) and (3.73, 3.598) .. (3.833, 3.639) .. controls (3.936, 3.681) and (4.054, 3.741) .. (4.151, 3.8) .. controls (4.248, 3.858) and (4.323, 3.913) .. (4.389, 3.966) .. controls (4.455, 4.019) and (4.511, 4.068) .. (4.565, 4.117) .. controls (4.618, 4.166) and (4.668, 4.213) .. (4.717, 4.261) .. controls (4.765, 4.308) and (4.812, 4.355) .. (4.868, 4.411) .. controls (4.924, 4.468) and (4.989, 4.534) .. (5.046, 4.59) .. controls (5.102, 4.646) and (5.15, 4.693) .. (5.197, 4.737) .. controls (5.244, 4.781) and (5.289, 4.822) .. (5.343, 4.865) .. controls (5.396, 4.909) and (5.457, 4.955) .. (5.52, 4.997) .. controls (5.582, 5.038) and (5.645, 5.075) .. (5.692, 5.101) .. controls (5.739, 5.126) and (5.771, 5.141) .. (5.824, 5.161) .. controls (5.878, 5.181) and (5.953, 5.206) .. (6.037, 5.227) .. controls (6.121, 5.249) and (6.213, 5.266) .. (6.294, 5.276) .. controls (6.375, 5.286) and (6.445, 5.288) .. (6.506, 5.286) .. controls (6.568, 5.283) and (6.621, 5.275) .. (6.68, 5.258) .. controls (6.74, 5.242) and (6.807, 5.217) .. (6.86, 5.192) .. controls (6.914, 5.167) and (6.955, 5.142) .. (6.995, 5.114) .. controls (7.035, 5.086) and (7.075, 5.055) .. (7.116, 5.019) .. controls (7.157, 4.983) and (7.2, 4.941) .. (7.251, 4.887) .. controls (7.303, 4.833) and (7.362, 4.767) .. (7.419, 4.699) .. controls (7.476, 4.631) and (7.531, 4.561) .. (7.578, 4.499) .. controls (7.626, 4.437) and (7.666, 4.383) .. (7.709, 4.323) .. controls (7.752, 4.262) and (7.799, 4.196) .. (7.854, 4.114) .. controls (7.91, 4.033) and (7.975, 3.936) .. (8.184, 3.623);		
		
		\filldraw[draw=BrickRed!90, ultra thick, fill=Salmon, fill opacity=0.3, draw opacity=.7] (1.976, 6.445) rectangle (5.927, 1);
		
		\node[circle, fill, inner sep=3pt] at (3.927, 3.666) {};
		
		\filldraw[draw=RoyalBlue!80, ultra thick, fill=Lavender, fill opacity=0.3, draw opacity=.7] (3.106, 6.92) -- (3.106, 6.01) -- (3.556, 6.01) -- (3.558, 6.921) -- cycle;
		\filldraw[draw=RoyalBlue!80, ultra thick, fill=Lavender, fill opacity=0.3, draw opacity=.7] (4.903, 5.028) rectangle (5.349, 4.137);
		
		\filldraw[draw=RoyalBlue!80, thick, fill=Lavender] (3.351, 6.914) .. controls (3.189, 6.803) and (3.109, 6.748) .. (3.109, 6.748);		
		\filldraw[draw=RoyalBlue!80, thick, fill=Lavender] (3.542, 6.884) -- (3.124, 6.587);
		\filldraw[draw=RoyalBlue!80, thick, fill=Lavender] (3.536, 6.689) -- (3.186, 6.47);
		\filldraw[draw=RoyalBlue!80, thick, fill=Lavender] (3.531, 6.533) .. controls (3.457, 6.489) and (3.419, 6.467) .. (3.419, 6.467);
		\filldraw[draw=RoyalBlue!80, thick, fill=Lavender] (5.347, 4.715) -- (5.01, 4.519);
		\filldraw[draw=RoyalBlue!80, thick, fill=Lavender] (5.344, 4.521) -- (4.93, 4.309);
		\filldraw[draw=RoyalBlue!80, thick, fill=Lavender] (5.328, 4.345) .. controls (5.084, 4.217) and (4.963, 4.154) .. (4.963, 4.154);
		\filldraw[draw=RoyalBlue!80, thick, fill=Lavender] (5.341, 4.201) -- (5.259, 4.157);
		
		\node[rotate=90, anchor=center, text=black, font=\small] at (1.657, 2.5) {$\mathcal{B}\mathbf{Q}_R(\mbf{0}) \subset \mathcal{T}_{=T_{\min}(\mathbf{Q}_R(\mbf{0}))}$};
		\node[rotate=90, anchor=center, text=black, font=\small] at (6.302, 2.839) {$\partial_s (\mathbf{Q}_R(\mbf{0})) \subset \mathcal{T}_{=T_{\max}(\mathbf{Q}_R(\mbf{0}))}$};
		\node[anchor=center, font=\Large] at (4.781, 6.753) {$\Sigma_R$};
		\node[anchor=center, font=\large] at (4.188, 3.296) {$\mbf{0} = (0, 0)$};
		\node[anchor=center, font=\Large] at (4.028, 1.59) {$\mathbf{Q}_R(\mbf{0})$};
		\node[anchor=center, font=\LARGE] at (7.46, 7.564) {$\Omega$};
		\node[anchor=center, font=\Large] at (7.38, 5.295) {$\Sigma$};
		\node[anchor=center, font=\LARGE] at (3.845, 5.318) {$\Omega_R$};
	\end{tikzpicture}
	}
	\caption{The boundary of $\Omega_R = \Omega \cap \Q_R(\mbf{0})$ is divided into several lateral and vertical regions.}
	\label{fig:restricted_boundary}
\end{figure}
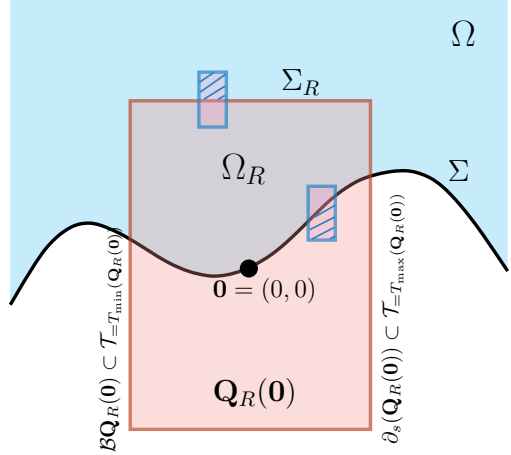

Before proceeding with the main proof of the section, let us present as an auxiliary result the fact that solutions to the continuous Dirichlet problem satisfy maximum principles even in unbounded domains. This is a consequence of the fact that the boundary value at infinity must be attained continuously, because $\infty \in \partial_e \Omega$ when $\Omega$ is unbounded.

\begin{lemma}[Classical maximum principle in unbounded domains]
	\label{lem:Uniqueness}
	Let $L$ be a parabolic operator as in Definition~\ref{def:operator}. Let $\Omega\subseteq\R^{n+1}$ be an open set, and $f\in C(\partial_e\Omega)$. If $u$ is a weak solution to $Lu= 0$ in $\Omega$ that continuously attains the boundary value $f$ on $\partial_e \Omega$, then $u$ satisfies the maximum principle 
	\[
	\inf_{\partial_e \Omega} f 
	\leq 
	\inf_\Omega u 
	\leq 
	\sup_\Omega u 
	\leq 
	\sup_{\partial_e \Omega} f.
	\]
	Concretely, there is at most one solution to the continuous Dirichlet problem.
\end{lemma}

\begin{proof}
	The only meaningful case is when $\Omega$ is unbounded since otherwise, this is a consequence of classical maximum principles. In such a case, fix $\X_0 \in \Omega$ and let $\e > 0$. Since $\lim_{\Omega \ni \X \to \infty} u(\X) = f(\infty)$, we must have 
	\[
	\inf_{\partial_e \Omega} f - \e
	\leq 
	f(\infty) - \e
	\leq 
	u(\X)
	\leq 
	f(\infty) + \e 
	\leq 
	\sup_{\partial_e \Omega} f + \e, 
	\quad 
	\text{for } \X \in \Omega \cap \Q_R(\mbf{0})^c, 
	\]
	if $R = R(\e) > 0$ is large enough. By the classical maximum principle in bounded domains, 
	\[
	\inf_{\partial_e \Omega} f - \e
	\leq 
	u(\X)
	\leq 
	\sup_{\partial_e \Omega} f + \e, 
	\quad 
	\text{for } \X \in \Omega \cap \Q_{2R}(\mbf{0}).
	\]
	Concretely, for $\e > 0$ small enough (so $R$ large enough), this applies to $\X = \X_0$, and letting $\e \to 0$ yields the result. The uniqueness of solutions follows by linearity of the equation.
\end{proof}

With all the tools already developed, we can construct solutions to the continuous Dirichlet problem in unbounded domains that satisfy the maximum principle and yield a parabolic measure. For that, we will approximate our domain by bounded domains to take advantage of Theorem~\ref{thm:ExistenceOfParabolicMeasureOnBoundedDomains}, and then carry out a detailed analysis of what happens at the point at infinity. It will be important to, at all points in the proof, be certain that solutions attain their boundary values continuously, for that is a key ingredient of basic results like the maximum principle in Lemma~\ref{lem:Uniqueness}.

\begin{theorem}[Existence of parabolic measure on unbounded domains]
\label{thm:ExistenceOfParabolicMeasureOnUnboundedDomains}
    Let $L=\partial_t-\div A\nabla$ be a parabolic operator as in Definition~\ref{def:operator}, and $\Omega\subseteq\R^{n+1}$ be an unbounded open set that satisfies the TBCDC for $\partial_t-M_2\Delta$, where $M_2$ is the constant from Lemma~\ref{lem:capLtoM}. Then, for each $(X,t)\in\Omega$, there exists a unique positive Radon measure $\omega_L^{X,t}$ supported on $\partial_e\Omega$ so that for all $f\in C(\partial_e\Omega)$, 
    \[
        u(X,t)=\int_{\partial_e\Omega}f\,d\omega_L^{X,t}, 
        \qquad (X, t) \in \Omega,
    \]
    solves the continuous Dirichlet problem with boundary data $f$ and satisfies the maximum principle $\norm{u}_{L^\infty(\Omega)} \leq \norm{f}_{L^\infty(\partial_e \Omega)}$.
\end{theorem}

\begin{proof}
    From Lemma \ref{lem:StructureOfUnboundedTBCDCDomains}, we know that $\partial_e\Omega\setminus\{\infty\}$ is either unbounded or empty. The case that it is empty is easy (we may just consider solutions which constantly take the value $f(\infty)$), so let us concentrate on the case that $\partial_e\Omega\setminus\{\infty\}$ is unbounded. We will now split the proof into several parts.
    
    \textbf{Part I: when $f$ is compactly supported.} Let us assume that $\supp(f)\subseteq \Q_{R_0}(0,0)$ for some ${R_0}>0$. Assume also $f\geq 0$.

    \textbf{Step 1: construction of $u$.} Consider $R \geq 2 {R_0}$. By Theorem \ref{thm:ExistenceOfParabolicMeasureOnBoundedDomains} (we can apply it because $\Omega_R$ satisfies the TBCDC by Lemma~\ref{lem:TBCDConOmegaR}, and constants will not depend on $R$), the parabolic measure $\omega_R^\X$ for $L$ on $\Omega_R$ exists. Thus,
    \[
        u_R(\X):=\int_{\partial_e\Omega}f\,d\omega_R^\X, 
        \qquad \X \in \overline{\Omega_R},
    \]
    solves the continuous Dirichlet problem on $\Omega_R$ with boundary data $f\mathbf{1}_{\partial_e \Omega}$ (which is a continuous function on $\partial \Omega_R$ because $\supp f \subseteq \Q_{R_0} \Subset \Q_R$). Let us extend $u_R$ by 0 in $\overline{\Omega} \setminus \overline{\Omega_R}$, so that $u_R \in C(\overline{\Omega})$.
    
    It is easy to see that if $R^* > R$, the maximum principle yields $u_{R^*} \geq u_R$ in $\Omega_R$ (see Figure~\ref{fig:monotone}, recalling that $f \geq 0$, so $u_R, u_{R^*} \geq 0$). Hence, $u_{R^*} \geq u_R$ in $\Omega$ by the way we have extended these solutions. Moreover, by the maximum principle for bounded domains, $\norm{u_R}_\infty \leq \norm{f}_\infty$ for any $R$. Hence, pointwise limits exist and we can define
    \[
        u(\X):=\lim_{R\to\infty}u_R(\X), \qquad \X \in \overline{\Omega}.
    \]
    This definition and the upper bound we had for the $u_R$ readily yield $\norm{u}_\infty \leq \norm{f}_\infty$.
    
    \begin{figure}
    	\centering 
    	\scalebox{.8}{
    	\begin{tikzpicture}[scale=1]
    		\fill[SkyBlue, opacity=0.3] 
    		(0, 3.685)
    		.. controls (0, 3.685) and (0, 3.685) .. (0.188, 3.874)
    		.. controls (0.376, 4.062) and (0.753, 4.438) .. (1.317, 4.344)
    		.. controls (1.881, 4.25) and (2.634, 3.685) .. (3.246, 3.591)
    		.. controls (3.857, 3.497) and (4.327, 3.874) .. (4.986, 4.015)
    		.. controls (5.644, 4.156) and (6.491, 4.062) .. (7.009, 3.968)
    		.. controls (7.526, 3.874) and (7.714, 3.78) .. (7.808, 3.733)
    		.. controls (7.902, 3.685) and (7.902, 3.685) .. (7.902, 3.685)
    		.. controls (7.902, 3.685) and (7.902, 3.685) .. (7.902, 4.438)
    		.. controls (7.902, 5.191) and (7.902, 6.696) .. (7.902, 7.448)
    		.. controls (7.902, 8.201) and (7.902, 8.201) .. (7.902, 8.201)
    		.. controls (7.902, 8.201) and (7.902, 8.201) .. (6.585, 8.201)
    		.. controls (5.268, 8.201) and (2.634, 8.201) .. (1.317, 8.201)
    		.. controls (0, 8.201) and (0, 8.201) .. (0, 8.201)
    		.. controls (0, 8.201) and (0, 8.201) .. (0, 8.201)
    		.. controls (0, 8.201) and (0, 8.201) .. (0, 7.448)
    		.. controls (0, 6.696) and (0, 5.191) .. (0, 4.438)
    		.. controls (0, 3.685) and (0, 3.685) .. cycle;
    		
    		\draw[ultra thick] (0, 3.685) .. controls (0.224, 3.91) and (0.339, 4.01) .. (0.427, 4.08) .. controls (0.515, 4.15) and (0.575, 4.189) .. (0.61, 4.211) .. controls (0.645, 4.233) and (0.654, 4.238) .. (0.683, 4.25) .. controls (0.712, 4.263) and (0.76, 4.283) .. (0.824, 4.302) .. controls (0.887, 4.32) and (0.966, 4.337) .. (1.036, 4.346) .. controls (1.106, 4.355) and (1.169, 4.355) .. (1.241, 4.347) .. controls (1.313, 4.339) and (1.396, 4.323) .. (1.459, 4.307) .. controls (1.523, 4.292) and (1.566, 4.278) .. (1.62, 4.258) .. controls (1.675, 4.237) and (1.74, 4.211) .. (1.801, 4.185) .. controls (1.863, 4.159) and (1.921, 4.132) .. (1.975, 4.107) .. controls (2.028, 4.082) and (2.077, 4.059) .. (2.133, 4.032) .. controls (2.188, 4.005) and (2.25, 3.974) .. (2.315, 3.943) .. controls (2.379, 3.912) and (2.446, 3.879) .. (2.505, 3.852) .. controls (2.564, 3.824) and (2.614, 3.801) .. (2.673, 3.776) .. controls (2.732, 3.751) and (2.8, 3.723) .. (2.867, 3.698) .. controls (2.935, 3.674) and (3.003, 3.651) .. (3.068, 3.634) .. controls (3.133, 3.616) and (3.194, 3.603) .. (3.264, 3.594) .. controls (3.334, 3.586) and (3.412, 3.582) .. (3.493, 3.585) .. controls (3.575, 3.588) and (3.66, 3.599) .. (3.73, 3.612) .. controls (3.8, 3.624) and (3.854, 3.638) .. (3.919, 3.658) .. controls (3.984, 3.677) and (4.061, 3.703) .. (4.117, 3.723) .. controls (4.172, 3.742) and (4.207, 3.755) .. (4.261, 3.776) .. controls (4.316, 3.797) and (4.39, 3.826) .. (4.466, 3.854) .. controls (4.541, 3.882) and (4.618, 3.908) .. (4.688, 3.931) .. controls (4.758, 3.954) and (4.821, 3.973) .. (4.868, 3.985) .. controls (4.915, 3.998) and (4.945, 4.005) .. (5.001, 4.016) .. controls (5.058, 4.027) and (5.14, 4.041) .. (5.228, 4.052) .. controls (5.316, 4.063) and (5.41, 4.072) .. (5.497, 4.077) .. controls (5.585, 4.082) and (5.666, 4.084) .. (5.744, 4.085) .. controls (5.821, 4.085) and (5.896, 4.084) .. (5.984, 4.08) .. controls (6.073, 4.077) and (6.176, 4.07) .. (6.28, 4.062) .. controls (6.385, 4.053) and (6.49, 4.042) .. (6.605, 4.027) .. controls (6.719, 4.012) and (6.842, 3.994) .. (6.955, 3.974) .. controls (7.068, 3.955) and (7.17, 3.934) .. (7.284, 3.907) .. controls (7.399, 3.88) and (7.525, 3.846) .. (7.902, 3.685);
    		
    		\node[circle, fill, inner sep=3pt] at (3.967, 3.673) {};
    		
    		\filldraw[draw=BrickRed!90, ultra thick, fill=Salmon, opacity=0.3] (2.776, 1.393) rectangle (5.135, 5.939);
    		\filldraw[draw=BrickRed!90, ultra thick, fill=Salmon, opacity=0.3] (1.801, 0) rectangle (6.197, 7.366);
    		
    		\draw[RoyalBlue!80, ultra thick] (3.532, 2.621) rectangle (4.405, 4.72);
    		\draw[RoyalBlue!80, ultra thick] (3.667, 3.595) .. controls (3.85, 3.638) and (3.983, 3.679) .. (4.089, 3.715) .. controls (4.195, 3.751) and (4.275, 3.782) .. (4.275, 3.782);
    		
    		\node[anchor=center, font=\Large] at (3.974, 1.767) {$\mathbf{Q}_R(0, 0)$};
    		\node[anchor=center, font=\Large] at (3.924, 0.374) {$\mathbf{Q}_{R\textsuperscript{*}}(0, 0)$};    		
    		\node[rotate=15.899, anchor=center, text=RoyalBlue!80] at (3.947, 3.371) {supp($f$)};
    		\node[anchor=center, font=\large] at (6.2, 3.5) {$u_R = u_{R\textsuperscript{*}} = f$};
    		\node[anchor=center, font=\Large] at (0.628, 4.574) {$\Sigma$};
    		\node[anchor=center, font=\large] at (3.922, 6.215) {$u_R = 0$};
    		\node[anchor=center, font=\large] at (3.906, 7.675) {$u_{R\textsuperscript{*}} = 0$};
    		\node[anchor=center, font=\huge] at (7.368, 7.527) {$\Omega$};
    		\node[anchor=center, font=\Large] at (4.693, 5.444) {$\Omega_R$};
    		\node[anchor=center, font=\Large] at (5.711, 6.887) {$\Omega_{R\textsuperscript{*}}$};
    	\end{tikzpicture}
    	}
    	\caption{The solutions $u_R$ grow (by the maximum principle) as $R$ grows.}
    	\label{fig:monotone}
	\end{figure}
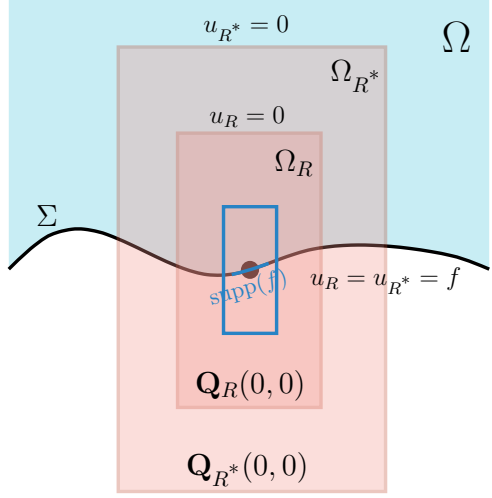

    \textbf{Step 2: $u\in W^{1,2}_\loc(\Omega)$.} Let $\Q$ be a parabolic cube such that $4\Q\subseteq\Omega$. For $R$ large enough, we have $4\Q \subseteq \Omega_R$, which allows us to use Caccioppoli's estimate from Lemma~\ref{lem:caccioppoli} (along with the maximum principle discussed before) to obtain 
    \[
        \limsup_{R \to \infty}\iint_\Q|\nabla u_R|^2
        \lesssim_\Q
        \limsup_{R \to \infty}\iint_{2\Q}|u_R|^2
        \lesssim_\Q
        \limsup_{R \to \infty} \, \lVert u_R\rVert_\infty
        \leq
        \norm{f}_\infty.
    \]
    This uniform bound in $W^{1, 2}(\Q)$ yields, by standard compactness and covering arguments (and recalling Step 1), that (up to a subsequence) $u \in W^{1, 2}_\loc(\Omega)$ and $\nabla u_R \rightharpoonup \nabla u$ in $L^2_\loc(\Omega)$.
	

    \textbf{Step 3: $Lu=0$ in the weak sense on $\Omega$.} Let $\psi\in C_c^\infty(\Omega)$. Then, $\supp \psi \subseteq \Omega_R$ for $R \gg 1$, and we may repeat the very same computation in Theorem~\ref{thm:ExistenceOfParabolicMeasureOnBoundedDomains}, Part II, Step 4, replacing $u_j$ by $u_R$, to determine that $Lu = 0$ in the weak sense in $\Omega$. 
    
    
    \textbf{Step 4: $u\in C(\Omega\cup\partial_n\Omega)$.} Since we just showed that $u$ solves $Lu=0$ in $\Omega$, interior H\"older continuity (Lemma~\ref{lem:interiorHolder}) implies that $u \in C(\Omega)$. Next, fix $\x_0 \in \partial_e \Omega \setminus \{\infty\}$, and $r > 0$. If $R > 2R_0$ is large enough so that $\Q_{10r}(\x_0) \subseteq \Q_{R/2}(\mbf{0})$, and $R^* > R$, we have, for $\X \in \Q_r(\x_0) \cap \Omega$, that
    \[
    0 
    \leq 
    (u_{R^*} - u_R)(\X)
    \lesssim 
    \left( \frac{\delta_{\Omega_R}(\X)}{r} \right)^{\alpha_H} 
    \norm{u_{R^*} - u_R}_\infty 
    \lesssim 
    \left( \frac{\delta(\X)}{r} \right)^{\alpha_H} 
    \norm{f}_\infty 
    \underset{\X \to \x_0}{\longrightarrow}
    0,
    \]
    where we have used \eqref{maineqn:HolderEstimate} for the solution $u_{R^*} - u_R$ in $\Omega_R$, because it vanishes continuously on $\partial_e \Omega \cap \Q_{2r}(\x_0)$, and $\Omega_R$ satisfies the TBCDC by Lemma~\ref{lem:TBCDConOmegaR} (and the TBCDC constants do not depend on $R$, so in the last display the constants are independent of $R$, too). Letting $R^* \to \infty$, this implies that
    \[
    	\lim_{\X \to \x_0}(u-u_R)(\X)=0. 
    \]
    Taking into account also that $u_R \in C(\overline{\Omega_R})$ and $u_R = f$ on $\partial_e \Omega \cap \Omega_R$, this yields
    \[
    \lim_{\X \to \x_0} u(\X)
    =
    \lim_{\X \to \x_0} (u-u_R)(\X) + \lim_{\X \to \x_0}u_R(\X)
    =
    f(\x_0)
    =
    u(\x_0), 
    \]
    so $u$ is continuous at $\x_0$. That is, $u$ is continuous on $\partial_e \Omega \setminus \{\infty\}$, taking the value $f$.
    
    It remains to study the behavior at infinity: since $u_R(\infty) = 0$ for every $R$ (recall that we extended the $u_R$ by 0), we need to show that $\lim_{\X \to \infty} u(\X) = 0$ to acknowledge continuity at this point. First note that since $\supp f \subseteq \Q_{R_0}(\mbf{0})$, we clearly have $u_R(\X) = 0$ for $\X \in \T_{< -R_0^2}$, so only points in $\T_{\geq -R_0^2}$ are relevant. For that purpose, fix $\Y_0 \in \T_{< -2R_0^2}$ and note that Lemma~\ref{lem:fund_sol} implies that $\Gamma_L(\X; \Y_0) \geq c_6$ for all $\X \in \supp f$, where $c_6 > 0$ may depend on $R_0, \norm{\Y_0}, n$ and $\lambda$. With this in mind, define 
    \[
    v(\X)
    :=
    \frac{\Gamma_L(\X; \Y_0)}{c_6} \norm{f}_\infty, 
    \qquad 
    \X \in \Omega \cap \T_{\geq -2R_0^2}.
    \]
    Clearly, $v \geq 0$, and $v \geq f$ on $\partial_e \Omega$ (recall our choice of $c_6$). Then, it is easy to infer that, for any $R$, $v \geq u_R$ on $\partial (\Omega_R \cap \T_{> -2R_0^2})$, which extends to $\Omega_R \cap \T_{> -2R_0^2}$ by the maximum principle because $Lv = 0$ in $\Omega \cap \T_{> -2R_0^2}$. By the definition of $u$, it must then hold that $v \geq u$ in $\Omega \cap \T_{>-2R_0^2}$. Furthermore, $\lim_{\X \to \infty} v(\X) = 0$ by the inherent decay of fundamental solutions (see e.g. \eqref{eq:aronson_easy}), which finally implies that $\lim_{\X \to \infty} u(\X) = 0$, as desired.
    
    
    \textbf{Part II: when $f$ vanishes at infinity.} Let us now consider the case that $f(\infty)=0$, without necessarily having compact support as in Part I. It will suffice to restrict to the case that $f\geq 0$. Set, for every $k \in \N$,
    \[
        f_k(\mbf{x}):=\max\Big\{f(\mbf{x})-\frac{1}{k},0\Big\}, 
        \qquad 
        \text{for } \mbf{x}\in\partial_e\Omega,
    \]
    so that, for each $k\in\N$, $f_k$ satisfies the assumptions of Part I (because $f \in C(\partial_e \Omega)$), whence there exists an associated continuous solution $u_k$. It clearly holds $f_{k+1}\geq f_k$ for $k\in\N$, so $u_{k+1}\geq u_k$ by the construction in Part I. Moreover, from Part I we also know that $0 \leq u_k \leq \norm{f_k}_\infty \leq \norm{f}_\infty$. This monotonicity allows us to define 
    \[
    u(\mbf{X}):=\lim_{k\to\infty}u_k(\mbf{X}), \qquad \X \in \overline{\Omega},
    \]
    and it still satisfies $0 \leq u \leq \norm{f}_\infty$. 
    
    Now, the facts that $u\in W^{1,2}_\loc(\Omega)$, $Lu=0$, $u\in C(\overline{\Omega})$, and $u|_{\partial_e\Omega\setminus\{\infty\}}=f$ follow in exactly the same way as in Part I, Steps 2, 3, and 4. The only part in Step 4 before that requires a new proof is the fact that
    \[
        \lim_{\mbf{X}\to\infty}u(\mbf{X})=0=f(\infty).
    \]

    To show this, let $\e>0$. 
    Then, it is easy to note that $0\leq f_\ell-f_k<\e$ for all $\ell>k>1/\e$. Since $u_\ell-u_k$ is a solution attaining the boundary value $f_\ell - f_k$ continuously, the maximum principle (see Lemma~\ref{lem:Uniqueness}) yields $0\leq u_\ell-u_k\leq\e$ in $\Omega$. Letting $\ell \to \infty$, we obtain $0\leq u - u_k\leq\e$ in $\Omega$ for all $k > 1/\e$. By what was shown in Part I, $\abs{u_k(\X)} < \e$ for $\X \in \Omega \cap \Q_R(\mbf{0})^c$ if $R$ is large enough, whence $0 \leq u(\X) < 2\e$ for such $\X$. This precisely shows that $\lim_{\X \to \infty} u (\X) = 0$, as desired.
	

    \textbf{Part III: existence of the parabolic measure.} Now consider a general $f\in C(\partial_e\Omega)$. Associated to the functions $(f - f(\infty))_+$ and $(f - f(\infty))_-$, we can construct, following Part II, positive solutions $u_+$ and $u_-$ in $\Omega$ attaining the boundary values continuously, satisfying the maximum principle, and vanishing at infinity. Thus, $u := f(\infty) + u_+ - u_-$ is also a solution in $\Omega$ satisfying the maximum principle, which attains the boundary value $f$ continuously on $\partial_e \Omega$ (including the point at infinity). This construction generates a map $f \mapsto u$ between $C(\partial_e \Omega)$ and $C(\overline{\Omega})$, which is linear by the uniqueness of solutions to the continuous Dirichlet problem provided by the maximum principle in unbounded domains in Lemma~\ref{lem:Uniqueness}. The map is also positive (see the construction of $u$ above in the paragraph, and Parts I and II, where we only considered non-negative data).

    Given $\X \in \Omega$, the discussion in last paragraph means that the map $C(\partial_e \Omega) \ni f \mapsto u(\X) \in \R$ is linear, bounded and positive. Thus, viewing $\partial_e\Omega$ as a compact set under the one-point compactification of $\R^{n+1}$ (since $\infty \in \partial_e\Omega$),   
    the Riesz representation theorem implies the existence of the parabolic measure, as in the statement of the theorem.
\end{proof}

What we have already shown is basically what was asserted in Theorem~\ref{mainthm:ExistenceOfParabolicMeasure}. Let us quickly finish the proof.

\begin{proof}[\textbf{Proof of Theorem \ref{mainthm:ExistenceOfParabolicMeasure}}]
    In Lemma~\ref{thm:TBCDCimpliesWiener} we saw that the parabolic measure exists in the case that $L$ has smooth coefficients $\Omega$ is bounded, which we generalized to merely bounded coefficients (as in Definition~\ref{def:operator}) in Theorem~\ref{thm:ExistenceOfParabolicMeasureOnBoundedDomains}, and even further in Theorem~\ref{thm:ExistenceOfParabolicMeasureOnUnboundedDomains} to unbounded domains. Therefore, we are only left to see that this parabolic measure is a probability: the reader can readily check that if we take $f \equiv 1$ as boundary datum, then $u \equiv 1$ is a solution. Moreover, $u$ is the unique solution by Lemma~\ref{lem:Uniqueness}. Thus, $\omega_L^\X(\partial_e\Omega) = \int_{\partial_e\Omega}1\,d\omega_L^\X = u(\X) = 1$ for all $\X \in \Omega$, as desired.
\end{proof}


\section{What happens if we do not impose values at infinity?} \label{sec:probability}

Until now, we have taken $\infty \in \partial_e \Omega$ whenever $\Omega$ is unbounded. Since $\partial_e \Omega$ is the set supporting the parabolic measure, this means that we have only considered the case that we can impose boundary values at the point at infinity. In this section, we explore a different possibility (closer e.g. to the one in \cite{CHMPZ}), when one does not impose any condition at infinity. This gives more freedom to find solutions to the Dirichlet problem, which can of course be beneficial, but as a consequence, we may run into the problem that these problems do not have unique solutions.

In this section, we will explore which geometrical/topological conditions ensure that the Dirichlet problem has (or do not have) unique solutions.
As for elliptic equations, the fact that the caloric/parabolic measure is a probability (or not) is heavily influenced by the domain and its boundary being bounded (or not). Hence, whether or not we consider the point at infinity as part of our boundary really does make a difference.

\begin{proposition} \label{prop:probability}
	Let $L$ be a parabolic operator as in Definition~\ref{def:operator}, and $\Omega \subseteq \R^{n+1}$ be an open set. Let $(X, t) \in \Omega$, and assume that the parabolic measure with pole $(X, t)$ exists for $L$ in $\Omega$. Let $U$ be the connected component of $\Omega \cap \mathcal{T}_{< t}$ for which $(X, t) \in \overline{U}$. Then it holds:
	\begin{enumerate}
		\item \label{it:bdd} If $U$ is bounded, then $\omega_L^{X, t}(\partial_e \Omega) = 1$.
		\item \label{it:unbdd_bdd} If $U$ is unbounded and $\partial_e U \setminus \{\infty\}$ is bounded, then $\omega_L^{X, t}(\partial_e \Omega \setminus \{\infty\}) < 1$.
		\item \label{it:unbdd_unbdd} If $U$ is unbounded and $\partial_e U \setminus \{\infty\}$ is unbounded:
		\begin{enumerate}
			\item \label{it:unbdd_unbdd_Tmin} if $T_{\min}(U) > -\infty$, then $\omega_L^{X, t}(\partial_e \Omega \setminus \{\infty\}) = 1$, 
			\item \label{it:unbdd_unbdd_TBHCC} if $T_{\min}(U) = -\infty$ and $\Omega$ satisfies the TBCDC for any parabolic operator, then $\omega_L^{X, t}(\partial_e \Omega \setminus \{\infty\}) = 1$,
			\item \label{it:unbdd_unbdd_counterexample} if $T_{\min}(U) = -\infty$, but $\R^{n+1}\setminus U$ does not satisfy a backwards thickness assumption (like the TBCDC or TBHCC) for large scales, then it may happen that $\omega_L^{X, t}(\partial_e \Omega \setminus \{\infty\}) < 1$.
		\end{enumerate}
	\end{enumerate}
\end{proposition} 

\begin{remark} \label{rem:probability}
	Some remarks are in order before diving into the proof:
	\begin{itemize}
        \item We note that, in presence of the TBHCC or TBCDC (for any operator as in Definition~\ref{def:operator}), it always holds $\omega_L^{X, t}(\partial_e \Omega \setminus \{\infty\}) =1$ for every $(X, t) \in \Omega$. Indeed, Case~\eqref{it:unbdd_bdd} is ruled out by Lemma~\ref{lem:StructureOfUnboundedTBCDCDomains}, and Case \eqref{it:unbdd_unbdd_counterexample} is clearly not possible.
    
		\item For simplicity in the notation, we will restrict the discussion to the heat equation $L = \partial_t - \Delta$, but everything will work very similarly for the operators in Proposition~\ref{prop:probability}. The reader can check that everything translates to general $L$ by only replacing the appearances of Gaussians like $\Gamma := \Gamma_{\partial_t-\Delta}$ with other Gaussians (with constants) that bound $\Gamma_L$ from above and below by Lemma~\ref{lem:capLtoM} (and \eqref{eq:fund_sol_heat}). We will also abbreviate $\omega := \omega_{\partial_t - \Delta}$.
		
		\item As already discussed in the previous sections, since we consider that $\infty \in \partial_e \Omega$, maximum principles are available and it always holds $\omega_L^\X (\partial_e \Omega) = 1$ for any $\X \in \Omega$ (see also \cite[Theorem 8.27]{W2}). Hence, the only relevant difference with regards to Proposition~\ref{prop:probability} is what happens precisely at the point at infinity.
		
		\item As we have already discussed several times, a very fundamental property of parabolic equations is that heat only flows towards the future, so to determine whether $\omega^{X, t}$ is a probability or not, only what happens to the past of $(X, t)$ matters. (For the heat equation, see e.g. \cite[Lemma 8.29]{W2},\footnote{The notation for that result is explained in \cite[p. 25]{W2}} but this is also reflected in the one-sided nature of fundamental solutions, for instance). Mathematically, one could write $\omega^{X, t}(\partial_e \Omega \cap \T_{\geq t}) = 0$. Thus, because $\omega^{X, t}$ is actually supported on $\partial_e U$, and not on the whole essential boundary $\partial_e \Omega$, only properties about $U$ (as opposed to the whole domain $\Omega$) matter in Proposition~\ref{prop:probability}. We will make use of the property 
		\begin{equation*} 
			\omega^{X, t}(\partial_e \Omega \setminus \{\infty\}) 
			= \omega^{X, t}_U(\partial_e \Omega \cap \partial_e U \setminus \{\infty\}) 
			= \omega^{X, t}_U(\partial_e U \setminus \{\infty\}).
		\end{equation*}
	\end{itemize}
\end{remark}

Along the proof of Proposition~\ref{prop:probability}, we will make extensive use of a simple (not meant to be optimal) estimate of how heat evolves in the complement of a cube in the long term.

\begin{lemma} \label{lem:complement_cube}
	It holds, for $R \geq 2$ large enough (depending on $n$), 
	\[
	\omega^{(0, R^2)}_{\R^{n+1} \setminus \Q_1(0, -1)} (\partial \Q_1(0, -1))
    \lesssim 
    R^{-n/2},
    \]
    where the implicit constant only depends on $n$.
\end{lemma}
\begin{proof}
	Set $u(X, t) := \omega^{X, t}_{\R^{n+1} \setminus \Q_1(0, -1)}(\partial \Q_1(0, -1))$. Note that $u \equiv 1$ on $\partial \Q_1(0, -1)$, and $u \equiv 0$ in $\mathcal{T}_{< -2}$. We claim, and will shortly prove, that 
	\begin{equation} \label{eq:claim_exponential_cube}
		u(X, 0)
		\lesssim
		\int_{\abs{X}-1}^\infty e^{-c\rho^2} \, d\rho, 
		\qquad 
		\text{ if } \abs{X} \geq 1.
	\end{equation}
    Assuming this claim holds for the moment, and recalling that $u \equiv 1$ on $\partial \Q_1(0, -1)$, we can estimate $u(0, R^2)$ using the heat kernel (see Lemma~\ref{lem:fund_sol} and Figure~\ref{fig:complement_cube}):
	\begin{align*}
		u(0& , R^2)
		=
		\int_{\{\abs{Y} \leq 1\}} \Gamma(0, R^2; Y, 0) \, u(Y, 0) \, dY
		+ \int_{\{\abs{Y} > 1\}} \Gamma(0, R^2; Y, 0) \, u(Y, 0) \, dY
		\\ & \lesssim
		\int_{\{\abs{Y} \leq 1\}} R^{-n} e^{-c\frac{\abs{Y}^2}{R^2}} \, dY
		+ \int_{\{\abs{Y} > 1\}} R^{-n} e^{-c\frac{\abs{Y}^2}{R^2}} \int_{\abs{Y}-1}^\infty e^{-c\rho^2} \, d\rho \, dY
		\\ & \lesssim
		R^{-n}
		+ \int_{1}^{\sqrt{R}} r^{n-1} R^{-n} e^{-c\frac{r^2}{R^2}} \int_{r-1}^\infty e^{-c\rho^2} \, d\rho \, dr
		+ \int_{\sqrt{R}}^\infty r^{n-1} R^{-n} e^{-c\frac{r^2}{R^2}} \int_{r-1}^\infty e^{-c\rho^2} \, d\rho \, dr
		\\ & \lesssim
		R^{-n}
		+ \int_{1}^{\sqrt{R}} r^{n-1} R^{-n} e^{-c\frac{r^2}{R^2}} \, dr
		+ \int_{\sqrt{R}}^\infty r^{n-1} R^{-n} e^{-c\frac{r^2}{R^2}} \int_{\sqrt{R}/2}^\infty e^{-c\rho^2} \, d\rho \, dr
		\\ & =
		R^{-n}
		+ \int_{1/R}^{1/\sqrt{R}} s^{n-1} e^{-cs^2} \, ds
		+ \left( \int_{1/\sqrt{R}}^\infty s^{n-1} e^{-cs^2} \, ds \right) \left( \int_{\sqrt{R}/2}^\infty e^{-c\rho^2} \, d\rho\right) 
		\\ & \lesssim 
		R^{-n} 
		+ R^{-n/2}
		+ e^{-cR}
		\lesssim 
		R^{-n/2}
	\end{align*}
    if $R$ is large enough.
	In the second line, we have used \eqref{eq:fund_sol_heat} and \eqref{eq:claim_exponential_cube}.	
	The rest of calculations follow by elementary real variable techniques: changes of variables (first to polar coordinates, and later $s = r/R$) and estimates for easy integrals.\footnote{Concretely, we have used $\int_{-\infty}^{+\infty} e^{-t^2} \, dt = \sqrt{\pi}$, $\int_a^b s^{n-1} e^{-s^2} \, ds \leq \int_0^b s^{n-1} \, ds \approx b^n $ (if $0 < a < b$), $\int_0^{\infty} s^{n-1} e^{-s^2}\, ds \lesssim \int_0^{\infty} e^{-s^2/2}\, ds \approx 1$ (because $s \mapsto s^{n-1}e^{-s^2/2}$ is bounded in $(0, +\infty)$), and $\int_a^\infty e^{-\rho^2} \, d\rho \leq e^{-a^2} \int_0^\infty e^{-\theta^2} \, d\theta \approx e^{-ca^2}$ (by a translation, if $a > 0$).}

	To finish the proof, let us prove the claim \eqref{eq:claim_exponential_cube}. Upon rotating, we may assume $X = \abs{X}e_n$. Then, let $v$ be the solution to (keep Figure~\ref{fig:complement_cube} in mind along the proof)
	\[
	\begin{cases}
		\partial_t v - \Delta v = 0 \qquad \qquad & \text{ in } \T_{> -2}, \\ 
		v(X_1, \ldots, X_n, t) = 3 \qquad & \text{ on } \mathcal{T}_{=-2} \cap \{-M < X_n < 1\}, \\ 
		v(X_1, \ldots, X_n, t) = 0 \qquad & \text{ on } \mathcal{T}_{=-2} \cap \{X_n \geq 1 \text { or } X_n \leq -M\}, \\ 
	\end{cases}
	\]
	where $M$ is a large constant to be determined soon. By Lemma~\ref{lem:fund_sol}, 
	\[v(Z, \tau) = 3 \int_{\R^{n-1} \times (-M, 1)} \Gamma(Z, \tau; Y, -2) \, dY, 
		\qquad (Z, \tau) \in \T_{>-2},
	\] 
	so it is easy to check that, if $M$ is large enough, $v \geq 1$ in $\partial \Q_{1}(0, -1)$. Indeed, by symmetry, $v(Z, \tau) \geq v(0, \ldots, 0, 1, 0)$ for $(Z, \tau) \in \partial \Q_1(0, -1)$, and 
	\begin{multline*}
	v(0, \ldots, 0, 1, 0) 
	= 3 \int_{\R^{n-1} \times (-M, 1)} \Gamma(0, \ldots, 0, 1, 0; Y, -2) \, dY
	\\ \underset{M \to \infty}{\longrightarrow} 
	3 \int_{\R^{n-1} \times (-\infty, 1)} \Gamma(0, \ldots, 0, 1, 0; Y, -2) \, dY
	= \frac32 
	> 1.
	\end{multline*}
	The last equality follows simply because we are integrating the Gaussian in a half-space, so we can use its symmetry and the fact that it integrates to 1. Hence, we may choose $M$ large enough so that $v(0, \ldots, 0, 1, 0) \geq 1$. 
	
	\begin{figure}
		\centering 
		\scalebox{.7}{
		\begin{tikzpicture}[scale=1]
			\fill[Cyan!40, opacity=0.3] (0.457, 8.308) rectangle (10.617, 0.405);
			
			\filldraw[draw=BrickRed!90, ultra thick, fill=white] (2.715, 8.308) -- (2.715, 0.405);
			
			\node[circle, fill, inner sep=2pt] at (10.053, 4.357) {};
			\node[circle, fill, inner sep=2pt] at (3.844, 7.743) {};
			\filldraw[draw=BrickRed!90, ultra thick, <->, fill=white] (2.433, 8.308) -- (2.433, 4.921);
			\filldraw[draw=BrickRed!90, ultra thick, <->, fill=white] (2.433, 4.921) -- (2.433, 0.405);
			\draw[RoyalBlue, dashed, ultra thick] (3.844, 8.308) -- (3.844, 0.405);
			
			\filldraw[ultra thick, fill=white] (2.715, 4.921) rectangle (3.844, 3.792);
			
			\filldraw[ultra thick, ->, fill=white] (0.457, 0.405) -- (2.15, 0.405);
			\filldraw[ultra thick, ->, fill=white] (0.457, 0.405) -- (0.457, 2.099);
			\node[anchor=center, font=\large] at (1.304, 1.817) {$X_n \in \mathbb{R}$};
			\node[anchor=center, font=\large] at (1.586, 0.7) {$t \in \mathbb{R}$};
			
			\node[anchor=center, font=\large] at (3.279, 4.357) {$u = 1$};
			\node[anchor=center, font=\large, text=BrickRed!90] at (1.868, 6.614) {$v = 0$};
			\node[anchor=center, font=\large, text=BrickRed!90] at (1.868, 2.663) {$v=3$};
			\node[anchor=center, font=\large] at (4.5, 7.743) {$(X, 0)$};
			\node[anchor=center, text=RoyalBlue] at (3.947, 0.13) {$t = 0$};
			\node[anchor=center, text=BrickRed!90] at (2.79, 0.13) {$t = -2
				$};
			\node[anchor=center] at (0.482, 4.325) {$X_n = 0$};
			\node[anchor=center, font=\large] at (10.047, 3.912) {$(0, R\textsuperscript{2})$};
			\node[anchor=center] at (4.006, 3.538) {$\mathbf{Q}_1(0, -1)$};
			\node[anchor=center, rotate=90] at (3, 6.578) {$u = 0$};
			\node[anchor=center, rotate=90] at (3, 1.814) {$u=0$};
		\end{tikzpicture}
		}
		\caption{Sketch of the main points when comparing $u$ and $v$. By symmetry, if $M$ is large enough, $v$ should be close to $3/2$ on $\{X_n=1\}$, and hence $v\geq 1\geq u$ on $\partial \Q_1(0, -1)$. Then, evaluating $v(X, 0)$ is just a matter of using a Gaussian going from time $t=-2$ to $(X, 0)$. Similarly, we evaluate $u(0, R^2)$ using Gaussians going from time $t=0$ to $(0, R^2)$.}
		\label{fig:complement_cube}
	\end{figure}
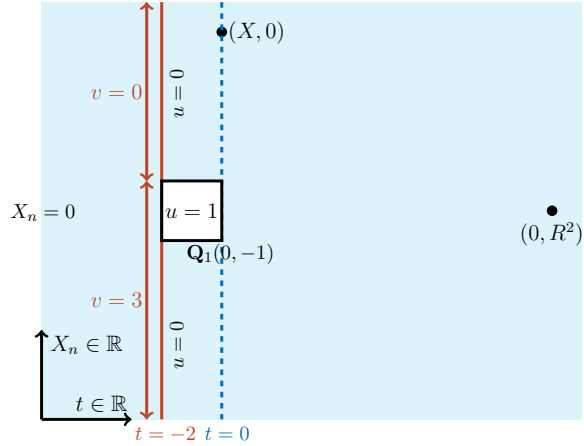
	
	The computations in last paragraph confirm that $v \geq 1 \geq u$ over $\partial \Q_1(0, -1)$, and clearly $v \geq 0 = u$ over $\T_{=-2}$. Therefore, by the maximum principle, $v \geq u$ in $\T_{> -2} \setminus \Q_1(0, -1)$. Using the representation of $v$ by integration against the heat kernel (see Lemma~\ref{lem:fund_sol}):
	\begin{align*}
        u(X, 0)
        \leq
		v(\abs{X}e_n, 0)
		& \leq
		3 \int_{\{Y_n \leq 1\}} \Gamma(\abs{X}e_n, 0; Y, -2) \, dY
		\approx 
		\int_{\{Y_n \leq 1\}} e^{-\frac{\abs{\abs{X}e_n - Y}^2}{8}} \, dY
		\\ & = 
		\int_{-\infty}^1 \int_\R \cdots \int_\R e^{-Y_1^2/8} \cdots e^{-Y_{n-1}^2/8} e^{-(\abs{X} - Y_n)^2/8}\, dY_1 \cdots dY_{n-1} \, dY_n 
		\\ & \approx 
		\int_{-\infty}^1 e^{-(\abs{X} - Y_n)^2/8} \, dY_n 
		=
		\int_{\abs{X}-1}^\infty e^{-\rho^2/8} \, d\rho, 
	\end{align*}
	which concludes the proof of the claim \eqref{eq:claim_exponential_cube}.
\end{proof}

\begin{proof}[\textbf{Proof of Proposition~\ref{prop:probability}}] 
	Let us examine each case separately.
	\begin{itemize}
		\item Case~\eqref{it:bdd} is trivial because $\infty \notin \partial_e \Omega$ (check Remark~\ref{rem:probability}).
		
		\item Case~\eqref{it:unbdd_bdd}. Assume $(X, t) = (0, 0) = \mbf{0}$, and extend $U$ to also contain $\mathcal{T}_{> 0}$, since it does not change the value of $\omega^{\mbf{0}}(\partial_e U)$ (see Remark~\ref{rem:probability}).
		Using that $\partial_e U$ is bounded, take $R \gg 1$ so that $\partial_e U \Subset \Q_{R/2}(\mbf{0})$. By the strong maximum principle, $\omega^{\mbf{0}}_U(\partial_e U) = 1$ if and only if $\omega^{(0, R^4)}_U(\partial_e U) = 1$. Moreover, by the maximum principle (just compare the values at $\partial \Q_R(\mbf{0})$ taking into account that $\partial_e U \subseteq \Q_{R/2}(\mbf{0})$), we have 
		\[
		\omega^{(0, R^4)}_U(\partial_e U)
		\leq 
		\omega^{(0, R^4)}_V(\partial_e V), 
		\qquad 
		\text{where } V := \R^{n+1} \setminus \Q_R(\mbf{0}).
		\]
		(This means that this case reduces to the complement of a cube.) Moreover, the estimate in Lemma~\ref{lem:complement_cube} (appropriately translated and rescaled) yields 
		\[
		\omega_V^{(0, R^4)}(\partial_e V)
		\lesssim 
		R^{-n/2}
		< 
		1
		\] 
		if $R$ is chosen large enough, which finishes the proof. 
		
		\item Case \eqref{it:unbdd_unbdd_Tmin}.
		Pick $\x_0 \in \partial_e U \cap T_{\min}(U)$. Then, for $R$ large enough, it holds $(X, t) \in \Q_R(\x_0)$. By Hölder continuity at the bottom boundary (the estimate \eqref{maineqn:BourgainEstimate} is true by Case 2 in the proof of Theorem~\ref{maineqn:BourgainEstimate}, and this allows to apply Lemma~\ref{lem:iteration_holder} to obtain \eqref{maineqn:HolderEstimate}), we can estimate
		\[
		\omega^{X, t}_U(\{\infty\})
		=
		\lim_{R \to \infty} \omega^{X, t}_U (\partial_e U \setminus \Q_R(\x_0))
		\lesssim 
		\lim_{R \to \infty} \left( \frac{\delta(X, t)}{R} \right)^\alpha
		=
		0, 
		\]
		which gives the result by Remark~\ref{rem:probability}.
		
		\item Case \eqref{it:unbdd_unbdd_TBHCC}.
		The proof follows from the same computation as in Case~\eqref{it:unbdd_unbdd_Tmin}, using that 
		$\omega^{X, t}_U (\partial_e U \setminus \Q_R(\x_0))
		\lesssim 
		\left( \frac{\delta(X, t)}{R} \right)^\alpha$
		if we choose any $\x_0 \in \Sigma$ and $R$ large enough, because we can invoke \eqref{maineqn:HolderEstimate} directly from Theorem~\ref{mainthm:Bourgain}.
		
		\item Case \eqref{it:unbdd_unbdd_counterexample}.
		Let us construct a domain whose exterior does not satisfy the TBCDC at large scales, and hence we will be able to show that the associated caloric measure is not a probability. Concretely, we assume $(X, t) = (0, 0) = \mbf{0}$ and set 
		\[\Omega := \R^{n+1} \setminus \bigcup_{j = 1}^\infty K_j, 
		\qquad 
		\text{where } K_j := \Q_1(0, -R_j^2), \]
		for $R_j$ to be shortly determined. Then, by the maximum principle and Lemma~\ref{lem:complement_cube}, 
		\[
		\omega_U^{\mbf{0}}(\partial_e U)
		=
		\sum_{j=1}^\infty \omega_U^{\mbf{0}}(\partial_e K_j)
		\leq 
		\sum_{j=1}^\infty \omega_{\R^{n+1} \setminus K_j}^{\mbf{0}}(\partial_e K_j)
		\lesssim 
		\sum_{j=1}^\infty R_j^{-1/2}
		< 
		1
		\]
		by choosing $R_j$ growing fast enough, which proves our claim.
	\end{itemize}
\end{proof}

\begin{remark}
    The example that we have constructed for Case~\eqref{it:unbdd_unbdd_counterexample} satisfies the Wiener criterion (see Theorem~\ref{thm:ParabolicWienerCriterion}) because that only concerns small scales, but does not satisfy the TBCDC at large scales. The important part is what happens at large scales: locally around boundary points, our set has \textit{large} exterior, but globally the exterior is fairly small.
\end{remark}

\begin{remark}
    The same kind of counterexample could have been used in \cite[Remark 2.8]{CHMPZ} (i.e. it is not needed that the radii of the balls get smaller). What really matters is having some thickness assumption for large scales. Indeed, we take advantage of it in Cases~\eqref{it:unbdd_unbdd_Tmin} and \eqref{it:unbdd_unbdd_TBHCC}.
\end{remark}


\end{document}